\documentclass[standalone]{article}

\usepackage{PRIMEarxiv}

\usepackage{tikz}
\usetikzlibrary{arrows.meta, positioning, calc}
\usepackage{times}
\usepackage{apxproof}
\usepackage{thmtools,amsfonts} 
\usepackage{amssymb,amsmath,mathrsfs}% 

\usepackage[colorlinks=true,allcolors=blue]{hyperref}       % hyperlinks
\usepackage[nameinlink,noabbrev,capitalize]{cleveref}
\makeatletter

\newcommand{\E}{\mathbb{E}}

\newcommand{\R}{\mathbb{R}}

\newcommand{\I}{\mathbb{I}}
\usepackage{enumitem}

\usepackage{bm}
\usepackage{natbib}

\newtheorem{thm}{Theorem}[section]
\newtheorem{lem}{Lemma}[section]
\DeclareMathOperator{\Expect}{\mathbb{E}}
\DeclareMathOperator{\pro}{\mathbb{P}}

\DeclareMathOperator{\Pro}{\mathbb{P}}
\newtheorem{rem}{Remark}

\newtheorem{cor}[thm]{Corollary}

\newtheorem{theorem}{Theorem}

\newtheorem{remark}[theorem]{Remark}

\newtheorem{assumpt}{Assumption}[section]
\crefname{assumpt}{Assumption}{Assumptions}

\usepackage[draft]{changes}
\allowdisplaybreaks

\newtheorem{claims}[remark]{Claim}
\newtheorem{property}[lem]{\text{Property}}
\newtheorem{pros}[thm]{\text{Proposition}}

\usepackage[utf8]{inputenc} % allow utf-8 input
\usepackage[T1]{fontenc}    % use 8-bit T1 fonts

\usepackage{url}            % simple URL typesetting
\usepackage{booktabs}       % professional-quality tables
\usepackage{amsfonts}       % blackboard math symbols
\usepackage{nicefrac}       % compact symbols for 1/2, etc.
\usepackage{microtype}      % microtypography
\usepackage{lipsum}
\usepackage{fancyhdr}       % header
\usepackage{graphicx}  
\graphicspath{{media/}}     % organize your images and other figures under media/ folder
\allowdisplaybreaks
\fancyheadoffset{0pt}
\title{Stability and convergence analysis of AdaGrad for non-convex optimization via novel stopping time-based techniques
}
\author{
  Ruinan Jin $^{1, 3}$ \quad \quad \quad \quad
  Xiaoyu Wang$^{2}$\thanks{The corresponding author is Xiaoyu Wang <maxywang@ust.hk>.} \quad \quad \quad \quad
  Baoxiang Wang$^{1, 3}$ \\
  $^1$The Chinese University of Hong Kong, Shenzhen, China \\
 $^2$The Hong Kong University of Science and Technology, Hong Kong, China \\
  $^3$Vector Institute, Toronto, Canada\\
\texttt{jinruinan@cuhk.edu.cn}  \quad \texttt{maxywang@ust.hk} \\
\quad  \texttt{bxiangwang@cuhk.edu.cn}
}

\begin{document}

% \title{Asymptotic and Non-Asymptotic Convergence Analysis of AdaGrad for Non-Convex Optimization via Novel Stopping Time-based Analysis}

% \author{\name Author One \email one@stat.washington.edu \\
%        \addr Department of Statistics\\
%        University of Washington\\
%        Seattle, WA 98195-4322, USA
%        \AND
%        \name Author Two \email two@cs.berkeley.edu \\
%        \addr Division of Computer Science\\
%        University of California\\
%        Berkeley, CA 94720-1776, USA}

% \editor{My editor}

\maketitle

\begin{abstract}
Adaptive gradient optimizers (AdaGrad), which dynamically adjust the learning rate based on iterative gradients, have emerged as powerful tools in deep learning. These adaptive methods have significantly succeeded in various deep learning tasks, outperforming stochastic gradient descent. However, despite AdaGrad's status as a cornerstone of adaptive optimization, its theoretical analysis has not adequately addressed key aspects such as asymptotic convergence and non-asymptotic convergence rates in non-convex optimization scenarios. This study aims to provide a comprehensive analysis of AdaGrad and bridge the existing gaps in the literature. We introduce a new stopping time technique from probability theory, which allows us to establish the stability of AdaGrad under mild conditions. We further derive the asymptotically almost sure and mean-square convergence for AdaGrad. In addition, we demonstrate the near-optimal non-asymptotic convergence rate measured by the average-squared gradients in expectation, which is stronger than the existing high-probability results. The techniques developed in this work are potentially of independent interest for future research on other adaptive stochastic algorithms.

\end{abstract}
% \begin{keywords}
% Adaptive gradient methods, Asymptotic convergence, Non-asymptotic convergence
% \end{keywords}

\section{Introduction}
 
Adaptive gradient methods have achieved remarkable success across various machine learning domains. State-of-the-art adaptive methods like AdaGrad~\citep{duchi2011adaptive}, RMSProp~\citep{RMSProp}, Adam~\citep{Adam},  which automatically adjust the learning rate based on past stochastic gradients, often outperform 
vanilla stochastic gradient descent (SGD) on non-convex optimization~\citep{vaswani2017attention,2013Estimation,2018Canonical,DBLP:conf/iclr/DosovitskiyB0WZ21}. AdaGrad~\citep{duchi2011adaptive,mcmahan2010adaptive} is the first prominent algorithm in this category. This paper investigates the norm version of AdaGrad (known as AdaGrad-Norm), which is a single stepsize adaptation method and is formally described as
\begin{equation}\label{AdaGrad_Norm}\begin{aligned}
S_{n}=S_{n-1}+\big\|\nabla g(\theta_{n},\xi_{n})\big\|^{2},
	\quad \theta_{n+1}=\theta_{n}-\frac{\alpha_0}{\sqrt{S_{n}}}\nabla g(\theta_{n},\xi_{n}),
\end{aligned}\end{equation} 
where $S_{0}$ and  $\alpha_{0}$ are pre-determined positive constants, and the stochastic gradient $\nabla g(\theta_{n},\xi_{n})$ is an estimation of the true gradient $\nabla g(\theta_n)$ with the noise variable $\xi_n$. In recent years, the simplicity and popularity of AdaGrad-Norm have attracted many research studies~\citep{zou2018weighted,ward2020adagrad,defossez2020simple,kavis2022high,faw2022power,wang2023convergence,jin2022convergence}. However, the correlation of step size $\alpha_n = \alpha_0/\sqrt{S_n}$ with current stochastic gradient and all past stochastic gradients poses significant challenges for the theoretical analysis of AdaGrad-Norm,  in both asymptotic
and non-asymptotic contexts. This study aims to address these limitations and provide a comprehensive understanding of the asymptotic and non-asymptotic convergence behaviors of AdaGrad in smooth non-convex optimization.

\subsection{Key Challenges and Contribution}\label{subsection 1.1}

\paragraph{\textbf{Challenges in asymptotic convergence}} Our work focuses on two fundamental criteria: almost sure convergence and mean-square convergence. Almost sure convergence, defined as $\lim_{n \rightarrow \infty} \left\|\nabla g(\theta_n) \right\|=0 \ \  a.s.$, provides a robust guarantee that the algorithm will converge to the critical point with probability 1 during a single run of the stochastic method. In practical scenarios, algorithms are typically executed only once, with the last iterate returned as the output. The asymptotic almost sure convergence of SGD and its momentum variants generally relies on the Robbins-Monro (RM) conditions for the step size $\alpha_n$, i.e. $\sum_{n=1}^{+\infty}\alpha_n=+\infty,\ \sum_{n=1}^{+\infty}\alpha_n^{2}<+\infty$~\citep{robbins1971convergence,li2022unified}. 
Under the $ L$-smoothness assumption, the classic descent lemma for SGD is
\begin{align}\label{descent:sgd}
\E[g(\theta_{n+1}) \mid \mathscr{F}_{n-1}] - g(\theta_{n})\le -\alpha_{n}\|\nabla g(\theta_{n})\|^{2}+\frac{ L\alpha_{n}^{2}}{2}\E\left[\|\nabla g(\theta_{n},\xi_{n})\|^{2} \mid \mathscr{F}_{n-1}\right].
\end{align}
The RM conditions are essential to ensure the summability of the quadratic error in \eqref{descent:sgd}. However, the situation is different for the original AdaGrad-Norm as the quadratic error does not exhibit such summability because $S_n$ could go to infinity \hspace{-1in}
\begin{equation*}\sum_{n=1}^{+\infty} \alpha_n^2\|\nabla g(\theta_{n},\xi_{n})\|^{2} = \sum_{n=1}^{+\infty} S_n^{-1}\|\nabla g(\theta_{n},\xi_{n})\|^{2}= \lim_{n\rightarrow \infty}O(\ln S_{n}) .
\end{equation*}
% However, the situation is different for AdaGrad-Norm as \emph{it violates typical Robbins-Monro conditions}\hspace{-1in}
% \begin{equation*}\sum_{n=1}^{+\infty} \alpha_n^2\|\nabla g(\theta_{n},\xi_{n})\|^{2} = \sum_{n=1}^{+\infty} S_n^{-1}\|\nabla g(\theta_{n},\xi_{n})\|^{2}= \lim_{n\rightarrow \infty}O(\ln S_{n}) = +\infty.
% \end{equation*}
Moreover, the step size of AdaGrad-Norm is influenced by both the current stochastic gradient and all past stochastic gradients, making the derivation of its almost sure convergence particularly challenging. %Together, deriving the almost sure convergence of AdaGrad-Norm poses significant challenges. %Moreover, since we are dealing with a potentially unbounded loss function, almost sure convergence does not imply mean-square convergence and not the other way as stated in probability theory. 

In addition to almost sure convergence, mean-square convergence (MSE) is another critical criterion, formulated by $\lim_{n \rightarrow \infty} \E\left\|\nabla g(\theta_{n}) \right\|^2 = 0$. This criterion assesses the asymptotic averaged behavior of stochastic optimization methods over infinitely many runs. %This criterion provides profound insights into the averaged behavior of stochastic methods over infinitely many runs. 
Importantly, as in probability theory, mean-square convergence does not imply almost-sure convergence, nor vice versa. The mean-square convergence has been extensively discussed in the literature \citep{li2022unified,bottou2018optimization} for SGD in non-convex settings. Nevertheless, the mean-square convergence of adaptive methods has not been explored, making it a significant and non-trivial study area. %Additionally, since the learning rate is also a random variable, we cannot obtain mean-square convergence by trivially taking the expectation on both sides of the sufficient decrease lemma.

\paragraph{\textbf{Contribution in asymptotic convergence} } To achieve asymptotic convergence, our first major contribution is demonstrating the stability of the loss function in expectation under mild conditions. We utilize a novel stopping-time partitioning technique to accomplish this.
\begin{lem}\label{stable'} (Informal)
Consider AdaGrad-Norm under appropriate conditions, there exists a constant $\tilde{M} > 0$ such that
\begin{equation}\nonumber
\Expect\Big[\sup_{n \ge 1}g(\theta_{n})\Big] < \tilde{M} < +\infty.
\end{equation}
\end{lem}
To establish the asymptotic convergence for gradient-based methods, it is important to ensure the global stability of the trajectories. Many existing studies on SGD~\citep{ljung1977analysis,benaim2006dynamics,bolte2021conservative} and adaptive methods \citep{barakat2021convergence,JMLR:v25:23-0576} explicitly assumed bounded trajectories, i.e. \(\sup_{n\ge 1}\|\theta_{n}\|<+\infty\) almost surely. However, this assumption is quite stringent, as trajectory stability can only be verified if the algorithm runs through all iterations, which is practically infeasible.  Recent works by  \cite{josz2023lyapunov,xiao2023convergence} have established the stability of SGD under the coercivity condition. In contrast, our result in \cref{stable'} indicates that the trajectories are bounded for AdaGrad-Norm, i.e., \(\sup_{n\ge 1}\|\theta_{n}\|<+\infty\ \text{a.s.}\) given coercivity. To the best of our knowledge, this represents the first demonstration of the stability of an adaptive method, marking a significant advancement in the understanding of adaptive gradient techniques. 

%Directly assuming trajectories' stability is strong, which can only be verified if the algorithm completes all iterations (something practically impossible).

With the stability result established, we adopt a divide-and-conquer approach based on the gradient norm to demonstrate the asymptotic almost-sure convergence for AdaGrad-Norm. Notably, our analysis does not rely on the assumption of the absence of saddle points, which makes an important improvement over the findings of  \cite{jin2022convergence}. Furthermore, we establish the novel mean-square convergence result for AdaGrad-Norm, leveraging the stability discussed in \cref{stable'} alongside the almost sure convergence.  

In addition, we extend the proof techniques developed for AdaGrad to investigate the asymptotic convergence of another adaptive method, RMSProp~\citep{RMSProp}, under a specific choice of hyperparameters. This investigation yields insights into the stability and asymptotic convergence behavior of RMSProp and deepens our understanding of its performance in various optimization scenarios. This also showcases how the techniques developed in this work could be applied to other problems.

\paragraph{\textbf{Challenges in non-asymptotic result}} Our next objective is to explore the non-asymptotic convergence rate, which captures the overall trend of the method during the first $T$ iterations. 
The convergence rate, measured by the expected average-squared gradients, $\frac{1}{T}\sum_{k=1}^{T}\E[\left\|\nabla g(\theta_k) \right\|^2]$, is commonly used as metric in SGD~\citep{Ghadimi-2013,bottou2018optimization}. However, such analyses are rare for adaptive methods that do not assume bounded stochastic gradients. Therefore, our study aims to bridge this gap by providing convergence for AdaGrad-Norm in the expectation sense, without the restrictive assumption of uniform boundedness of stochastic gradients.

\paragraph{\textbf{Contribution in non-asymptotic expected rate}} To address the non-asymptotic convergence rate, we first estimate the expected value of $S_T$ under relaxed conditions, which specifically focuses on the smoothness and affine noise variance conditions (i.e., $\E[\big\|\nabla g(\theta_n,\xi_{n})\big\|^{2} \mid \mathscr{F}_{n-1}]\le \sigma_{0}\big\|\nabla g(\theta_n)\big\|^{2}+\sigma_{1}$, see \cref{ass_noise}~\ref{ass_noise:i2}). %, we achieve the following estimation for $\E[S_T]$:
\begin{lem}\label{lem:st:informal}(Informal) Consider AdaGrad-Norm under appropriate conditions
\begin{align*}
\E(S_{T})=O(T).
\end{align*}
\end{lem}
Our estimation of $S_T$ in \cref{lem:st:informal} is more precise than that of \cite{wang2023convergence} which only established $\Expect[\sqrt{S_{T}}] = \mathcal{O}(\sqrt{T})$. This refined estimation allows us to achieve a near-optimal (up to $\log$ factor) convergence  $\frac{1}{T}\sum_{k=1}^{T}\E[\left\|\nabla g(\theta_k) \right\|^2] \leq \mathcal{O}(\ln T /\sqrt{T})$. %$\frac{1}{T}\sum_{n=1}^{T}\Expect\|\nabla g(\theta_{n})\|^{2}$. 
To the best of our knowledge, this is the first convergence rate measured by expected average-squared gradients for adaptive methods without uniform boundedness gradient assumption. This result is stronger than the high probability results presented in \cite{faw2022power,wang2023convergence}. Furthermore, we improve the dependence on $1/\delta$ from quadratic to linear in the high-probability $1-\delta$ convergence rate, surpassing the results in \cite{faw2022power,wang2023convergence}.

\subsection{Related Work}

\paragraph{\textbf{Asymptotic convergence of AdaGrad and its variants}} The authors in \cite{jin2022convergence} demonstrated the asymptotic almost sure convergence of AdaGrad-Norm for nonconvex functions. However, their analysis relied on the unrealistic assumption that the loss function contains no saddle points (as noted in item 1 of Assumption 5 of \cite{jin2022convergence})). Since saddle points are common in non-convex scenarios, this significantly limits the practical applicability of their convergence results.
The authors of \cite{li2019convergence} established the almost-sure (the inferior limit) convergence for an AdaGrad variant under the global boundedness of gradient when the loss function is non-convex. The variant in~\cite{li2019convergence} is modified from the original AdaGrad algorithm by replacing the current stochastic gradient with a past one in step size (delayed AdaGrad) and incorporating the higher order of $S_n$ in the adaptive learning rate.  %, unlike the specific challenges investigated in this paper. 
Note that our focus remains to be on the original AdaGrad without any modifications. 

The study of \cite{gadat2022asymptotic} examined the asymptotic almost sure behavior of a subclass of adaptive gradient methods. However, their analysis involved modifications to the algorithm. For instance, for AdaGrad, they make the step size $\alpha_n$ (conditionally) independent of the current stochastic gradient and enforce that the step size satisfies the Robbins-Monro conditions by decreasing $\alpha_0$ and increasing the mini-batch size.  In \cite{barakat2021convergence}, they obtained the almost sure convergence towards critical points for Adam, under the stability assumption to ensure that the iterates do not explode in the long run.

\paragraph{\textbf{Non-asymptotic convergence of AdaGrad}} The study by \cite{duchi2011adaptive} proved the efficiency of AdaGrad for sparse gradients in convex optimization problems. In \cite{levy2017online}, rigorous convergence results for AdaGrad-Norm were provided specifically for convex minimization problems. However, establishing results for non-convex functions presents significant challenges, particularly due to the dependence of $S_n$ with current and all past stochastic gradients. In the context of non-convex optimization, a line of research \citep{zou2018weighted,zhou2018convergence,chen2019convergence,ward2020adagrad,defossez2020simple,kavis2022high} has explored the non-asymptotic convergence results for AdaGrad and its close variants. For instance, Li and Orabona \citep{li2019convergence} examined the convergence of delayed AdaGrad-Norm for non-convex objectives under a hard threshold $\alpha_0 < \sqrt{S_0}/L$ and sub-Gaussian noise.  Zou et al.~\citep{zou2018weighted} established the convergence for coordinate-wise AdaGrad with either heavy-ball or Nesterov momentum. %In \cite{zhou2018convergence}, they proved convergence rate of $\mathcal{O}(\sqrt{d/T} + d/T)$ in expectation or high probability for AdaGrad. 
In \cite{ward2020adagrad}, a convergence rate of $\mathcal{O}(\ln T/\sqrt{T})$ was established in high probability for AdaGrad-Norm under conditions of globally bounded gradients. However, these studies typically require that stochastic gradients are uniformly upper bounded~\citep{zou2018weighted,zhou2018convergence,chen2019convergence,ward2020adagrad,defossez2020simple,kavis2022high}. The assumption is often violated in the presence of Gaussian random noise in stochastic gradients and does not hold even for quadratic loss~\citep{wang2023convergence}. Recent works by~\cite{faw2022power,wang2023convergence} have addressed this limitation by removing the assumption of uniform boundedness of stochastic gradients through the use of affine noise variance. Despite this advancement, %under $ L$-smooth and affine noise variance conditions,  \cite{ward2020adagrad,defossez2,faw2022power,wang2023convergence}
the convergence rates for the original AdaGrad-Norm, as described in~\cite{faw2022power,wang2023convergence}, are derived only in the context of \emph{high probability}. %In this paper, we focus on the non-asymptotic convergence rate measured by expected average-squared gradients. 

\subsection{Organization and Notation}
\paragraph{\noindent{\bf Organization}}The rest of this paper is organized as follows. \cref{sec:prelim} formalizes
the general problem statement and the basic assumptions required in the analysis. %Specifically, the non-asymptotic results rely on the fundamental assumptions of affine noise variance and \( L\)-smoothness. For the asymptotic results, some additional assumptions are necessary. 
In \cref{sec:asympt:result}, we present the two asymptotic convergence results for AdaGrad-Norm. Specifically, In \cref{subsec:stability}, we establish the stability properties of AdaGrad-Norm. \cref{subsec:almost:sure} is dedicated to proving the asymptotic almost sure convergence of AdaGrad-Norm, while \cref{sec:mean:convergence} addresses its asymptotic mean-square convergence. In \cref{sec:nonasympt}, we establish the non-asymptotic convergence results for AdaGrad-Norm under affine noise variance and \(L\)-smoothness. In \cref{sec:AdaGrad:coordinate}, we extend our asymptotic results to the RMSProp algorithm with near-optimal hyperparameter configurations. \cref{sec:conclusion} concludes the paper. 

\paragraph{\noindent{\bf Notation}} We use $\I_{X}(x) = 1 $ if $x \in X$ and $\I_{X}(x) = 0$ otherwise to denote the indicator function. Given an objective function $g(\theta)$. We define the critical points set $\Theta^{\ast}:=\{\theta \mid \nabla g(\theta)=0\}$ and the critical value set $g(\Theta^{\ast}):=\{g(\theta) \mid \nabla g(\theta)=0\}$. We use $\E[\cdot]$ denote the expectation on the probability space and $\E[\cdot \mid \mathscr{F}_{n}]$ denote the conditional expectation on the $\sigma$-field $\mathscr{F}_{n}$. For notational convenience, $\E[X^2]$ denotes the expectation on the square of the random variable $X$ and $\E^2[X]$ represents the square of the expectation on the random variable $X$. To make the notation $\sum_{a}^{b}(\cdot)$ consistent, we let $\sum_{a}^{b}(\cdot) \equiv -\sum_{b}^{a}(\cdot)\ (\forall\ b< a).$ The notation $[d]$ denotes the set of the integers $\left\lbrace 1,2,\cdots, d\right\rbrace$.

\section{Problem Setup and Preliminaries}\label{sec:prelim}
Throughout the sequel, we consider the unconstrained non-convex optimization problem
\begin{align}\label{P1}
    \min_{\theta \in \R^d} \,\, g(\theta),
\end{align}
where $g: \R^d \rightarrow \R$ satisfies the following assumptions.%(\textcolor{red}{Assumption on the function and stochastic oracle})% %$g: \R^d \rightarrow \R$ is continuously differentiable,  and satisfies the following assumptions. 
\begin{assumpt}\label{ass_g_poi} The objective function $g(\theta)$ satisfies the following conditions:
\begin{enumerate}[label=\textnormal{(\roman*)},leftmargin=*]
\item\label{ass_g_poi:i} $g(\theta)$ is continuously differentiable and non-negative.
\item\label{ass_g_poi:i2} $\nabla g(\theta)$ is Lipschitz continuous, i.e., $
\big\|\nabla g(\theta)-\nabla g(\theta')\big\|\le L\|\theta-\theta'\| $, for all $\theta, \theta' \in \mathbb{R}^{d}$.
\item\label{ass_g_poi:i3}(\textbf{Only for asymptotic convergence}) $g(\theta)$ is not asymptotically flat, i.e., there exists $\eta>0$ such that $\liminf_{\|\theta\|\rightarrow+\infty} \|\nabla g(\theta)\|^{2}>\eta.$

\end{enumerate}
\end{assumpt}

The conditions $\ref{ass_g_poi:i}\sim\ref{ass_g_poi:i2}$ of  \cref{ass_g_poi}  are standard in most literature on non-convex optimization~\citep{bottou2018optimization}. Note that the non-negativity of $g$ in \cref{ass_g_poi:i} is equivalent to stating that $g$ is bounded from below. %~\cite{defossez2020simple,faw2022power,wang2023convergence}
\cref{ass_g_poi:i3} has been utilized by \cite{mertikopoulos2020almost} to analyze the almost sure convergence of SGD under the step-size that may violate Robbins-Monro conditions. The purpose is to exclude functions such as $g(x)=-e^{-x^{2}}$ or $g(x)=\ln x$, which exhibit near-critical behavior at infinity.  Non-asymptotically flat objectives are common in machine learning, especially with $L_2$ or $L_1$ regularization~\citep{ng2004feature,bishop2006pattern,zhang2004solving,goodfellow2016deep}. Additionally,  \cref{ass_g_poi:i3} is specifically employed for asymptotic convergence and is \textbf{NOT} required for the non-asymptotic convergence rates.

Typical examples of Problem~\eqref{P1} include modern machine
learning, deep learning, and underdetermined inverse problems. In these contexts, obtaining precise gradient information is often impractical. This paper focuses on the stochastic methods through a stochastic first-order oracle (SFO) which takes an input $\theta_n \in \R^d$ and returns a random vector $\nabla g(\theta_n, \xi_n)$ drawn from the probability space $(\Omega, \left\lbrace\mathscr{F}_n \right\rbrace_{n\ge 1}, \mathbb{P})$. The noise sequence $\{\xi_{n}\}$ consists of independent random variables. We denote the $\sigma$-filtration $\mathscr{F}_{n}:=\sigma\{\theta_{1},\xi_{1},\xi_{2},...,\xi_{n}\}$ for $n\geq 1$,  with $\mathscr{F}_{i}:=\{\emptyset,\ \Omega\}$ for $i=0$, and  define $\mathscr{F}_{\infty}:=\bigcup_{n=1}^{+\infty}\mathscr{F}_{n}$. Thus, $\theta_n$ is $\mathscr{F}_n$ measurable for all $ n \geq 0$. 

We make the following assumptions regarding the stochastic gradient oracle.
\begin{assumpt}\label{ass_noise} The stochastic gradient $\nabla g(\theta_n, \xi_n)$ satisfies 
\begin{enumerate}[label=\textnormal{(\roman*)},leftmargin=*]
\item\label{ass_noise:i}  $\E \left[\nabla g(\theta_n, \xi_n) \mid \mathscr{F}_{n-1}\right] = \nabla g(\theta_n)$.
 \item\label{ass_noise:i2} (\textbf{Affine noise variance})   $\E\left[\big\|\nabla g(\theta_n,\xi_{n})\big\|^{2} \mid \mathscr{F}_{n-1}\right]\le \sigma_{0}\big\|\nabla g(\theta_n)\big\|^{2}+\sigma_{1}$, for some constants $\sigma_{0}, \sigma_{1} \geq 0$.
 \item\label{ass_noise:i3}(\textbf{Only for asymptotic convergence}) For any $\theta_n$ satisfying $\|\nabla g(\theta_{n})\|^{2}<D_0$, it holds that $\|\nabla g(\theta_{n},\xi_{n})\|^{2}< D_1$ a.s.. for some constants $D_0, D_1 >0$. 
\end{enumerate}
\end{assumpt}
\cref{ass_noise}~\ref{ass_noise:i} is standard in the theory of SGD and its variants. %~\cite{faw2022power,wang2023convergence,jin2022convergence}. 
\cref{ass_noise}~\ref{ass_noise:i2} is milder than the typical bounded variance assumption~\citep{li2019convergence} and bounded gradient assumption~\citep{mertikopoulos2020almost,kavis2022high}. \cite{gadat2022asymptotic} requires that the variance of the stochastic gradient asymptotically converge to $0,$ i.e., $\lim_{n\rightarrow+\infty}\Expect_{\xi_{n}}\|\nabla g(\theta_{n},\xi_{n})-\nabla g(\theta_{n})\|^{2}=0,$ which is not satisfied in common settings with a fixed mini-batch size.  We note that \cref{ass_noise}~\ref{ass_noise:i3} only restricts the sharpness of stochastic gradient near the critical points. It is possible to allow $D_0$ to be arbitrarily small (approaching zero) while allowing $D_1$ to be sufficiently large. \cref{ass_noise}~\ref{ass_noise:i3} is only used to demonstrate the asymptotic convergence, which is \textbf{NOT} necessary for the non-asymptotic convergence rate.

\begin{rem}
Under~\cref{ass_g_poi}, the widely used mini-batch stochastic gradient model  satisfies~\cref{ass_noise:i3} of \cref{ass_noise}. Since the near-critical case at infinity is excluded (\cref{ass_g_poi}~\ref{ass_g_poi:i3}), we can identify a sufficiently small $D_0$ such that the near-critical points set $\{\theta \mid \|\nabla g(\theta)\|<D_0\}$ remains bounded. Consequently, when the stochastic gradient is Lipschitz continuous, the mini-batch stochastic gradients will remain within a bounded set, thereby satisfying \cref{ass_noise:i3}.
\end{rem}

\section{Asymptotic Convergence of AdaGrad-Norm}
\label{sec:asympt:result}
This section will establish the two types of asymptotic convergence guarantees including almost sure convergence and mean-square convergence for AdaGrad-Norm in the smooth non-convex setting under~\cref{ass_g_poi,ass_noise}.

By $ L$-smooth property, we have the so-called descent inequality for AdaGrad-Norm
\begin{align}\label{inequ:smooth:inequality}
g(\theta_{n+1})-g(\theta_{n}) & \le -\frac{\alpha_{0}\nabla g(\theta_{n})^{\top}\nabla g(\theta_{n},\xi_{n})}{\sqrt{S_{n}}}+\frac{ L\alpha_{0}^{2}}{2}\cdot\frac{\|\nabla g(\theta_{n},\xi_{n})\|^{2}}{S_{n}}.
\end{align}
We then deal with the correction in AdaGrad-Norm to approximate $S_n$ by the past $S_{n-1}$~\citep{ward2020adagrad,defossez2020simple,faw2022power,wang2023convergence} and the RHS of \cref{inequ:smooth:inequality} can be decomposed as
\begin{align}\label{inequ:sufficient:decrease}
& g(\theta_{n+1})-g(\theta_{n}) \notag\\ 
&\le- \alpha_{0}\E\left(\frac{\nabla g(\theta_{n})^{\top}\nabla g(\theta_{n},\xi_{n})}{\sqrt{S_{n}}} \mid \mathscr{F}_{n-1}\right) +\alpha_{0}\E\left(\frac{\nabla g(\theta_{n})^{\top}\nabla g(\theta_{n},\xi_{n})}{\sqrt{S_{n}}}  \mid \mathscr{F}_{n-1}\right) \notag \\& - \alpha_0\frac{\nabla g(\theta_{n})^{\top}\nabla g(\theta_{n},\xi_{n})}{\sqrt{S_{n}}}  + \frac{ L\alpha_{0}^{2}}{2}\cdot\frac{\|\nabla g(\theta_{n},\xi_{n})\|^{2}}{S_{n}}\notag\\& =  -\alpha_{0}\frac{\|\nabla g(\theta_{n})\|^{2}}{\sqrt{S_{n-1}}} +\alpha_{0}\E\left(\nabla g(\theta_{n})^{\top}\nabla g(\theta_{n},\xi_{n})\left(\frac{1}{\sqrt{S_{n-1}}} - \frac{1}{\sqrt{S_n}}\right) \mid \mathscr{F}_{n-1}\right) \notag\\&+\alpha_{0}\left(\Expect\bigg[\frac{\nabla g(\theta_{n})^{\top}\nabla g(\theta_{n},\xi_{n})}{\sqrt{S_{n}}}\bigg|\mathscr{F}_{n-1}\bigg]-\frac{\nabla g(\theta_{n})^{\top}\nabla g(\theta_{n},\xi_{n})}{\sqrt{S_{n}}}\right)  + \frac{ L\alpha_{0}^{2}}{2}\cdot \frac{\|\nabla g(\theta_{n},\xi_{n})\|^{2}}{S_{n}} \notag \\
& \mathop{\leq}^{(a)} -\alpha_{0}\overbrace{\frac{\|\nabla g(\theta_{n})\|^{2}}{\sqrt{S_{n-1}}}}^{\zeta(n)}+\alpha_{0}\Expect\Bigg[\overbrace{\frac{\|\nabla g(\theta_{n})\|\cdot \|\nabla g(\theta_{n},\xi_{n})\|}{\sqrt{S_{n-1}}}}^{R_n}\cdot \overbrace{\frac{\|\nabla g(\theta_{n},\xi_{n})\|^{2}}{\sqrt{S_{n}}(\sqrt{S_{n-1}}+\sqrt{S_{n}})}}^{\Lambda_n}\Bigg|\mathscr{F}_{n-1}\Bigg]\notag \\
&
+\alpha_{0}\underbrace{\left(\Expect\bigg[\frac{\nabla g(\theta_{n})^{\top}\nabla g(\theta_{n},\xi_{n})}{\sqrt{S_{n}}}\bigg|\mathscr{F}_{n-1}\bigg]-\frac{\nabla g(\theta_{n})^{\top}\nabla g(\theta_{n},\xi_{n})}{\sqrt{S_{n}}}\right)}_{X_n}  + \frac{ L\alpha_{0}^{2}}{2}\cdot \underbrace{\frac{\|\nabla g(\theta_{n},\xi_{n})\|^{2}}{S_{n}}}_{\Gamma_n},
\end{align} 
where for $(a)$ we use the Cauchy-Schwartz inequality, and 
\begin{align}\label{inequ:sn:minus:sn1}
\frac{1}{\sqrt{S_{n-1}}} - \frac{1}{\sqrt{S_{n}}} =  \frac{\|\nabla g(\theta_{n},\xi_{n})\|^{2}}{\sqrt{S_{n-1}}\sqrt{S_{n}}\cdot(\sqrt{S_{n-1}}+\sqrt{S_{n}})}.
\end{align}
In this decomposition, we define the martingale sequence $X_{n}$ and introduce the notations $\zeta(n), R_n, \Lambda_n, \Gamma_n$ to simplify the expression given in \cref{inequ:sufficient:decrease}. Furthermore, we introduce $\hat{g}(\theta_n)$ as the Lyapunov function and $\{ \hat{X}_n, \mathscr{F}_{n}\}_{n\ge 1}$ is a new martingale difference sequence (MDS) to achieve the key sufficient decrease inequality as follows. 
\begin{lem}({\bf Sufficient decrease inequality})\label{sufficient:lem}
Under \cref{ass_g_poi}~\ref{ass_g_poi:i}$\sim$\ref{ass_g_poi:i2} and  \cref{ass_noise}~\ref{ass_noise:i}$\sim$ \ref{ass_noise:i2}, consider the sequence $\{\theta_{n}\}$ generated by AdaGrad-Norm, we have 
\begin{align}
\hat{g}(\theta_{n+1}) - \hat{g}(\theta_n) \leq -\frac{\alpha_{0}}{4}\zeta(n)+C_{\Gamma, 1}\cdot \Gamma_n  + C_{\Gamma, 2}\frac{\Gamma_n}{\sqrt{S_{n}}} + \alpha_0 \hat{X}_{n},
\end{align}
where $\hat{g}(\theta_{n}):=g(\theta_{n})+ \frac{\sigma_{0}\alpha_{0}}{2}\zeta(n)$, $\hat{X}_n = X_n + V_n$, and $V_n$ is defined in \cref{notation:Vn}. The constant terms $C_{\Gamma, 1}, C_{\Gamma, 2}$ are defined in \cref{notation:ghat:const}.
\end{lem}
\begin{proof}(of~\cref{sufficient:lem})
We first recall \cref{inequ:sufficient:decrease}
\begin{align}\label{adagrad:sufficient:inequ:1}
g(\theta_{n+1})-g(\theta_{n}) & \le - \alpha_{0}\zeta(n)+\alpha_{0}\Expect\left[R_n \Lambda_n \mid \mathscr{F}_{n-1}\right]+\frac{ L\alpha_{0}^{2}}{2}\Gamma_n +\alpha_0 X_{n} .
\end{align} 
Next, we focus on dealing with the second term on the RHS of \cref{adagrad:sufficient:inequ:1} and achieve:
\begin{align}\label{adagrad:sufficient:inequ:2}
& \Expect\left[R_n \Lambda_n \mid \mathscr{F}_{n-1}\right] := \frac{\|\nabla g(\theta_{n})\|}{\sqrt{S_{n-1}}}\cdot\Expect\left[\|\nabla g(\theta_{n},\xi_{n})\|\Lambda_n \mid \mathscr{F}_{n-1}\right]  \notag \\ \mathop{\le}^{(a)} &\frac{\|\nabla g(\theta_{n})\|^{2}}{2\sqrt{S_{n-1}}}+\frac{1}{2\sqrt{S_{n-1}}}\Expect^{2}\left[\|\nabla g(\theta_{n},\xi_{n})\| \Lambda_n \mid \mathscr{F}_{n-1}\right] \notag \\ \mathop{\le}^{(b)} &\frac{\zeta(n)}{2}+\frac{\Expect[\|\nabla g(\theta_{n},\xi_{n})\|^{2}|\mathscr{F}_{n-1}]}{2\sqrt{S_{n-1}}}\cdot\Expect\left[\Lambda_n^2 \mid \mathscr{F}_{n-1}\right] \notag \\ \mathop{\le}^{(c)}
&\frac{\zeta(n)}{2}+\frac{\sigma_{1}\Expect\left[\Lambda_n^2 \mid \mathscr{F}_{n-1}\right]}{2\sqrt{S_{n-1}}} +\frac{\sigma_{0}}{2}\cdot\frac{\|\nabla g(\theta_{n})\|^{2}}{\sqrt{S_{n-1}}}\cdot \Expect\left[\Lambda_n^2 \mid \mathscr{F}_{n-1}\right] \notag \\
 \mathop{\le}^{(d)}
&\frac{\zeta(n)}{2}+\frac{\sigma_{1}}{2\sqrt{S_{0}}}\Gamma_n^2 +\frac{\sigma_{0}}{2}\cdot \zeta(n) \cdot \Lambda_n^2 +V_{n},
\end{align}
where for $(a), (b)$ we use \emph{Cauchy-Schwartz inequality}, $(c)$ is by applying the affine noise variance condition, and $(d)$ is by applying $\Lambda_n \leq \Gamma_n $  and $S_n \geq S_0$ for $(d)$. In the inequality, the martingale sequence $V_n$ is defined as
\begin{align}\label{notation:Vn}
&V_{n}:=\frac{\sigma_{1}}{2\sqrt{S_{0}}}\Big(\Expect\big[ \Gamma_n^2 \mid \mathscr{F}_{n-1}\big]- \Gamma_n^2 \Big)+\frac{\sigma_{0}}{2}\cdot\left(\Expect\left[\zeta(n)\cdot\Lambda_n^2 \mid \mathscr{F}_{n-1}\right]-\zeta(n) \cdot \Lambda_n^2\right).
\end{align}
We then substitute \cref{adagrad:sufficient:inequ:2} into \cref{adagrad:sufficient:inequ:1} and define $\hat{X}_{n}:=X_{n}+V_{n}$
\begin{align}
g(\theta_{n+1})-g(\theta_{n})  \le & - \frac{\alpha_{0}}{2} \zeta(n)+\frac{\alpha_{0}\sigma_{1}}{2\sqrt{S_{0}}}\cdot \Gamma_n^2 +\frac{\sigma_{0}\alpha_{0}}{2}\cdot \zeta(n)\cdot \Lambda_n^2 +\frac{ L\alpha_{0}^{2}}{2}\cdot  \Gamma_n\notag\\&+ \alpha_0 \hat{X}_{n}.
\end{align}
Recalling the definition of $\Lambda_n$ in ~\cref{inequ:sufficient:decrease} and applying $\Lambda_n \leq 1$ and \cref{inequ:sn:minus:sn1}, we have 
\begin{align}\label{inequ:cross:grad:lambda}
\zeta(n)\cdot \Lambda_n^2 & \leq   \frac{\|\nabla g(\theta_{n})\|^{2} \cdot \|\nabla g(\theta_{n}, \xi_n)\|^{2}}{\sqrt{S_{n-1}}\sqrt{S_{n}}(\sqrt{S_{n-1}} + \sqrt{S_n})} = \|\nabla g(\theta_{n})\|^{2} \left(\frac{1}{\sqrt{S_{n-1}}}-\frac{1}{\sqrt{S_{n}}} \right) \notag \\
& = \left(\frac{\|\nabla g(\theta_{n})\|^{2}}{\sqrt{S_{n-1}}}-\frac{\|\nabla g(\theta_{n+1})\|^{2}}{\sqrt{S_{n}}}\right) + \frac{\|\nabla g(\theta_{n+1})\|^{2}-\|\nabla g(\theta_{n})\|^{2}}{\sqrt{S_{n}}}. 
\end{align}
By the smoothness of $g$, we estimate the last term of \cref{inequ:cross:grad:lambda}
\begin{align}\label{grad:norm:diff}
&\|\nabla g(\theta_{n+1})\|^{2}-\|\nabla g(\theta_{n})\|^{2} \notag \\
=&\ (2\|\nabla g(\theta_{n})\|+\|\nabla g(\theta_{n+1})\|-\|\nabla g(\theta_{n})\|)\cdot(\|\nabla g(\theta_{n+1})\|-\|\nabla g(\theta_{n})\|) \notag \\ \mathop{\leq}^{(a)} &\frac{2 L\alpha_{0}\|\nabla g(\theta_{n})\|\cdot\|\nabla g(\theta_{n},\xi_{n})\|}{\sqrt{S_{n}}}+\frac{\alpha_{0}^{2} L^{2}\|\nabla g(\theta_{n},\xi_{n})\|^{2}}{S_{n}} \notag \\
\mathop{\leq}^{(b)} &\frac{1}{2\sigma_0}\left\| \nabla g(\theta_{n})\right\|^2 + 2\sigma_0\alpha_0^2 L^2\frac{\|\nabla g(\theta_{n},\xi_{n})\|^{2}}{S_{n}} + \frac{\alpha_{0}^{2} L^{2}\|\nabla g(\theta_{n},\xi_{n})\|^{2}}{S_{n}},
\end{align}
where $(a)$ uses the smoothness of $g$ such that
\begin{align*}
\|\nabla g(\theta_{n+1})\|-\|\nabla g(\theta_{n})\| \leq \|\nabla g(\theta_{n+1}) - \nabla g(\theta_{n})\| = \alpha_0  L \frac{\left\|\nabla g(\theta_n, \xi_n) \right\|}{\sqrt{S_n}},
\end{align*}
and $(b)$ uses the Cauchy-Schwartz inequality.
Then applying \cref{grad:norm:diff} to \cref{inequ:cross:grad:lambda} yields
\begin{align*}
\zeta(n) \Lambda_n^2 & \leq  \frac{\|\nabla g(\theta_{n})\|^{2}}{\sqrt{S_{n-1}}}-\frac{\|\nabla g(\theta_{n+1})\|^{2}}{\sqrt{S_{n}}} + \frac{\left\| \nabla g(\theta_{n})\right\|^2}{2\sigma_0} + \left( 2\sigma_0 + 1\right)\alpha_{0}^{2} L^{2}\frac{\Gamma_{n}}{\sqrt{S_{n}}}.
\end{align*}
Since $\Gamma_n \leq 1$, by applying the above estimation, the result can be formulated as
\begin{align*}
g(\theta_{n+1})-g(\theta_{n})  \le & -\frac{\alpha_{0}}{4}\zeta(n)+\left(\frac{\alpha_{0}\sigma_{1}}{2\sqrt{S_{0}}}  + \frac{ L\alpha_{0}^{2}}{2}\right)\cdot \Gamma_n  + \frac{\sigma_{0}\left( 2\sigma_0 + 1\right)\alpha_{0}^{3} L^{2}}{2}\frac{\Gamma_{n}}{\sqrt{S_{n}}} \notag \\
&+ \frac{\sigma_{0}\alpha_{0}}{2}\left(\zeta(n) - \zeta(n+1) \right) + \alpha_0 \hat{X}_{n}.
\end{align*}
We further introduce 
\begin{align}\label{notation:ghat:const}
\hat{g}(\theta_n) & = g(\theta_{n}) + \frac{\sigma_{0}\alpha_{0}}{2}\zeta(n), C_{\Gamma,1} = \left(\frac{\alpha_{0}\sigma_{1}}{2\sqrt{S_{0}}}  + \frac{ L\alpha_{0}^{2}}{2}\right);
C_{\Gamma,2}  = \frac{\sigma_{0}\left( 2\sigma_0 + 1\right)\alpha_{0}^{3} L^{2}}{2}
\end{align} to simplify this inequality, which rewrites the inequality to
\begin{align*}
\hat{g}(\theta_{n+1}) - \hat{g}(\theta_n) \leq -\frac{\alpha_{0}}{4}\zeta(n)+C_{\Gamma, 1}\cdot \Gamma_n  + C_{\Gamma, 2}\frac{\Gamma_n}{\sqrt{S_{n}}} + \alpha_0 \hat{X}_{n}.
\end{align*}
\end{proof}

\subsection{The Stability Property of AdaGrad-Norm}\label{subsec:stability}
In this subsection, we will prove the stability of AdaGrad-Norm, which is the foundation for the subsequent asymptotic convergence results, including almost-sure and mean-square convergence. The stability of AdaGrad-Norm is described in the following theorem.
\begin{thm}\label{stable}
If \cref{ass_g_poi,ass_noise} hold, then for AdaGrad-Norm there exists a sufficiently large constant $\tilde{M}>0,$ such that
\begin{equation}\nonumber\begin{aligned}
\Expect\Big[\sup_{n\ge 1}g(\theta_{n})\Big]< \tilde{M}<+\infty,
\end{aligned}\end{equation}
where $\tilde{M}$  depends on the initial state of the algorithm and the constants in assumptions.
\end{thm}
To the best of our knowledge, this is the first result that can establish the stability property of the adaptive gradient methods. The finding in \cref{stable} is crucial for demonstrating the asymptotic convergence of AdaGrad-Norm. 

From \cref{stable}, we can conclude that for any given trajectory, the value of the function remains bounded  (\(\sup_{n \ge 1} g(\theta_{n}) < +\infty \)) almost surely. Note that the boundedness of the expected supremum function value $\mathbb{E}[\sup_{n \ge 1} g(\theta_n)] < \infty$ is a stronger form of stability than the almost-sure boundedness of the supremum alone, i.e., $\sup_{n \ge 1} g(\theta_n) < +\infty$ a.s. The latter condition is insufficient to ensure mean-square convergence.

To prove the stability in \cref{stable}, we first need to introduce and prove \cref{lem:adj:ghat} and \cref{pro_0}.
\begin{lem}\label{lem:adj:ghat}
For the Lyapunov function $\hat{g}(\theta_n)$ we have $$
\hat{g}(\theta_{n+1})-\hat{g}(\theta_{n})\leq h(\hat{g}(\theta_{n})),
$$ where $h(x):=\alpha_0\sqrt{2 L}\left(1 + \frac{\sigma_0 L}{2\sqrt{S_0}}\right)\sqrt{x} + \left(1  + \frac{\sigma_0\alpha_0 L}{2\sqrt{S_0}}\right)\frac{ L\alpha_0^2}{2}$ and  
$h(x) < \frac{x}{2}$ for any $x \geq C_0$ with some constants $C_0$.
\end{lem}
\begin{proof}%(of \cref{lem:adj:ghat})
By the dynamics of AdaGrad-Norm, we have $\|\theta_{n+1}-\theta_{n}\| = \left\|\alpha_0 \frac{\nabla g(\theta_n, \xi_n)}{\sqrt{S_n}}\right\|\le \alpha_{0}\ (\forall\ n>0).$ Then we estimate the change of the Lyapunov function $\hat{g}$ at two adjacent points as
\begin{align*}
\hat{g}(\theta_{n+1})-\hat{g}(\theta_{n})&=g(\theta_{n+1})-g(\theta_{n})+\frac{\sigma_{0}\alpha_{0}}{2}\left(\frac{\left\|\nabla g(\theta_{n+1}) \right\|^2}{\sqrt{S_{n+1}}}-\frac{\left\|\nabla g(\theta_{n}) \right\|^2}{\sqrt{S_{n}}}\right) \notag\\& \mathop{\le}^{(a)} g(\theta_{n+1})-g(\theta_{n})+\frac{\sigma_{0}\alpha_{0}}{2}\frac{\|\nabla g(\theta_{n+1})\|^{2}-\|\nabla g(\theta_{n})\|^{2}}{\sqrt{S_{n}}}
\notag\\&\mathop{\le}^{(b)} \alpha_0\sqrt{2 L\hat{g}(\theta_{n})}+\frac{ L\alpha_0^2}{2} +\frac{\sigma_{0}\alpha_{0}}{2\sqrt{S_{0}}}\big( L\sqrt{2 L\hat{g}(\theta_{n})} \alpha_0 +  L^{2}\alpha_0^2\big), \notag \\
h(\hat{g}(\theta_{n})) & := \sqrt{2 L}\left(1 + \frac{\sigma_0 L}{2\sqrt{S_0}}\right)\alpha_0\sqrt{\hat{g}(\theta_n)} + \left(1  + \frac{\sigma_0\alpha_0 L}{2\sqrt{S_0}}\right)\frac{ L\alpha_0^2}{2},
\end{align*}
{where $(a)$ uses the fact that $S_n \leq S_{n+1}$, $(b)$ follows from the $ L$-smoothness of $g$ and \cref{loss_bound} such that $\|\nabla g(\theta_{n})\|\le \sqrt{2 Lg(\theta_{n})}<\sqrt{2 L\hat{g}(\theta_{n})}$ we have
\begin{align}\label{inequ:g:adj} 
 g(\theta_{n+1}) - g(\theta_n) &\leq \nabla g(\theta_n)^{\top}(\theta_{n+1} - \theta_n) + \frac{ L}{2}\left\|\theta_{n+1} - \theta_n \right\|^2  \notag \\
& \leq \left\|\nabla g(\theta_n)\right\|\left\|\theta_{n+1} - \theta_n\right\|  +\frac{ L}{2}\left\|\theta_{n+1} - \theta_n \right\|^2 \notag\\& \leq \alpha_0\sqrt{2 L\hat{g}(\theta_n)} +\frac{ L\alpha_0^2}{2} 
\end{align}
}
and
\begin{align}
& \left\| \nabla g(\theta_{n+1})\right\|^2 -\left\| \nabla g(\theta_{n})  \right\|^2 \notag \\
\leq &\ \left(2\left\| \nabla g(\theta_{n})\right\| + \left\| \nabla g(\theta_{n+1})\right\| -\left\| \nabla g(\theta_{n})  \right\|\right)\left(\left\| \nabla g(\theta_{n+1})\right\| -\left\| \nabla g(\theta_{n})  \right\|\right) \notag \\
\leq &\ 2 L\left\|\nabla g(\theta_{n}) \right\| \left\|\theta_{n+1} - \theta_n \right\|  +  L^2\left\|\theta_{n+1} - \theta_n \right\|^2 \leq 2 L\alpha_0\sqrt{2 L \hat{g}(\theta_n)} +  L^2\alpha_0^2,
\end{align} 
since $\left\|\nabla g(\theta_{n+1}) \right\| - \left\| \nabla g(\theta_n) \right\| \leq \left\|\nabla g(\theta_{n+1})  - \nabla g(\theta_n) \right\| \leq  L \left\| \theta_{n+1} - \theta_n\right\| $.
There exists a constant \(C_{0}\) that only depends on the parameters of the problem and the initial state of the algorithm, such that if $x \geq C_0$, the following inequality holds
\begin{align*}
h(x)=\sqrt{2 L}\left(1 + \frac{\sigma_0 L}{2\sqrt{S_0}}\right)\alpha_0\sqrt{x} + \left(1  + \frac{\sigma_0\alpha_0 L}{2\sqrt{S_0}}\right)\frac{ L\alpha_0^2}{2}<\frac{x}{2}.
\end{align*} Since we treat $x$ as the variable: LHS is of order  \(\sqrt{x}\) while RHS is of order as \(x\). 
\end{proof}

\begin{property}\label{pro_0}
Under~\cref{ass_g_poi}~\ref{ass_g_poi:i3}, the gradient sublevel set $J_{\eta}:=\{\theta \mid \|\nabla g(\theta)\|^{2}\le \eta\}$ with $\eta >0$ is closed and bounded. Then, by~\cref{ass_g_poi}~\ref{ass_g_poi:i}, there exist a constant $\hat{C}_{g} > 0$  such that $\hat{g}(\theta) < \hat{C}_{g}$ for any $\theta \in J_{\eta}$. 
\end{property}
\begin{proof}%(of \cref{pro_0})
Denote the gradient sublevel set $J_{\eta}:=\{\theta \mid \|\nabla g(\theta)\|^{2}\leq \eta\}$ with $\eta >0$. According to \cref{ass_g_poi}~\ref{ass_g_poi:i3}, $J_{\eta}$ is a closed bounded set. Then by the continuity of $g$, there exist a constant $C_{g} > 0$  such that objective  $g(\theta) \leq C_{g}$ for any $\theta \in J_{\eta}$. For the Lyapunov function $\hat{g}$, we have $\hat{g}(\theta_n
) = g(\theta_n) + \frac{\sigma_0\alpha_0}{2}\frac{\left\|\nabla g(\theta_n) \right\|^2}{\sqrt{S_n}} \leq C_{g}+\frac{\sigma_{0}\alpha_{0}\eta}{2\sqrt{S_{0}}}$ for any $\theta \in J_{\eta}$. Conversely, if there exists $\hat{g}(\theta) >  \hat{C}_{g}:=C_{g}+\frac{\sigma_{0}\alpha_{0}\eta}{2\sqrt{S_{0}}},$ then we must have  $\|\nabla g(\theta)\|^{2} > \eta.$ 
\end{proof}

We are now prepared to present the formal description of the proof of \cref{stable}. To facilitate understanding, we outline the structure of this proof for the readers in~\cref{fig:proof:stability}.
\begin{figure}[ht]
\centering
\begin{tikzpicture}[
  node distance=1.2cm,
  every node/.style={draw, rectangle, minimum width=1.2cm, minimum height=0.6cm, text centered, font=\small},
  every comment/.style={rectangle, draw=none, font=\small},
  >=Stealth, 
  thick]
\node [label={sufficient decrease}](lemma31) {\cref{sufficient:lem}}; % lemma 3.1

\node (lemma32) [right=of lemma31]{ \cref{lem:adj:ghat}}; % lemma 3.2
\node (lemma33) [right= of lemma32]{\cref{pro_0}};  % Property 3.2

\node (lemma34) [below= of lemma31]{ \cref{lem:estimation:supg}}; % lemma 3.3 {lem_su}{lem:psi:i1}
\node (lemma35) [right= of lemma34]{\cref{lem_su}}; % Lemma 3.4

\node (lemma36) [right= of lemma35]{ \cref{lem:psi:i1}}; % Lemma 3.5
\node (lemma37) [label={stability}, below right = 0.3cm and 1.2 cm of lemma33]{\cref{stable}}; % Theorem 3.1

\draw[->] (lemma31) to (lemma34);
\draw[->] (lemma31) to (lemma35);
\draw[->] (lemma31) to (lemma36);
\draw[->] (lemma32) to (lemma34);
\draw[->] (lemma32) to (lemma36);
\draw[->] (lemma33) to (lemma34);
\draw[->] (lemma33) to (lemma37);
\draw[->] (lemma36) to (lemma37);
\draw[->] (lemma35) to node[draw=none, left=-6.8cm, font=\small]{\text{+ Lebesgue's monotone theorem}}(lemma36);
\end{tikzpicture}
\caption{The structure of the proof of \cref{stable}} \label{fig:proof:stability}
\end{figure}
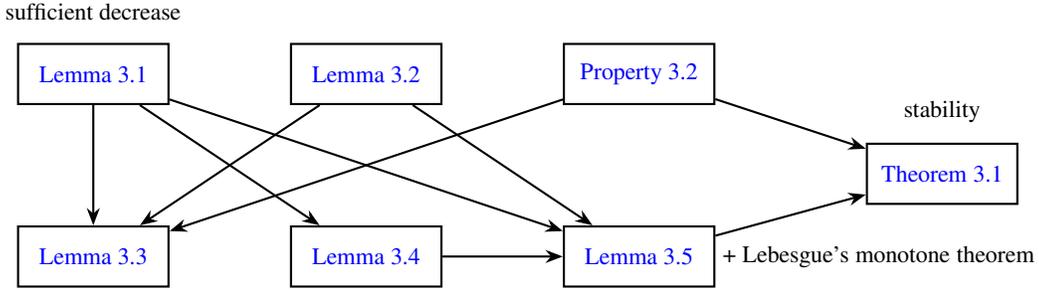
\begin{proof}(of \cref{stable}) \\
\noindent{\bf Phase I:} To demonstrate the stability of the loss function sequence \(\{g(\theta_{n})\}_{n\ge 1}\), the key technique is to segment the entire iteration process according to the value of the Lyapunov function \(\hat{g}(\theta_{n})\). Specifically, we define the non-decreasing stopping times $\{\tau_{t}\}_{t\ge 1}$ as 
\begin{align}\label{stopping_time}
&\tau_{1}:=\min\{k\ge 1:\hat{g}(\theta_{k})>\Delta_0\},\ \tau_{2}:=\min\{k\ge \tau_{1}: \hat{g}(\theta_{k})\le  \Delta_0\ \text{or}\ \hat{g}(\theta_{k})>2\Delta_0\},\notag\\&\tau_{3}:=\min\{k\ge \tau_{2}:\hat{g}(\theta_{k})\le \Delta_0\},...,
\notag\\&\tau_{3i-2}:=\min\{k> \tau_{3i-3}:\hat{g}(\theta_{k})>\Delta_0\},\ \notag \\
&\tau_{3i-1}:=\min\{k\ge \tau_{3i-2}:\hat{g}(\theta_{k})\le  \Delta_0\ \text{or}\ \hat{g}(\theta_{k})>2\Delta_0\},\notag\\&\tau_{3i}:=\min\{k\ge  \tau_{3i-1}:\hat{g}(\theta_{k})\le  \Delta_0\}.
\end{align}
where $\Delta_0:=\max\{C_{0},\hat{C}_{g}\}$ and $C_0, \hat{C}_g$
are defined in \cref{lem:adj:ghat} and \cref{pro_0}, respectively.  
{For the first three stopping time $\tau_1, \tau_2, 
\tau_3$, we have $\tau_1 \leq \tau_2 \leq \tau_3$. When $\tau_1 = \tau_2$, we have $\hat{g}(\theta_{\tau_1}) > 2 \Delta_0$ while we have $\tau_2 < \tau_3$ such that $\hat{g}(\theta_{\tau_3}) \leq \Delta_0$ and $\hat{g}(\theta_n) > \Delta_0$ for $n \in [\tau_1, \tau_3)$. If $\tau_1 < \tau_2$ (that is $\Delta_0 < \hat{g}(\theta_{\tau_1}) < 2\Delta_0$), no matter $\tau_2 = \tau_3$ or $\tau_2 < \tau_3$, we have $\hat{g}(\theta_n) > \Delta_0$ for any $n \in [\tau_1, \tau_3)$. We thus conclude that $\hat{g}(\theta_n) > \Delta_0$ for any $n \in [\tau_{1}, \tau_3)$. }

Next, by the definition of the stopping times $\tau_{3i}$ and $\tau_{3i+1}$, $\forall\ n \in [\tau_{3i}, \tau_{3i+1})$ ($i \geq 1$), we have
\begin{align}\label{fab_2}
\hat{g}(\theta_{n}) \leq \Delta_0.\ \,\, 
\end{align}
Meanwhile, the stopping time $\tau_{3i-1} > \tau_{3i-2}$  holds for $i\ge 2$, because for any $\ i\ge 2$ we have
\[
\Delta_{0} < \hat{g}(\theta_{\tau_{3i-2}}) \le \hat{g}(\theta_{\tau_{3i-2}-1}) + h(\hat{g}(\theta_{\tau_{3i-2}-1})) \le \Delta_{0} + h(\Delta_{0}) \mathop{<}^{(a)} \frac{3\Delta_{0}}{2} < 2\Delta_{0},
\]
where $(a)$ is due to our choice of $\Delta_{0} > C_{0}$ such that $h(\Delta_{0}) < \frac{\Delta_{0}}{2}$ (\cref{lem:adj:ghat}). Combining with this result and the definition of the stopping times $\tau_{3i-1}$, we have for any $n \in [\tau_{3i-2}, \tau_{3i-1})\ (\forall\ i\ge 2)$
\begin{align}\label{fabulous}
g(\theta_{n}) < \hat{g}(\theta_{n}) < 2\Delta_{0}\quad  \text{and }  \quad \hat{g}(\theta_{n}) > \Delta_0.
\end{align}
Thus, the outliers only appear between the stopping times \([ \tau_{3i-1}, \tau_{3i})\). To demonstrate stability in \cref{stable}, we aim to prove that  for any \(T \geq 1\),   \(\mathbb{E}\left[\sup_{1 \le n <T} g(\theta_{n}) \right]\) has an finite upper bound that is independent of \(T\). By the \emph{Lebesgue's monotone convergence} theorem, \(\mathbb{E}\left[\sup_{n \ge 1} g(\theta_{n}) \right]\) is also controlled by this bound.

\noindent{\bf Phase II:} In this step, for any $T \geq 1$, our aim is to estimate $\Expect[\sup_{1\le n<T}g(\theta_{n})]$ based on the segment of $g$ on the stopping time $\tau_t$ defined in the Phase I. For any $T \geq 1$, we define $\tau_{t,T} = \tau_t \wedge T$.
Specifically, we conclude the following auxiliary lemma, whose proof is provided in \cref{appendix:add:proof}.
\begin{lem}\label{lem:estimation:supg}
For the stopping time sequence defined in  \cref{stopping_time} and the intervals $I_{1,\tau}=[\tau_{1,T}, \tau_{3,T})$ and   $I_{i,\tau}^{'}=[\tau_{3i-1,T}, \tau_{3i,T})$, we have %the following estimation for $\Expect[\sup_{1\le n<T}g(\theta_{n})]$:
\begin{align}\label{lem:inequ:supg:main}
& \Expect\Big[\sup_{1\le n< T}g(\theta_{n})\Big] \notag \\
\le&\ \overline{C}_{\Pi,0}+ C_{\Pi,1}C_{\Delta_{0}}\cdot\sum_{i=2}^{+\infty}\underbrace{\Expect\big[\I_{\tau_{3i-1,T}<\tau_{3i,T}}\big]}_{\Psi_{i,1}}+C_{\Pi,1}C_{\Gamma,1}\underbrace{\Expect\left[\bigg(\sum_{I_{1,\tau}}+\sum_{i=2}^{+\infty}\sum_{n=I_{i,\tau}^{'}} \bigg)\Expect[\Gamma_n|\mathscr{F}_{n-1}]\right]}_{\Psi_{2}} \notag \\&+C_{\Pi,1}C_{\Gamma,2}\underbrace{\Expect\Bigg[\bigg(\sum_{n=I_{1,\tau}}+\sum_{i=2}^{+\infty}\sum_{n=I_{i,\tau}^{'}} \bigg)\frac{\Gamma_{n}}{\sqrt{S_{n}}}\Bigg]}_{\Psi_{3}},
\end{align}
where $\overline{C}_{\Pi,0}:=\hat{g}(\theta_{1})+\frac{3\Delta_{0}}{2}+C_{\Pi,0}$,  $C_{\Pi,0},\ C_{\Pi,1}$ and $C_{\Delta_{0}}$ are constants defined in \cref{constant_Pi} and \cref{inequ:C:Delta0} respectively, and $C_{\Gamma,1}, C_{\Gamma,2}$ are constants defined in \cref{sufficient:lem}. 
\end{lem}

\noindent{\bf Phase III:} Next, we prove that the RHS of $\Expect\Big[\sup_{1\le n< T}g(\theta_{n})\Big]$ in \cref{lem:estimation:supg} is uniformly bounded for any \(T\). First, we introduce the following lemma, while the complete proof is provided in~\cref{appendix:add:proof}.
\begin{lem}\label{lem_su} 
Consider AdaGrad-Norm  and suppose that   \cref{ass_g_poi} \cref{ass_g_poi:i}$\sim$\cref{ass_g_poi:i2} and  \cref{ass_noise} \cref{ass_noise:i}$\sim$\cref{ass_noise:i2} hold. Then for any  $\nu>0$, %the following result holds:
\[
\Expect\Bigg[\sum_{n=1}^{+\infty}\I_{\|\nabla g(\theta_{n})\|^{2}>\nu}\frac{\|\nabla g(\theta_{n},\xi_{n})\|^{2}}{S_{n-1}}\Bigg]<\Big({\sigma_{0}}+{\frac{\sigma_{1}}{{\nu}}}\Big)\cdot M<+\infty,\]
where $M$ is a constant that  depends only on the parameters $\theta_{1},$ $S_{0},$ $\alpha_{0}$, $\sigma_{0},$ $\sigma_{1},$ $ L.$
\end{lem}
Then, for the second term $\Psi_{2}$ of RHS of the result in \cref{lem:estimation:supg}, we have
\begin{align}\label{inequ:phi_2}
\Psi_{2} &= \Expect\Bigg[\bigg(\sum_{n=I_{1,\tau}}+\sum_{i=2}^{+\infty}\sum_{n=I_{i,\tau}^{'}} \bigg)\Expect[\Gamma_n|\mathscr{F}_{n-1}]\Bigg] \notag \\
&\mathop{=}^{(a)} \Expect\Bigg[\bigg(\sum_{n=I_{1,\tau}}+\sum_{i=2}^{+\infty}\sum_{n=I_{i,\tau}^{'}} \bigg)\I_{\|\nabla g(\theta_{n})\|^{2}>\eta}\frac{\|\nabla g(\theta_{n},\xi_{n})\|^{2}}{S_{n}}\Bigg] \mathop{<}^{\text{\cref{lem_su}}}\Big({\sigma_{0}}+{\frac{\sigma_{1}}{{\eta}}}\Big)\cdot M,
\end{align}
where  $(a)$ is due to the fact that when the intervals $I_{1,\tau}=[\tau_{1,T}, \tau_{3,T})$ and   $I_{i,\tau}^{'}=[\tau_{3i-1,T}, \tau_{3i,T})$ are non-degenerated, we have $\hat{g}(\theta_{n}) > \Delta_0 \geq \hat{C}_{g}$,  which implies $\|\nabla g(\theta_{n})\|^{2} > \eta$ for any $ n \in I_{1,\tau}  \cup I_{i,\tau}^{'}$  (by \cref{pro_0}). 
For the last term \(\Psi_{3}\) of RHS of the result in \cref{lem:estimation:supg}, by using the series-integral comparison test, we have
\begin{align}\label{inequ:phi_3}
\Psi_{3} = \sum_{i=2}^{+\infty}\Expect\Bigg[\sum_{n=\tau_{3i-1,T}}^{\tau_{3i,T}-1}\frac{\Gamma_{n}}{\sqrt{S_{n}}}\Bigg]<\int_{S_{0}}^{+\infty}\frac{1}{x^{\frac{3}{2}}}\text{d}x<\frac{2}{\sqrt{S_{0}}}.
\end{align}
Then we will prove that there exists a uniform upper bound for $\Psi_{i,1}$ in the following lemma, which is the most challenging part of evaluating $\Expect\Big[\sup_{1\le n< T}g(\theta_{n})\Big]$ in \cref{lem:estimation:supg}.
\begin{lem}\label{lem:psi:i1}
We achieve the following upper bound for $\Psi_{i,1}$ defined in \cref{lem:inequ:supg:main}
\begin{align*} % &\Psi_{i,1} \notag \\
%& \le 
& \frac{4C_{\Gamma,1}}{\Delta_0}\cdot \Expect\Bigg[\sum_{n=\tau_{3i-2,T}}^{\tau_{3i-1,T}-1}\Expect[\Gamma_n|\mathscr{F}_{n-1}]\Bigg]  + \frac{4C_{\Gamma,2}}{\Delta_0}\Expect\Bigg[\sum_{n=\tau_{3i-2,T}}^{\tau_{3i-1,T}-1}\frac{\Gamma_{n}}{\sqrt{S_{n}}}\Bigg] + \frac{4\alpha_0^{2}}{\Delta_0^{2}}\Expect \left[\sum_{n=\tau_{3i-2,T}}^{\tau_{3i-1,T}-1}\hat{X}_{n}^{2}\right].
\end{align*}
\end{lem}
Based on the estimation for the term $\Psi_{i,1}$ in \cref{lem:psi:i1}, we obtain an estimation for its sum
\begin{align}\label{power_01}
\sum_{i=2}^{+\infty}\Psi_{i,1} 
=&\sum_{i=2}^{+\infty}\Expect[\I_{\tau_{3i-1,T} < \tau_{3i,T}}]<\frac{4}{\Delta_0}C_{\Gamma,1}\cdot \sum_{i=2}^{+\infty}\Expect\Bigg[\sum_{n=\tau_{3i-2,T}}^{\tau_{3i-1,T}-1}\Expect[\Gamma_n|\mathscr{F}_{n-1}]\Bigg]\notag  \\&+ \frac{4C_{\Gamma,2}}{\Delta_0}\sum_{i=2}^{+\infty}\Expect\Bigg[\sum_{n=\tau_{3i-2,T}}^{\tau_{3i-1,T}-1}\frac{\Gamma_{n}}{\sqrt{S_{n}}}\Bigg]  + \frac{4\alpha_0^{2}}{\Delta_0^{2}}\sum_{i=2}^{+\infty}\Expect\Bigg[ \sum_{n=\tau_{3i-2,T}}^{\tau_{3i-1,T}-1}\hat{X}_{n}^{2}\bigg].
\end{align}
First, we bound the first term on the RHS of \cref{power_01}. When the interval $[\tau_{3i-2,T}, \tau_{3i-1,T})$ is non-degenerated (i.e., $\tau_{3i-2} < \tau_{3i-1}$), we must have $\hat{g}(\theta_{n}) > \Delta_0 \geq \hat{C}_{g}$. By \cref{pro_0} we have $\|\nabla g(\theta_{n})\|^{2} > \eta$ for any $n \in [\tau_{3i-2,T}, \tau_{3i-1,T})$. Then, we obtain that
\begin{align}\label{jrnn_00}
\sum_{i=2}^{+\infty}\E\left[\sum_{n=\tau_{3i-2,T}}^{\tau_{3i-1,T}-1}\Expect[\Gamma_n|\mathscr{F}_{n-1}] \right]&=\sum_{i=2}^{+\infty}\E\left[\sum_{n=\tau_{3i-2,T}}^{\tau_{3i-1,T}-1}\Expect\bigg[\I_{\|\nabla g(\theta_{n})\|^{2}>\eta}\frac{\|\nabla g(\theta_{n},\xi_{n})\|^{2}}{S_{n}}\bigg]\right]\notag\\& \mathop{<}^{\text{\cref{lem_su}}} \left(\sigma_0 + \frac{\sigma_1}{\eta} \right)M.
\end{align}
For the second term on the RHS of \cref{power_01}, by using the series-integral comparison test, we have:
\begin{align}\label{jrnn_01}
\sum_{i=2}^{+\infty}\Expect\Bigg[\sum_{n=\tau_{3i-2,T}}^{\tau_{3i-1,T}-1}\frac{\Gamma_{n}}{\sqrt{S_{n}}}\Bigg]<\int_{S_{0}}^{+\infty}\frac{1}{x^{\frac{3}{2}}}\text{d}x<\frac{2}{\sqrt{S_{0}}}.
\end{align}
For the third term of \cref{power_01}, we have:
\begin{align}\label{jrnn_02}
& \sum_{i=2}^{+\infty}\Expect\left[ \sum_{n=\tau_{3i-2,T}}^{\tau_{3i-1,T}-1}\hat{X}_{n}^{2}\right] \le 2\sum_{i=2}^{+\infty}\Expect\left[\sum_{n=\tau_{3i-2,T}}^{\tau_{3i-1,T}-1}(X_{n}^{2}+V_{n}^{2})\right]\notag\\& \le\ 2\sum_{i=2}^{+\infty}\Expect\Bigg[\sum_{n=\tau_{3i-2,T}}^{\tau_{3i-1,T}-1}\|\nabla g(\theta_{n})\|^{2}\Gamma_{n}+\bigg(\frac{\sigma_{1}}{2\sqrt{S_{0}}}\Gamma^{2}_{n}+\frac{\sigma_{0}}{2}\Lambda_{n}^{2}\bigg)^{2}\Bigg]\notag\\ & \mathop{\leq}^{(a)}\ 2\Big(4{ L\Delta_0}+\frac{\sigma_{1}}{2\sqrt{S_{0}}}+\frac{\sigma_{0}}{8}\Big)\sum_{i=2}^{+\infty}\Expect\Bigg[\sum_{n=\tau_{3i-2,T}}^{\tau_{3i-1,T}-1}\Gamma_{n}\Bigg]\notag\\ &\mathop{=}^{(b)}\ 2\Big(4{ L\Delta_0}+\frac{\sigma_{1}}{2\sqrt{S_{0}}}+\frac{\sigma_{0}}{8}\Big)\sum_{i=2}^{+\infty}\Expect\bigg[\sum_{n=\tau_{3i-2,T}}^{\tau_{3i-1,T}-1}\I_{\|\nabla g(\theta_{n})\|^{2}>\eta}\frac{\|\nabla g(\theta_{n},\xi_{n})\|^{2}}{{S_{n}}} \bigg]\notag\\
& \leq \ 2\Big(4{ L\Delta_0}+\frac{\sigma_{1}}{2\sqrt{S_{0}}}+\frac{\sigma_{0}}{8}\Big)\sum_{i=2}^{+\infty}\Expect\bigg[\sum _{n=\tau_{3i-2,T}}^{\tau_{3i-1,T}-1}\I_{\|\nabla g(\theta_{n})\|^{2}>\eta}\frac{\|\nabla g(\theta_{n},\xi_{n})\|^{2}}{{S_{n-1}}} \bigg]\notag\\
&\mathop{<}^{\text{\cref{lem_su}}}\ 2\Big(4{ L\Delta_0}+\frac{\sigma_{1}}{2\sqrt{S_{0}}}+\frac{\sigma_{0}}{8}\Big) \left(\sigma_0 + \frac{\sigma_1}{\eta} \right)M,
\end{align}
where $(a)$ is due to when $n\in[\tau_{3i-2,T},\tau_{3i-1,T}),$ there is $\|\nabla g(\theta_{n})\|^{2} \leq 2 Lg(\theta_n) \leq  {4 L\Delta_0},$ and $\Lambda_{n}\le \frac{1}{2}\Gamma_{n}$; $(b)$ is because when the interval $[\tau_{3i-2,T}, \tau_{3i-1,T})$ is non-degenerated (i.e., $\tau_{3i-2} < \tau_{3i-1}$), we have $\hat{g}(\theta_{n}) > \Delta_0 \geq \hat{C}_{g}$. By \cref{pro_0} we have $\|\nabla g(\theta_{n})\|^{2} > \eta$ for any $n \in [\tau_{3i-2,T}, \tau_{3i-1,T})$. Substituting \cref{jrnn_00}, \cref{jrnn_01}, and \cref{jrnn_02} into \cref{power_01} yields
\begin{equation}\nonumber\begin{aligned}
\sum_{i=2}^{+\infty}\Psi_{i,1}&< \frac{4C_{\Gamma,1}}{\Delta_0}\left(\sigma_0 + \sigma_1/\eta \right) M + \frac{4C_{\Gamma,2}}{\Delta_0}\frac{2}{\sqrt{S_0}}\\& + \frac{4\alpha_0^2}{\Delta_0^2} 2\left(4 L\Delta_0+\frac{\sigma_{1}}{2\sqrt{S_{0}}}+\frac{\sigma_{0}}{8}\right) \left(\sigma_0 + \frac{\sigma_1}{\eta} \right)M :=\overline{M},
\end{aligned}\end{equation}
which means there exists a constant $\overline{M} < +\infty$ such that $\sum_{i=2}^{+\infty}\Psi_{i,1}<\overline{M}$.
Combining the above estimation of $\sum_{i=2}^{+\infty}\Psi_{i,1}$ and estimations of $\Psi_2$ and $\Psi_3$ in \cref{inequ:phi_2,inequ:phi_3} into \cref{lem:inequ:supg:main}, we have 
\begin{equation}\nonumber\begin{aligned}
\Expect\Big[\sup_{1\le n<T}g(\theta_{n})\Big]&< \overline{C}_{\Pi,0} + C_{\Pi,1}C_{\Delta_0} \overline{M} + C_{\Pi,1}C_{\Gamma,1}\left(\sigma_0 + \frac{\sigma_1}{\eta} \right)M + C_{\Pi,1}C_{\Gamma,2} \frac{2}{\sqrt{S_0}}\\&:=\overline{M}_{1}<+\infty.
\end{aligned}\end{equation}
Therefore, there exists a constant $\overline{M}_{1} < +\infty$ that is independent on $T$ such that $\Expect\Big[\sup_{1\le n<T}g(\theta_{n})\Big]<+\infty$.
Since \(\overline{M}_{1}\) is independent of \(T\), according to the \emph{ Lebesgue's monotone convergence} theorem, we have $\Expect\Big[\sup_{n\ge 1}g(\theta_{n})\Big]<\overline{M}_{1}<+\infty,$ as we desired.
\end{proof}

\subsection{Almost Sure Convergence of AdaGrad-Norm}\label{subsec:almost:sure}

We now prove the asymptotic convergence under the stability result in \cref{subsec:stability}. We consider the function $g$ to satisfy the following assumptions.
\begin{assumpt}\label{extra}
\begin{enumerate}[label=\textnormal{(\roman*)}, leftmargin=*]
\item\label{extra:i1} (\textbf{Coercivity}) The function $g$ is coercive, that is, $\lim_{\|\theta\|\rightarrow+\infty} g(\theta) = +\infty$.
\item\label{extra:i2} (\textbf{Weak Sard Condition}) The critical value set $\{g(\theta) \mid \nabla g(\theta) = 0\}$ is nowhere dense in $\mathbb{R}$.
\end{enumerate}
\end{assumpt}
Coercivity is commonly employed to ensure the existence of minimizers and to make optimization problems well-posed \citep{rockafellar1970convex}. The weak Sard condition is a relaxed version of the Sard theorem used in non-convex optimization \citep{clarke1990optimization}. It indicates that the set of critical values (where the gradient vanishes) is ``small" in measure. %, helping to demonstrate the convergence to a set of critical points. These conditions are crucial in analyzing stability and convergence in non-convex optimization problems.

{We note that the \emph{weak Sard condition} is implied from the conditions made in \cite{mertikopoulos2020almost}, which requires the $d$-time differentiable objective and the boundedness of the critical points set (the latter is implied from the \emph{non-asymptotically flat} assumption made in their paper). Now we prove this claim.
\begin{claims}\label{pros:assump:loss}
Suppose that $f:\mathbb{R}^{d}\rightarrow\mathbb{R}$ is  $d$-time differentiable and the critical points set $J$  is bounded where $J:=\{\theta \mid \nabla f(\theta)=0\}$. Then, the critical values set $f(J_{f}):=\{f(\theta) \mid \nabla f(\theta)=0\}$, are nowhere dense in $\mathbb{R}$.
\end{claims}
\begin{proof}
Since the critical point set $J$ is bounded, the critical values set $f(J_{f})$ is closed. Suppose that there exists an interval $(a, b)$ such that the set $f(J_{f})$ is dense on this interval. This condition is both necessary and sufficient to guarantee $f(J_{f})$ to have an interior point. Given that $f$ is  $d$-times differentiable, we can apply \textit{Sard's theorem} \citep{sard1942measure,bates1993toward} and deduce that $m(f(J_{f}))=0,$ where $m(\cdot)$ denotes \textit{Lebesgue's Measure}. It is well known that a set containing an interior point cannot have a zero measure. Thus, we conclude that $f(J_{f})$ is nowhere dense in $\mathbb{R}.$ 
\end{proof}
}
Based on the function value's stability in \cref{stable} and the \emph{coercivity} in \cref{extra}~\ref{extra:i1}, it is straightforward to derive the stability of the iteration shown below.
\begin{cor}\label{stable':iterate}
If \cref{ass_g_poi,ass_noise} and \cref{extra}~\ref{extra:i1} hold, given AdaGrad-Norm, we have
$$\sup_{n\ge 1}\|\theta_{n}\|<+\infty\ \ \text{a.s.}$$
\end{cor}
\begin{proof}
From \cref{stable}, we obtain $\Expect[\sup_{n\ge 1}g(\theta_{n})]<+\infty,$ which implies $\sup_{n\ge 1}g(\theta_{n})<+\infty\ \ \text{a.s.}$ Then, by the \emph{coercivity}, it is evident that $\sup_{n\ge 1}\|\theta_{n}\|<+\infty\ \ \text{a.s.}$. 
\end{proof}
For recent studies, \citep{JMLR:v25:23-0576}  directly assumed the iteration's stability (see Assumption 2 in \cite{JMLR:v25:23-0576}) to prove the almost-sure convergence for Adam. \cite{mertikopoulos2020almost} attached the stability for SGD but assumed the uniformly bounded gradient across the entire space \(\theta \in \mathbb{R}^{d}\) which is a strong assumption. \cite{xiao2023convergence,josz2023lyapunov} have achieved the stability of SGD under coercivity. In contrast, our work is the first to establish the stability of adaptive gradient algorithms and to achieve even stronger results regarding the expected function value, as outlined in \cref{stable}.

Before we prove the asymptotic convergence, we establish a key lemma. This demonstrates that the adaptive learning rate of the AdaGrad-Norm algorithm is sufficiently 'large' to prevent premature termination of the algorithm.
\begin{lem}\label{step size}
   Consider AdaGrad-Norm, if \cref{ass_g_poi,ass_noise} hold,  then we have $\sum_{n=1}^{+\infty}\frac{1}{\sqrt{S_{n}}}=+\infty \, \ \text{a.s.}$
\end{lem}

In this part, we will prove the almost sure convergence of AdaGrad-Norm. Combining the stability of $g(\theta_n)$  in \cref{stable} with the property of $S_n$ in \cref{step size}, we adopt the ODE method from stochastic approximation theory to demonstrate the desired convergence~\citep{benaim2006dynamics}. We follow the iterative formula of the standard stochastic approximation (as discussed on page 11 of \cite{benaim2006dynamics})
\begin{align}\label{SA}
x_{n+1} = x_{n} - \gamma_{n}(g(x_{n})+U_{n}),
\end{align}
where $\sum_{n=1}^{+\infty}\gamma_{n}=+\infty$ and $\lim_{n\rightarrow+\infty}\gamma_{n}=0$ and $U_n \in \R^d$ are the random noise (perturbations). Then, we provide the ODE method criterion (c.f. Proposition 4.1 and Theorem 3.2 of \cite{benaim2006dynamics}).
\begin{pros}\label{SA_p}
Let $F$ be a continuous globally integrable vector field. Assume that
\begin{enumerate}[label=\textnormal{(A.\arabic*)},leftmargin=*]
    \item\label{pros:a1} Suppose $\sup_n \|x_n\|< \infty,$ 
    \item\label{pros:a2} For all $T > 0$
    \[
    \lim_{n \to \infty} \sup \left\lbrace \left\lVert \sum_{i=n}^{k} \gamma_{i}U_{i} \right\lVert : k = n, \dots, m(\Sigma_{\gamma}(n) + T) \right\rbrace = 0,
    \]
    where 
    \[\Sigma_\gamma(n):=\sum_{k=1}^{n}\gamma_{k}\ \ \text{and}\ \ m(t):=\max\{j\ge 0:  \Sigma_{\gamma}(j)\le t\}.\] 
    \item\label{pros:a3} \(F(V)\) is nowhere dense on \(\mathbb{R}\), where \(V\) is the fixed point set of the ODE: \(\dot{x} = g(x)\).

\end{enumerate}
Then all limit points of the sequence \(\{x_{n}\}_{n\ge 1}\) are fixed points of the ODE: \(\dot{x} = g(x)\).
\end{pros}
\begin{rem}
\cref{SA_p} synthesizes results from Proposition 4.1, Theorem 5.7, and Proposition 6.4 in \cite{benaim2006dynamics}. Proposition 4.1 shows that the trajectory of an algorithm satisfying~\cref{pros:a1,pros:a2}  forms a precompact asymptotic pseudotrajectory of the corresponding ODE system. Meanwhile, Theorem 5.7 and Proposition 6.4 demonstrate that all limit points of this precompact asymptotic pseudotrajectory are fixed points of the ODE system. 
\end{rem}
We are now ready to present the following theorem on almost sure convergence. To help readers better understand the concepts underlying the proofs, we have included a dependency graph in \cref{fig:adagrad:norm:struc}  that visualizes the relationships among the key lemmas and theorems. 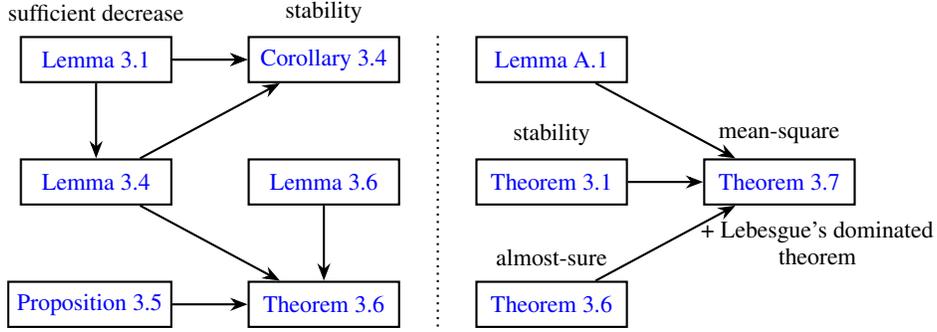
\begin{figure}[t]
\centering
\begin{tikzpicture}[
  node distance=1cm,
  every node/.style={draw, rectangle, minimum width=2cm, minimum height=0.5cm, text centered, font=\small},
  every comment/.style={rectangle, draw=none, font=\small},
  >=Stealth, 
  thick]
 % lemma 5.1

%\node (lemma33) [below=of lemma31]{ \cref{vital_0}}; % lemma D.1
%\node (lemma33) [right= of lemma32]{Lemma \ref{vital_0}};
\node (lemma35) [label={stability}]{\cref{stable':iterate}};  % coro 3.4
\node [label={sufficient decrease}, left= of lemma35](lemma31) {\cref{sufficient:lem}}; % lem 3.1

\node (lemma34) [below= of lemma31]{ \cref{lem_su}}; % lemma 3.4
\node (lemma36) [right= of lemma34]{\cref{step size}}; % Lemma 3.6

%\node (lemma37) [below= of lemma36]{ \cref{RMSProp_9}}; % Lemma D.4 
%\node (lemma38) [below= of lemma37]{ \cref{convergence_v}}; % Lemma D.5
\node (lemma39) [below= of lemma36]{ \cref{convergence_1}}; % Theorem 3.6

\node (lemma41) [right= of lemma35]{ \cref{loss_bound}}; % Lemma A.1
\node (lemma35-1) [label=stability, below= of lemma41]{\cref{stable}};  % Theorem 3.1
\node (lemma39-1) [label=almost-sure, below= of lemma35-1]{ \cref{convergence_1}}; % Theorem 3.6
\node (lemma40) [label=mean-square, right= of lemma35-1, distance=6cm]{ \cref{convergence_2}}; %  THm 3.7
\node (lemma49) [left= of lemma39]{\cref{SA_p}};
% \node (lemma42) [ below= of lemma40]{}; Prop 3.5

 % lemma 5.1

%\node (lemma33) [below =of lemma31] {Lemma \ref{RMSProp_0}};
%\node (lemma34) [below =of lemma33] {Lemma \ref{vital_1}};
%\node (lemma35) [below =of lemma34] {Lemma \ref{vital_0}};
%\node (lemma36) [below =of lemma35] {Theorem \ref{bound''}};

\draw[->] (lemma31) to (lemma35);
\draw[->] (lemma36) to (lemma39);
\draw[->] (lemma31) to (lemma34);
%\draw[->] (lemma33) to (lemma34);
\draw[->] (lemma34) to (lemma35);
\draw[->] (lemma31) to (lemma35);

%\draw[->] (lemma36) to (lemma37);
%\draw[->] (lemma34) to (lemma37);
%\draw[->] (lemma37) to (lemma38);
%\draw[->] (lemma38) to (lemma39);
\draw[->] (lemma34) to (lemma39);
\draw[thick, dotted] ($(lemma35.north east)!0.5!(lemma41.north west)$) to ($(lemma39.south east)!0.5!(lemma39-1.south west)$); 
\draw[->] (lemma49) to (lemma39);
\draw[->] (lemma41) to (lemma40);
\draw[->] (lemma35-1) to (lemma40);
\draw[->] (lemma39-1) to node[draw=none, left=-3.7cm, font=\small, align=center]{+ Lebesgue's dominated\\theorem}(lemma40);
% -- node[midway, left=0.3cm, draw=none, text width=2cm, font=\scriptsize, rotate=90]{\text{\emph{Lebesgue's Dominated Convergence} theorem}}
%\draw[->] (lemma35) to[out=340, in=200] (lemma40);
%\draw[->] (lemma42) -- (lemma45);
%draw[->] (lemma31) -- (lemma45);
%\draw[->] (lemma31) to[out=315, in=45] (lemma36);
%\draw[->] (lemma31) to[out=315, in=45] (lemma37);
%\draw[->] (lemma33) to[out=315, in=45] (lemma37);
%\draw[->] (lemma43) -- (lemma50);
%\draw[->] (lemma50) -- (lemma44);

\end{tikzpicture}
\caption{The proof structure of AdaGrad-Norm} \label{fig:adagrad:norm:struc}
\end{figure} 
\begin{thm}\label{convergence_1} 
Consider the AdaGrad-Norm algorithm defined in \cref{AdaGrad_Norm}. If \cref{ass_g_poi,ass_noise,extra} hold, then for any initial point $\theta_{1}\in \mathbb{R}^{d}$ and $S_{0}>0,$ we have
\begin{equation}\nonumber\begin{aligned}
\lim_{n\rightarrow\infty}\|\nabla g(\theta_{n})\|=0\ \ \text{a.s.}
\end{aligned}\end{equation}
\end{thm}

\begin{proof}%(of \cref{convergence_1} )
First, we consider a degenerate case that the $\mathcal{A}:=\big\{\lim_{n\rightarrow+\infty}S_{n}<+\infty\big\} $ event occurs.
According to \cref{lem_su}, we know that for any \( \nu > 0\), the following result holds
\[\sum_{n=1}^{+\infty}\mathbb{I}_{\|\nabla g(\theta_{n})\|^{2}> \nu}\frac{\|\nabla g(\theta_{n})\|^{2}}{S_{n-1}} < +\infty\ \ \text{a.s.}\]
When the event \(\mathcal{A}\) occurs, it is evident that \(\lim_{n\rightarrow+\infty}\mathbb{I}_{\|\nabla g(\theta_{n})\|^{2}>\nu}\|\nabla g(\theta_{n})\|^{2}=0\ \) \text{a.s.} Furthermore, we have
\begin{align*}
\limsup_{n\rightarrow+\infty}\|\nabla g(\theta_{n})\|^{2} &\le \limsup_{n\rightarrow+\infty}\mathbb{I}_{\|\nabla g(\theta_{n})\|^{2}\le\nu}\|\nabla g(\theta_{n})\|^{2}+\limsup_{n\rightarrow+\infty}\mathbb{I}_{\|\nabla g(\theta_{n})\|^{2}>\nu}\|\nabla g(\theta_{n})\|^{2} \notag \\
& \le \nu+0.
\end{align*}
Due to the arbitrariness of \(\nu\), we can conclude that when \(\mathcal{A}\) occurs, \(\lim_{n\rightarrow+\infty}\|\nabla g(\theta_{n})\|^{2} = 0.\) 

Next, we consider the case that  \(\mathcal{A}\) does not occur (that is $\mathcal{A}^{c}$ occurs), i.e., \(\lim_{n\rightarrow+\infty}S_{n}=+\infty.\) In this case, we transform the AdaGrad-Norm algorithm into the standard stochastic approximation algorithm as below
\[\theta_{n+1}-\theta_{n}=\frac{\alpha_{0}}{\sqrt{S_{n}}}\big(\nabla g(\theta_{n})+(\nabla g(\theta_{n},\xi_{n})-\nabla g(\theta_{n})\big)\] and the corresponding parameters in \cref{SA} are \(x_{n} = \theta_{n}\), \(g(x_n) = \nabla g(\theta_{n})\), \(U_{n} = \nabla g(\theta_{n}, \xi_{n}) - \nabla g(\theta_{n})\), and \(\gamma_{n} = \frac{\alpha_{0}}{\sqrt{S_{n}}}\). When \(\mathcal{A}^{c}\) occurs, it is clear that \(\lim_{n\rightarrow+\infty} \gamma_n = \lim_{n\rightarrow+\infty}\frac{\alpha_{0}}{\sqrt{S_{n}}}=0\). According to \cref{step size}, we know that \( \lim_{n \rightarrow \infty}\Sigma_{\gamma}(n) = \sum_{n=1}^{+\infty} \gamma_n = \sum_{n=1}^{+\infty}\frac{\alpha_{0}}{\sqrt{S_{n}}} = +\infty\ a.s.\) Therefore, it forms a standard stochastic approximation algorithm. 

Next, we aim to verify the two conditions, namely \cref{pros:a1,pros:a2} of \cref{SA_p}, hold for AdaGrad-Norm and use the conclusion of \cref{SA_p} to prove the almost sure convergence of AdaGrad-Norm.  Based on the stability of AdaGrad-Norm in \cref{stable':iterate}, we have \(\sup_{n\ge 1}\|\theta_{n}\|<+\infty\ a.s.\), thus Condition~\cref{pros:a1} holds. Next, we will check whether Condition \cref{pros:a2} is correct. For any \(N>0\), we define the stopping time sequence \(\{\mu_{t}\}_{t\ge 0}\)
\begin{equation}\nonumber\begin{aligned}
&\mu_{0}:=1,\ \mu_{1}:=\max\{n\ge 1:\Sigma_{\gamma}(n)\le N\},\ \mu_{t}:=\max\{n\ge \mu_{t-1}:\Sigma_{\gamma}(n)\le tN\},
\end{aligned}\end{equation}
where $ \Sigma_{\gamma}(n):=\sum_{k=1}^{n}\frac{\alpha_{0}}{\sqrt{S_{k}}}.$ By the definition of the stopping time $\mu_{t}$, we split the value of $\left\lbrace\Sigma_{\gamma}(n) \right\rbrace_{n=1}^{\infty}$ into pieces. 
For any $n>0,$ there exists a stopping time $\mu_{t_{n}}$ such that  $n\in[\mu_{t_{n}},\mu_{t_{n}+1}].$ We recall the definition of $m(t)$ in \cref{SA_p} and get that $m(\Sigma_{S}(n) + N) \leq \mu_{t_n+2}$. We then estimate the sum of $\gamma_{i}U_{i}$ in the interval $[n,m(\Sigma_{\gamma}(n)+N)]$ and achieve that (denote $\sum_{a}^{b}(\cdot)\equiv0\ (\forall\  b  < a)$)
\begin{align}\label{wxm_19}
 &\sup_{k\in[n,m(\Sigma_{\gamma}(n)+N)]}\Bigg\|\sum_{i=n}^{k}\gamma_{i}U_{i}\Bigg\| \notag\\
& =\sup_{k\in [n,m(\Sigma_{\gamma}(n)+N)]}\Bigg\|\sum_{i=\mu_{t_{n}}}^{k}\gamma_{i}U_{i}-\sum_{i=\mu_{t_{n}}}^{n-1}\gamma_{i}U_{i}\Bigg\|\notag\\&\le \sup_{k\in[n,m(\Sigma_{\gamma}(n)+N)]}\Bigg\|\sum_{i=\mu_{t_{n}}}^{k}\gamma_{i}U_{i}\Bigg\|+\sup_{k\in[n,m(\Sigma_{\gamma}(n)+N)]}\Bigg\|\sum_{i=\mu_{t_{n}}}^{n-1}\gamma_{i}U_{i}\Bigg\|\notag\\&\mathop{\le}^{(a)}\sup_{k\in[\mu_{t_{n}}, \mu_{t_{n}+2}]}\Bigg\|\sum_{i=\mu_{t_{n}}}^{k}\gamma_{i}U_{i}\Bigg\|+\sup_{k\in[\mu_{t_{n}}, \mu_{t_{n}+1}]}\Bigg\|\sum_{i=\mu_{t_{n}}}^{k}\gamma_{i}U_{i}\Bigg\|\notag\\&\le 2\sup_{k\in[\mu_{t_{n}}, \mu_{t_{n}+1}]}\Bigg\|\sum_{i=\mu_{t_{n}}}^{k}\gamma_{i}U_{i}\Bigg\| + \sup_{k\in[\mu_{t_{n}+1}, \mu_{t_{n}+2}]}\Bigg\|\sum_{i=\mu_{t_{n}}}^{\mu_{t_n}+1}\gamma_{i}U_{i} + \sum_{i=\mu_{t_{n}+1}}^{k}\gamma_{i}U_{i}\Bigg\| \notag \\
& \leq 3 \sup_{k\in[\mu_{t_{n}}, \mu_{t_{n}+1}]}\Bigg\|\sum_{i=\mu_{t_{n}}}^{k}\gamma_{i}U_{i}\Bigg\|  + \sup_{k\in[\mu_{t_{n}+1}, \mu_{t_{n}+2}]}\Bigg\|\sum_{i=\mu_{t_{n}+1}}^{k}\gamma_{i}U_{i}\Bigg\|,
\end{align}
where (a) follows from the fact that $n \in[\mu_{t_{n}},\mu_{t_{n}+1}]$ and $m(\Sigma_{S}(n) + N) \leq \mu_{t_n+2}$ which implies that \( [n, m(\Sigma_{S}(n) + N)] \subseteq [\mu_{t_{n}}, \mu_{t_{n}+2}] \). From \cref{wxm_19}, it is clear that to verify \cref{pros:a2} we only need to prove $$\lim_{t\rightarrow+\infty}\sup_{k\in[\mu_{t}, \mu_{t+1}]}\big\|\sum_{n=\mu_{t}}^{k}\gamma_{n}U_{n}\big\|=0. $$ First, we decompose $\sup_{k\in[\mu_{t}, \mu_{t+1}]}\big\|\sum_{n=\mu_{t}}^{k}\gamma_{n}U_{n}\big\|$ as below 
\begin{align}\label{wxm_06}
\sup_{k\in[\mu_{t}, \mu_{t+1}]}\Bigg\|\sum_{n=\mu_{t}}^{k}\gamma_{n}U_{n}\Bigg\|=&\sup_{k\in[\mu_{t}, \mu_{t+1}]}\Bigg\|\sum_{n=\mu_{t}}^{k}\frac{\alpha_{0}}{\sqrt{S_{n}}}(\nabla g(\theta_{n},\xi_{n})-\nabla g(\theta_{n}))\Bigg\|\notag\\\le&\underbrace{\sup_{k\in[\mu_{t}, \mu_{t+1}]}\Bigg\|\sum_{n=\mu_{t}}^{k}\frac{\alpha_{0}}{\sqrt{S_{n-1}}}(\nabla g(\theta_{n},\xi_{n})-\nabla g(\theta_{n}))\Bigg\|}_{\Omega_{t}}\notag \\&+\underbrace{\sup_{k\in[\mu_{t}, \mu_{t+1}]}\Bigg\|\sum_{n=\mu_{t}}^{k}\bigg(\frac{\alpha_{0}}{\sqrt{S_{n-1}}}-\frac{\alpha_{0}}{\sqrt{S_{n}}}\bigg)(\nabla g(\theta_{n},\xi_{n})-\nabla g(\theta_{n}))\Bigg\|}_{\Upsilon_{t}}.
\end{align}
Now we only need to demonstrate that \(\lim_{t\rightarrow+\infty}\Omega_{t}=0\) and \(\lim_{t\rightarrow+\infty}\Upsilon_{t}=0\). For the first term $\Omega_{t}$, we have
\begin{align}\label{wxm_05}
\Omega_{t}&=\sup_{k\in[\mu_{t}, \mu_{t+1}]}\Bigg\|\sum_{n=\mu_{t}}^{k}\frac{\alpha_{0}}{\sqrt{S_{n-1}}}(\nabla g(\theta_{n},\xi_{n})-\nabla g(\theta_{n}))\Bigg\|\notag\\&\le\sup_{k\in[\mu_{t}, \mu_{t+1}]}\Bigg\|\sum_{n=\mu_{t}}^{k}\frac{\alpha_{0}\I_{\|\nabla g(\theta_{n})\|^{2}<D_{0}}}{\sqrt{S_{n-1}}}(\nabla g(\theta_{n},\xi_{n})-\nabla g(\theta_{n}))\Bigg\|\notag\\&+\sup_{k\in[\mu_{t}, \mu_{t+1}]}\Bigg\|\sum_{n=\mu_{t}}^{k}\frac{\alpha_{0}\I_{\|\nabla g(\theta_{n})\|^{2}\ge D_{0}}}{\sqrt{S_{n-1}}}(\nabla g(\theta_{n},\xi_{n})-\nabla g(\theta_{n}))\Bigg\| \notag \\&\mathop{\le}^{(a)} \frac{2\delta^{\frac{3}{2}}}{3}+\frac{1}{3\delta^{3}}\underbrace{\sup_{k\in[\mu_{t}, \mu_{t+1}]}\Bigg\|\sum_{n=\mu_{t}}^{k}\frac{\alpha_{0}\I_{\|\nabla g(\theta_{n})\|^{2}<D_{0}}}{\sqrt{S_{n-1}}}(\nabla g(\theta_{n},\xi_{n})-\nabla g(\theta_{n}))\Bigg\|^{3}}_{\Omega_{t,1}}\notag\\&+\frac{\delta}{2}+\frac{1}{2\delta}\underbrace{\sup_{k\in[\mu_{t}, \mu_{t+1}]}\Bigg\|\sum_{n=\mu_{t}}^{k}\frac{\alpha_{0}\I_{\|\nabla g(\theta_{n})\|^{2}\ge D_{0}}}{\sqrt{S_{n-1}}}(\nabla g(\theta_{n},\xi_{n})-\nabla g(\theta_{n}))\Bigg\|^{2}}_{\Omega_{t,2}}
\end{align}
where $(a)$ uses \emph{Young's} inequality twice and $\delta>0$ is an arbitrary number. To check whether $\Omega_{t,1}$ and $\Omega_{t,2}$ converges, we will examine their series \(\sum_{t=1}^{+\infty}\mathbb{E}(\Omega_{t,1})\) and \(\sum_{t=1}^{+\infty}\mathbb{E}(\Omega_{t,2})\). For the series of $\Omega_{t,1}$ we have the following estimation:
\begin{align*}
&\sum_{t=1}^{+\infty}\mathbb{E}(\Omega_{t,1})\le \sum_{t=1}^{+\infty}\Expect\Bigg[\sup_{k\in[\mu_{t}, \mu_{t+1}]}\Bigg\|\sum_{n=\mu_{t}}^{k}\frac{\alpha_{0}\I_{\|\nabla g(\theta_{n})\|^{2}<D_{0}}}{\sqrt{S_{n-1}}}(\nabla g(\theta_{n},\xi_{n})-\nabla g(\theta_{n}))\Bigg\|^{3}\Bigg]\\&\mathop{\le}^{(a)} 3\sum_{t=1}^{+\infty}\Expect\Bigg[\sum_{n=\mu_{t}}^{\mu_{t+1}}\frac{\alpha^{2}_{0}\I_{\|\nabla g(\theta_{n})\|^{2}<D_{0}}}{{S_{n-1}}}\big\|\nabla g(\theta_{n},\xi_{n})-\nabla g(\theta_{n})\big\|^{2}\Bigg]^{\frac{3}{2}}\\&\mathop{\le}^{(b)}3\sum_{t=1}^{+\infty}{\Expect^{1/2}\Bigg[\sum_{n=\mu_{t}}^{\mu_{t+1}}\frac{1}{\sqrt{S_{n-1}}}\Bigg]}\cdot\Expect\Bigg[\sum_{n=\mu_{t}}^{\mu_{t+1}}\frac{\alpha^{3}_{0}\I_{\|\nabla g(\theta_{n})\|^{2}<D_{0}}}{{S^{\frac{5}{4}}_{n-1}}}\|\nabla g(\theta_{n},\xi_{n})-\nabla g(\theta_{n})\|^{3}\Bigg]\\&\mathop{\le}^{(c)}3\alpha^{3}_{0}(\sqrt{D_{0}}+\sqrt{D_{1}}) \cdot\sum_{t=1}^{+\infty}{\Expect^{1/2}\Bigg[\sum_{n=\mu_{t}}^{\mu_{t+1}}\frac{1}{\sqrt{S_{n-1}}}\Bigg]}\Expect\Bigg[\sum_{n=\mu_{t}}^{\mu_{t+1}}\frac{\I_{\|\nabla g(\theta_{n})\|^{2}<D_{0}}}{{S^{\frac{5}{4}}_{n-1}}}\|\nabla g(\theta_{n},\xi_{n})-\nabla g(\theta_{n})\|^{2}\Bigg]\\&\mathop{\le}^{(d)} \frac{3\alpha_{0}^{3}(\sqrt{D_{0}}+\sqrt{D_{1}})}{(N+S_0^{-1/2})^{-\frac{1}{2}}}\cdot\sum_{t=1}^{+\infty}\Expect\Bigg[\sum_{n=\mu_{t}}^{\mu_{t+1}}\frac{\I_{\|\nabla g(\theta_{n})\|^{2}<D_{0}}}{S_{n-1}^{\frac{5}{4}}}\Expect[\|\nabla g(\theta_{n},\xi_{n})-\nabla g(\theta_{n})\|^{2}|\mathscr{F}_{n-1}]\Bigg]\\& \mathop{\le}^{(e)} \frac{3\alpha_{0}^{3}(\sqrt{D_{0}}+\sqrt{D_{1}})}{(N+S_0^{-1/2})^{-\frac{1}{2}}}\Big(\frac{S_{0}+D_{1}}{S_{0}}\Big)^{\frac{5}{4}}\\&\cdot\sum_{t=1}^{+\infty}\Expect\Bigg[\sum_{n=\mu_{t}}^{\mu_{t+1}}\frac{\I_{\|\nabla g(\theta_{n})\|^{2}<D_{0}}}{(S_{n-1}+D_{1})^{\frac{5}{4}}}\Expect(\|\nabla g(\theta_{n},\xi_{n})\|^{2}|\mathscr{F}_{n-1})\Bigg]\\&\mathop{\le}^{(f)}\frac{3\alpha_{0}^{3}(\sqrt{D_{0}}+\sqrt{D_{1}})}{(N+S_0^{-1/2})^{-\frac{1}{2}}}\Big(\frac{S_{0}+D_{1}}{S_{0}}\Big)^{\frac{5}{4}}\sum_{t=1}^{+\infty}\Expect\Bigg[\sum_{n=\mu_{t}}^{\mu_{t+1}}\frac{\I_{\|\nabla g(\theta_{n})\|^{2}<D_{0}}\|\nabla g(\theta_{n},\xi_{n})\|^{2}}{(S_{n-1}+D_{1})^{\frac{5}{4}}}\Bigg]\\&\mathop{\le}^{(g)}\frac{3\alpha_{0}^{3}(\sqrt{D_{0}}+\sqrt{D_{1}})}{(N+S_0^{-1/2})^{-\frac{1}{2}}}\Big(\frac{S_{0}+D_{1}}{S_{0}}\Big)^{\frac{5}{4}}\sum_{t=1}^{+\infty}\Expect\Bigg[\sum_{n=\mu_{t}}^{\mu_{t+1}}\frac{\I_{\|\nabla g(\theta_{n})\|^{2}<D_{0}}\|\nabla g(\theta_{n},\xi_{n})\|^{2}}{S_{n}^{\frac{5}{4}}}\Bigg]\\&<\frac{3\alpha_{0}^{3}(\sqrt{D_{0}}+\sqrt{D_{1}})}{(N+S_0^{-1/2})^{-\frac{1}{2}}}\Big(\frac{S_{0}+D_{1}}{S_{0}}\Big)^{\frac{5}{4}}\int_{S_{0}}^{+\infty}\frac{1}{x^{\frac{5}{4}}}\text{d}x<+\infty.
\end{align*}
Inequality $(a)$ follows from  \emph{Burkholder's} inequality (\cref{vital2}) 
and Inequality $(b)$ uses \emph{Hölder's} inequality, i.e., \(\mathbb{E}(|XY|)^{\frac{3}{2}} \leq \sqrt{\mathbb{E}(|X|^3)} \cdot \mathbb{E}(|Y|^{\frac{3}{2}})\). For Inequality (c), we use \cref{ass_noise:i3} of \cref{ass_noise} such that 
\[\I_{\|\nabla g(\theta_{n})\|^{2}<D_{0}}\|\nabla g(\theta_{n},\xi_{n})-\nabla g(\theta_{n})\|\le \I_{\|\nabla g(\theta_{n})\|^{2}<D_{0}}(\sqrt{D_{0}}+\sqrt{D_{1}}).\]
For inequality (d), we follow from the fact that 
\[
\sum_{n=\mu_{t}}^{\mu_{t+1}} \frac{1}{\sqrt{S_{n-1}}} \leq \frac{1}{\sqrt{S_{\mu_{t}-1}}} + \sum_{n=\mu_{t}}^{\mu_{t+1}} \frac{1}{\sqrt{S_{n}}} \leq \frac{1}{\sqrt{S_{0}}} + N,
\]where we use the definition of the stopping time \(\mu_{t}\). In step (e), note that the function $g(x) = (x+D_1)/x$ is decreasing for $x > 0$. We have $\frac{x + D_{1}}{x} \leq \frac{S_{0} + D_{1}}{S_{0}}$  for any  $x \geq S_{0}$ and 
\begin{align}\label{inequ:var:grad}
\Expect[\|\nabla g(\theta_{n},\xi_{n})-\nabla g(\theta_{n})\|^{2}|\mathscr{F}_{n-1}] & =\Expect[\|\nabla g(\theta_{n},\xi_{n})\|^{2}-\|\nabla g(\theta_{n})\|^{2}|\mathscr{F}_{n-1}] \notag \\
& \le \Expect[\|\nabla g(\theta_{n},\xi_{n})\|^{2}|\mathscr{F}_{n-1}].
\end{align}
In (f), we use the \emph{Doob's stopped} theorem in \cref{vital1}.
In (g), when the event \(\{\|\nabla g(\theta_{n})\|^{2} \leq D_{0}\}\) holds, then \(\|\nabla g(\theta_{n}, \xi_{n})\|^{2} \leq D_{1}\ \text{a.s.}\) such that $S_n =S_{n-1} + \|\nabla g(\theta_{n}, \xi_{n})\|^{2} \leq S_{n-1} + D_1$. We thus conclude that the series $\sum_{t=1}^{+\infty}\mathbb{E}(\Omega_{t,1})$ is bounded. According to \cref{lem_summation}, we have \(\sum_{t=1}^{+\infty}\Omega_{t,1} < +\infty\ a.s.,\)  which implies  
\begin{align}\label{inequ:omega:t1}
\lim_{t\rightarrow+\infty}\Omega_{t,1} = 0\ \text{a.s.} 
\end{align}
Next, we consider the series \(\sum_{t=1}^{+\infty}\mathbb{E}(\Omega_{t,2})\)
\begin{align*}
&\sum_{t=1}^{+\infty}\mathbb{E}[\Omega_{n,2}]=\sum_{t=1}^{+\infty}\Expect\Bigg[\sup_{k\in[\mu_{t}, \mu_{t+1}]}\Bigg\|\sum_{n=\mu_{t}}^{k}\frac{\alpha_{0}\I_{\|\nabla g(\theta_{n})\|^{2}\ge D_{0}}}{\sqrt{S_{n-1}}}(\nabla g(\theta_{n},\xi_{n})-\nabla g(\theta_{n}))\Bigg\|^{2}\Bigg]\\&\mathop{\le}^{(a)} 4\sum_{t=1}^{+\infty}\Expect\Bigg[\sum_{n=\mu_{t}}^{\mu_{t+1}}\frac{\alpha_{0}\I_{\|\nabla g(\theta_{n})\|^{2}\ge D_{0}}}{{S_{n-1}}}\|\nabla g(\theta_{n},\xi_{n})-\nabla g(\theta_{n})\|^{2}\Bigg]\\& \mathop{=}^{\text{\cref{vital1}}} 4\sum_{t=1}^{+\infty}\Expect\Bigg[\sum_{n=\mu_{t}}^{\mu_{t+1}}\frac{\alpha_{0}\I_{\|\nabla g(\theta_{n})\|^{2}\ge D_{0}}}{{S_{n-1}}}\Expect[\|\nabla g(\theta_{n},\xi_{n})-\nabla g(\theta_{n})\|^{2}|\mathscr{F}_{n-1}]\Bigg]\\&\mathop{\le}^{(b)} 4\sum_{t=1}^{+\infty}\Expect\Bigg[\sum_{n=\mu_{t}}^{\mu_{t+1}} \alpha_{0}\I_{\|\nabla g(\theta_{n})\|^{2}\ge D_{0}}\frac{\|\nabla g(\theta_{n},\xi_{n})\|^{2}}{{S_{n-1}}}\Bigg]\\&\mathop{<}^{\text{\cref{lem_su}}} 4 \alpha_0 \left(\sigma_0 + \frac{\sigma_1}{D_0} \right)M,
\end{align*}
where $(a)$ follows from \emph{Burkholder's} inequality (\cref{vital2}) and $(b)$ uses \cref{inequ:var:grad} and the affine noise variance condition in \cref{ass_noise} \cref{ass_noise:i2} such that  \begin{align*}
&\I_{\|\nabla g(\theta_{n})\|^{2}\ge D_{0}}\Expect[\|\nabla g(\theta_{n},\xi_{n})-\nabla g(\theta_{n})\|^{2}|\mathscr{F}_{n-1}]\le \I_{\|\nabla g(\theta_{n})\|^{2}\ge D_{0}}\Expect[\|\nabla g(\theta_{n},\xi_{n})\|^{2}|\mathscr{F}_{n-1}].
\end{align*}
Thus, we obtain that the series $\sum_{t=1}^{+\infty}\mathbb{E}(\Omega_{n,2})$ is bounded. According to \cref{lem_summation}, we have $\sum_{t=1}^{+\infty}\Omega_{n,2}$ is bounded which induces that 
$\lim_{n \rightarrow +\infty} \Omega_{n,2} = 0 \ \text{a.s.} $
Combined with the result that $\lim_{n \rightarrow +\infty} \Omega_{n,1} = 0 \ \text{a.s.}$ in  \cref{inequ:omega:t1} and substituting them into \cref{wxm_05}, we can conclude that \(\limsup_{n\rightarrow+\infty}\Omega_{t}\le \frac{2\delta^{{3}/{2}}}{3}+\frac{\delta}{2}.\) Due to the arbitrariness of \(\delta\), we conclude that $\lim_{n\rightarrow+\infty}\Omega_{t}=0.
$. Next, we consider the term \(\Upsilon_{t}\) in \cref{wxm_06}.
\begin{align}\label{wxm_13}
\Upsilon_{t}=&\sup_{k\in[\mu_{t}, \mu_{t+1}]}\Bigg\|\sum_{n=\mu_{t}}^{k}\bigg(\frac{\alpha_{0}}{\sqrt{S_{n-1}}}-\frac{\alpha_{0}}{\sqrt{S_{n}}}\bigg)(\nabla g(\theta_{n},\xi_{n})-\nabla g(\theta_{n}))\Bigg\|\notag\\\le& \sup_{k\in[\mu_{t}, \mu_{t+1}]}\sum_{n=\mu_{t}}^{k}\bigg(\frac{\alpha_{0}}{\sqrt{S_{n-1}}}-\frac{\alpha_{0}}{\sqrt{S_{n}}}\bigg)\|\nabla g(\theta_{n},\xi_{n})-\nabla g(\theta_{n})\|\notag\\=&\sum_{n=\mu_{t}}^{\mu_{t+1}}\bigg(\frac{\alpha_{0}}{\sqrt{S_{n-1}}}-\frac{\alpha_{0}}{\sqrt{S_{n}}}\bigg)\|\nabla g(\theta_{n},\xi_{n})-\nabla g(\theta_{n})\| \notag\\=&\underbrace{\sum_{n=\mu_{t}}^{\mu_{t+1}}\I_{\|\nabla g(\theta_{n})\|^{2}<D_{0}}\bigg(\frac{\alpha_{0}}{\sqrt{S_{n-1}}}-\frac{\alpha_{0}}{\sqrt{S_{n}}}\bigg)\|\nabla g(\theta_{n},\xi_{n})-\nabla g(\theta_{n})\|}_{\Upsilon_{t,1}}\notag\\&+\underbrace{\sum_{n=\mu_{t}}^{\mu_{t+1}}\I_{\|\nabla g(\theta_{n})\|^{2}\ge D_{0}}\bigg(\frac{\alpha_{0}}{\sqrt{S_{n-1}}}-\frac{\alpha_{0}}{\sqrt{S_{n}}}\bigg)\|\nabla g(\theta_{n},\xi_{n})-\nabla g(\theta_{n})\|}_{\Upsilon_{t,2}}.
\end{align}
We now investigate the sum of the two terms. First, we consider the series \(\sum_{t=1}^{+\infty}\Upsilon_{t,1}\)
\begin{align*}
\sum_{t=1}^{+\infty}\Upsilon_{t,1}&=\sum_{t=1}^{+\infty}\sum_{n=\mu_{t}}^{\mu_{t+1}}\I_{\|\nabla g(\theta_{n})\|^{2}<D_{0}}\bigg(\frac{\alpha_{0}}{\sqrt{S_{n-1}}}-\frac{\alpha_{0}}{\sqrt{S_{n}}}\bigg)\|\nabla g(\theta_{n},\xi_{n})-\nabla g(\theta_{n})\|\\&\mathop{\le}^{(a)}\alpha_{0}(\sqrt{D_{1}}+\sqrt{D_{0}}) \sum_{t=1}^{+\infty}\sum_{n=\mu_{t}}^{\mu_{t+1}}\bigg(\frac{1}{\sqrt{S_{n-1}}}-\frac{1}{\sqrt{S_{n}}}\bigg)\\&<\alpha_{0}(\sqrt{D_{1}}+\sqrt{D_{0}})\sum_{n=1}^{+\infty}\bigg(\frac{1}{\sqrt{S_{n-1}}}-\frac{1}{\sqrt{S_{n}}}\bigg)<\frac{\alpha_{0}(\sqrt{D_{1}}+\sqrt{D_{0}})}{\sqrt{S_{0}}}\ \text{a.s.}, 
\end{align*}
which implies that $
\lim_{t\rightarrow+\infty}\Upsilon_{t,1}=0\ \text{a.s.} $
Inequality $(a)$ follows from \cref{ass_noise} \cref{ass_noise:i3} such that $\I_{\|\nabla g(\theta_{n})\|^{2}<D_{0}}\|\nabla g(\theta_{n},\xi_{n})-\nabla g(\theta_{n})\|\le \sqrt{D_{0}}+\sqrt{D_{1}}\ \text{a.s.}$ Then, we consider the series \(\sum_{t=1}^{+\infty}\mathbb{E}(\Upsilon_{t,2})\)
\begin{align*}
&\sum_{t=1}^{+\infty}\mathbb{E}[\Upsilon_{t,2}]\le \sum_{t=1}^{+\infty}\Expect\Bigg[\sum_{n=\mu_{t}}^{\mu_{t+1}}\I_{\|\nabla g(\theta_{n})\|^{2}\ge D_{0}}\bigg(\frac{\alpha_{0}}{\sqrt{S_{n-1}}}-\frac{\alpha_{0}}{\sqrt{S_{n}}}\bigg)\|\nabla g(\theta_{n},\xi_{n})-\nabla g(\theta_{n})\|\Bigg]\\&\le \alpha_{0}\sum_{t=1}^{+\infty}\Expect\Bigg[\sum_{n=\mu_{t}}^{\mu_{t+1}}\I_{\|\nabla g(\theta_{n})\|^{2}\ge D_{0}}\bigg(\frac{\sqrt{S_{n}}-\sqrt{S_{n-1}}}{\sqrt{S_{n-1}}\sqrt{S_{n}}}\bigg)\|\nabla g(\theta_{n},\xi_{n})-\nabla g(\theta_{n})\|\Bigg]\\&\mathop{\le}^{(a)}\alpha_{0}\sum_{t=1}^{+\infty}\Expect\Bigg[\sum_{n=\mu_{t}}^{\mu_{t+1}}\I_{\|\nabla g(\theta_{n})\|^{2}\ge D_{0}}\bigg(\frac{\|\nabla g(\theta_{n},\xi_{n})\|}{\sqrt{S_{n-1}}\sqrt{S_{n}}}\bigg)\|\nabla g(\theta_{n},\xi_{n})-\nabla g(\theta_{n})\|\Bigg]\\&\le \alpha_{0}\sum_{t=1}^{+\infty}\Expect\bigg[\sum_{n=\mu_{t}}^{\mu_{t+1}}\frac{\I_{\|\nabla g(\theta_{n})\|^{2}\ge D_{0}}}{S_{n-1}}\Expect[\|\nabla g(\theta_{n},\xi_{n})\|\cdot\|\nabla g(\theta_{n},\xi_{n})-\nabla g(\theta_{n})\||\mathscr{F}_{n-1}]\bigg]\\&\mathop{\le}^{(b)} \alpha_{0}\sum_{n=1}^{+\infty}\Expect\Bigg[\I_{\|\nabla g(\theta_{n})\|^{2}\ge D_{0}}\frac{\|\nabla g(\theta_{n},\xi_{n})\|^{2}}{S_{n-1}}\Bigg]\\&\mathop{\le}^{\text{\cref{lem_su}}}\alpha_{0} \left(\sigma_0 + \frac{\sigma_1}{D_0} \right)M,
\end{align*}
where $(a)$ uses the fact that $\sqrt{S_{n}}-\sqrt{S_{n-1}}\le \sqrt{S_{n}-S_{n-1}}=\|\nabla g(\theta_{n},\xi_{n})\|$, $(b)$ uses the similar results in \cref{asddsaasd,asddsa} which uses the affine noise variance condition  (\cref{ass_noise} \cref{ass_noise:i2}) such that \begin{align*}
& \I_{\|\nabla g(\theta_{n})\|^{2}\ge D_{0}}\Expect[\|\nabla g(\theta_{n},\xi_{n})\|\cdot\|\nabla g(\theta_{n},\xi_{n})-\nabla g(\theta_{n})\||\mathscr{F}_{n-1}]\\\le&\ \frac{1}{2}\I_{\|\nabla g(\theta_{n})\|^{2}\ge D_{0}}\left(\Expect[\|\nabla g(\theta_{n},\xi_{n})\|^{2}|\mathscr{F}_{n-1}]+\Expect[\|\nabla g(\theta_{n},\xi_{n})-\nabla g(\theta_{n})\|^{2}|\mathscr{F}_{n-1}] \right)\notag \\
\ \le&\ \I_{\|\nabla g(\theta_{n})\|^{2}\ge D_{0}}\|\nabla g(\theta_{n},\xi_{n})\|^{2}.
\end{align*}
We thus conclude that the series $\sum_{t=1}^{+\infty}\mathbb{E}(\Upsilon_{t,2})$ is bounded. Then, we apply \cref{lem_summation} and achieve that \(\sum_{t=1}^{+\infty}\Upsilon_{t,2} < +\infty \text{ a.s.}\) This induces the result that  \(\lim_{t \rightarrow +\infty} \Upsilon_{t,2} = 0 { a.s.}.\) Combining with the result $\lim_{t \rightarrow +\infty} \Upsilon_{t,1} = 0 { a.s.}$, we get that $\lim_{t\rightarrow+\infty}\Upsilon_{t}\le \lim_{t \rightarrow +\infty} \Upsilon_{t,1}+\lim_{t \rightarrow +\infty} \Upsilon_{t,2}=0\ \text{a.s.}$ Substituting the above results of $\Omega_t $ and $\Upsilon_{t}$ into  \cref{wxm_06}, we derive that
\[\lim_{t\rightarrow+\infty}\sup_{k\in[\mu_{t}, \theta_{t+1}]}\Bigg\|\sum_{n=\mu_{t}}^{k}\gamma_{n}U_{n}\Bigg\|=0\ \ \ \ \text{a.s.}\]
Based on \cref{wxm_19}, we now verify that \cref{pros:a2} in \cref{SA_p} holds. Moreover, by applying \cref{extra}$\sim$\cref{extra:i2}, we confirm that \cref{pros:a3} in \cref{SA_p} is also satisfied. Hence, by \cref{SA_p}, the theorem follows. 
\end{proof}

\subsection{Mean-Square Convergence for AdaGrad-Norm}\label{sec:mean:convergence}
Furthermore, based on the stability of the loss function $g(\theta_n)$ in \cref{stable} and the almost sure convergence in \cref{convergence_1}, it is straightforward to achieve mean-square convergence for AdaGrad-Norm.
\begin{thm}\label{convergence_2}
Consider the AdaGrad-Norm algorithm shown in  \cref{AdaGrad_Norm}. If \cref{ass_g_poi,ass_noise,extra} hold, then for any initial point $ \theta_{1}\in \mathbb{R}^{d}$ and $S_{0}>0,$  we have
\begin{equation}\nonumber\begin{aligned}
\lim_{n\rightarrow\infty}\Expect\|\nabla g(\theta_{n})\|^{2}=0.
\end{aligned}\end{equation}
\end{thm}
 \begin{proof}
By \cref{stable},
\[\Expect\Big[\sup_{n\ge 1}\|\nabla g(\theta_{n})\|^{2}\Big]   \mathop{\le}^{\text{\cref{loss_bound}}} 2 L\Expect\Big[\sup_{n\ge 1}g(\theta_{n})\Big]<+\infty.\]
Then, using the almost sure convergence from \cref{convergence_1} and \emph{Lebesgue's dominated convergence} theorem, we establish \(\lim_{n\rightarrow\infty}\Expect\|\nabla g(\theta_{n})\|^{2}=0.\) 
 \end{proof}
We are the first to establish the mean-square convergence of AdaGrad-Norm based on the stability result under milder conditions. In contrast, existing studies rely on the uniform boundedness of stochastic gradients or true gradients assumptions \citep{JMLR:v25:23-0576,mertikopoulos2020almost}. 

\begin{rem}(Almost-sure vs mean-square convergence)
As stated in the introduction, the almost sure convergence does not imply mean square convergence. To illustrate this concept, let us consider a sequence of random variables $\{\zeta_{n}\}_{n\ge 1},$ where $\pro[\zeta_{n}=0]=1-1/n^{2}$ and $\pro[\zeta_{n}=n^{2}]=1/n^{2}.$ According to \textit{the Borel-Cantelli lemma}, it follows that $\lim_{n\rightarrow+\infty}\zeta_{n}=0$ almost surely. However, it can be shown that $\Expect[\zeta_{n}]=1$ for all $n>0$ by simple calculations.
\end{rem}

\section{A Refined Non-Asymptotic Convergence Analysis of AdaGrad-Norm}\label{sec:nonasympt}
In this section, we present the non-asymptotic convergence rate of AdaGrad-Norm, which is measured by the expected averaged gradients $\frac{1}{T} \sum_{n=1}^{T} \E[\left\|\nabla g(\theta_n) \right\|^2]$. This measure is widely used in the analysis of SGD but is rarely investigated in adaptive methods. We examine this convergence rate under smooth and affine noise variance conditions, which is rather mild. %Moreover, we shall subsequently showcase the convergence rate of the coordinated variant of AdaGrad in \cref{sec:AdaGrad:coordinate}. 

A key step to achieve the expected rate of AdaGrad-Norm is to find an estimation of $\E[S_T]$. We first prepare the following two lemmas, which are important to deriving the convergence result. The proofs of the lemmas are deferred to \cref{appendix:add:proof}.
\begin{lem}\label{lem8}
Under \cref{ass_g_poi}~\ref{ass_g_poi:i}$\sim$\ref{ass_g_poi:i2} and  \cref{ass_noise}~\ref{ass_noise:i}$\sim$ \ref{ass_noise:i2}, for the AdaGrad-Norm algorithm we have
\begin{equation}\nonumber\begin{aligned}
\sum_{n=1}^{T}\Expect\bigg[\frac{\big\|\nabla g(\theta_{n})\big\|^{2}}{\sqrt{S_{n-1}}}\bigg] \leq \mathcal{O}(\ln T).
\end{aligned}\end{equation} 
\end{lem}
\begin{lem}\label{lem23} 
Under \cref{ass_g_poi}~\ref{ass_g_poi:i}$\sim$\ref{ass_g_poi:i2} and  \cref{ass_noise}~\ref{ass_noise:i}$\sim$ \ref{ass_noise:i2}, for the AdaGrad-Norm algorithm we have
\begin{equation}\label{adagrad:123321}\begin{aligned}
\sum_{n=1}^{T}\Expect\Bigg[\frac{g(\theta_{n})\cdot\|\nabla g(\theta_{n})\|^{2}}{\sqrt{S_{n-1}}}\Bigg]=\mathcal{O}(\ln^{2} T).
\end{aligned}\end{equation} 
\end{lem} 
We provide a more accurate estimation of $\E[S_T]$ in \cref{lem:st} than  that of \cite{wang2023convergence}, which only established $\Expect[\sqrt{S_{T}}] = \mathcal{O}(\sqrt{T})$. 
\begin{lem}\label{lem:st}
Consider AdaGrad-Norm in \cref{AdaGrad_Norm} and suppose that   \cref{ass_g_poi}~\ref{ass_g_poi:i}$\sim$\ref{ass_g_poi:i2} and  \cref{ass_noise}~\ref{ass_noise:i}$\sim$ \ref{ass_noise:i2} hold, then for any initial point $\theta_{1}\in \mathbb{R}^{d}$ and $S_{0}>0,$  we have
\begin{align}
\E[S_T] = \mathcal{O}\left(T\right).
\end{align}
\end{lem}
\begin{proof}
Recall the sufficient decrease inequality in \cref{sufficient:lem} and telescope the indices \( n \) from 1 to \( T \). We obtain
\begin{align}\label{inequ:ST:decrease}
\frac{\alpha_{0}}{4}\cdot\sum_{n=1}^{T}\zeta(n) \le&\ \hat{g}(\theta_{1})+\Big(\frac{\alpha_{0}\sigma_{1}}{2\sqrt{S_{0}}}+\frac{ L\alpha_{0}^{2}}{2}\Big)\cdot\sum_{n=1}^{T}\Gamma_n   \notag \\&+\Big( L^{2}\alpha_{0}^{3}\sigma_{0}^{2}+\frac{ L^{2}\alpha_{0}^{3}\sigma_{0}}{2}\Big)\sum_{n=1}^{T}\frac{\|\nabla g(\theta_{n},\xi_{n})\|^{2}}{S_{n}^{\frac{3}{2}}} +\alpha_0\sum_{n=1}^{T}\hat{X}_{n}.
\end{align}
Note that $S_T \geq S_{n-1}$ for all $n \geq [1, T]$. We have
\begin{align}\label{inequ:sum:grad:sn:1}
&\sum_{n=1}^{T}\frac{\|\nabla g(\theta_{n})\|^{2}}{\sqrt{S_{T}}} \leq \sum_{n=1}^{T}\frac{\|\nabla g(\theta_{n})\|^{2}}{\sqrt{S_{n-1}}} ,\notag \\& \sum_{n=1}^{T} \Gamma_n  =  \sum_{n=1}^{T} \frac{\|\nabla g(\theta_{n},\xi_{n})\|^{2}}{S_{n}} \leq \int_{S_{0}}^{S_{T}}\frac{1}{x}\text{d}x \leq \ln (S_T/S_0), \notag \\&\sum_{n=1}^{T}\frac{\|\nabla g(\theta_{n},\xi_{n})\|^{2}}{S^{\frac{3}{2}}_{n}}  \leq \int_{S_{0}}^{+\infty}\frac{1}{x^{\frac{3}{2}}}=\frac{2}{\sqrt{S_{0}}}.
\end{align}
Applying the above results and dividing $\alpha_0/(4\sqrt{S_T})$ over \cref{inequ:ST:decrease} and taking the mathematical expectation on both sides of the above inequality  give
\begin{align}\label{inequ:norm:grad:sum}
\sum_{n=1}^{T}\Expect\|\nabla g(\theta_{n})\|^{2}&\le \bigg(\frac{4g(\theta_{1})}{\alpha_0}+\frac{2\sigma_{0}\|\nabla g(\theta_{1})\|^{2}}{\sqrt{S_{0}}}+\frac{4 L^{2}\alpha_{0}^{2}\sigma_{0}}{\sqrt{S_{0}}}\Big(2\sigma_{0}+1\Big) - \ln(S_0)\bigg)\E\left(\sqrt{S_{T}} \right)\notag \\& +2\Big(\frac{\sigma_{1}}{\sqrt{S_{0}}}+ L\alpha_{0}\Big)\cdot \E\left(\sqrt{S_{T}}\ln (S_{T})\right)+ 4\Expect\bigg[\sqrt{S_{T}}\cdot\sum_{n=1}^{T}\hat{X}_{n}\bigg].
\end{align}
Because $f_1(x)=\sqrt{x}, f_2(x)=\sqrt{x}\ln(x)$ are concave functions, by \emph{Jensen's inequality}, we have 
\begin{align}
&\E\left(\sqrt{S_T} \right) \leq \sqrt{\E\left(S_T \right)}, \quad \E\left(\sqrt{S_T}\ln(S_T) \right) \leq \sqrt{\E\left(S_T \right)} \ln(\E(S_T)),  \label{inequ:ST}\\
& \Expect\bigg[\sqrt{S_{T}}\cdot\sum_{n=1}^{T}\hat{X}_{n}\bigg] \mathop{\leq}^{(a)}   \sqrt{\Expect[S_{T}]\cdot\Expect\bigg[\sum_{n=1}^{T}\hat{X}_{n}\bigg]^{2}},\label{inequ:ST:M'n}
\end{align}
where $(a)$ follows from  \emph{Cauchy Schwartz inequality} for expectation $\E(XY)^2 \leq \E(X^2)\E(Y^2)$. Applying the above estimations in \cref{inequ:ST} and \cref{inequ:ST:M'n} into \cref{inequ:norm:grad:sum}, we have
\begin{equation}\label{zxc_t54}\begin{aligned}
\sum_{n=1}^{T}\Expect\|\nabla g(\theta_{n})\|^{2}&\le C_1 \sqrt{\E\left(S_T \right)} + C_2\sqrt{\E\left(S_T \right)} \ln(\E(S_T)) +\sqrt{\Expect[S_{T}]\cdot\Expect\bigg[\sum_{n=1}^{T}\hat{X}_{n}\bigg]^{2}},
\end{aligned}\end{equation}
where $C_1 = \frac{4g(\theta_{1})}{\alpha_0}+\frac{2\sigma_{0}\|\nabla g(\theta_{1})\|^{2}}{\sqrt{S_{0}}}+\frac{4 L^{2}\alpha_{0}^{2}\sigma_{0}}{\sqrt{S_{0}}}\Big(2\sigma_{0}+1\Big) - \ln(S_0)$ and $C_2 = 2\Big(\frac{\sigma_{1}}{\sqrt{S_{0}}}+ L\alpha_{0}\Big)$.

Now we estimate the term \(\Expect\big[\sum_{n=1}^{T}\hat{X}_{n}\big]^{2}\) in \cref{zxc_t54}. Since  $\left\lbrace \hat{X}_{n},\mathscr{F}_{n}\right\rbrace_{n}^{+\infty}$ is a martingale difference sequence, that is $\forall\ T\ge 1,$ there is
$\Expect\bigg[\sum_{n=1}^{T}\hat{X}_{n}\bigg]^{2}=\sum_{n=1}^{T}\Expect[\hat{X}_{n}]^{2},$ by recalling the definition of $\hat{X}_{n}$ in \cref{sufficient:lem}, we have
\begin{equation}\nonumber\begin{aligned}
\sum_{n=1}^{T}\Expect[\hat{X}_{n}]^{2} \leq&\ 2 \sum_{n=1}^{T}\Expect X_{n}^{2} + 2 \sum_{n=1}^{T}\Expect V_{n}^{2} \notag \\
\le&\ 2\sum_{n=1}^{T}\Expect\bigg[\frac{\|\nabla g(\theta_{n})\|^{2}\cdot\|\nabla g(\theta_{n},\xi_{n})\|^{2}}{S_{n}}\bigg]+\frac{2\alpha_{0}^{2}\sigma_{1}^{2}}{4S_{0}}\sum_{n=1}^{T}\Expect\bigg[\Gamma_n^4\bigg]\notag\\& + \frac{\sigma_0^2}{2}\sum_{n=1}^{T}\Expect\left[\zeta(n)^2 \Lambda_n^4\right]\\
\mathop{\leq}^{(a)} &\ 2\sum_{n=1}^{T}\Expect\bigg[\frac{\|\nabla g(\theta_{n})\|^{2}\cdot\|\nabla g(\theta_{n},\xi_{n})\|^{2}}{S_{n-1}}\bigg]+\frac{\alpha_{0}^{2}\sigma_{1}^{2}}{2S_{0}}\sum_{n=1}^{T}\Expect\bigg[\Gamma_n\bigg]\notag\\& + \frac{\sigma_0^2}{2}\sum_{n=1}^{T}\Expect\left[\zeta(n)^2\right]   \notag \\
\mathop{\le}^{(b)} &\ 2\sigma_{1}\sum_{n=1}^{T}\Expect\bigg[\frac{\|\nabla g(\theta_{n})\|^{2}}{S_{n-1}}\bigg]+4\sigma_{0} L\sum_{n=1}^{T}\E\left(\frac{g(\theta_{n})\|\nabla g(\theta_{n})\|^{2}}{S_{n-1}}\right)\notag\\&+\frac{\alpha_{0}^{2}\sigma_{1}^{2}}{2S_{0}}\Expect[\ln( S_{T} / S_0)]  + \sigma_0^2 L\sum_{n=1}^{T}\E\left(\frac{g(\theta_{n})\|\nabla g(\theta_{n})\|^{2}}{S_{n-1}}\right) ,
\end{aligned}\end{equation}
where $(a)$ follows from the fact that $S_n \geq S_{n-1}$ and $\Lambda_n \leq \Gamma_n  \leq 1$, $(b)$ uses the affine noise variance condition of $\nabla g(\theta_{n},\xi_{n})$ and \cref{loss_bound}, i.e.
\[\Expect[\|\nabla g(\theta_{n},\xi_{n})\|^{2}|\mathscr{F}_{n-1}]\le \sigma_{0}\|\nabla g(\theta_{n})\|^{2}+\sigma_{1}\ \text{and}\ \|\nabla g(\theta_{n})\|^{2}\le 2 Lg(\theta_{n})\ (\text{\cref{loss_bound}}),\]
and the last two terms can be estimated as
\begin{align}
\sum_{n=1}^{T}\Expect\bigg[\Gamma_n\bigg] &  = \Expect \left[\sum_{n=1}^{T} \frac{\left\|\nabla g(\theta_n; \xi_n)\right\|^2}{S_n}\right] = \Expect \left[\int_{S_0}^{S_T} \frac{dx}{x} \right] = \E \left[\ln(S_T/S_0) \right] \\&\leq \ln \E\left[S_T\right]- \ln(S_0), \notag \\
\Expect\left[\zeta(n)^2\right]  &  = \Expect \left[\frac{\left\|\nabla g(\theta_n)\right\|^4}{S_{n-1}}\right] \leq 2 L\E\left[\frac{g(\theta_{n})\|\nabla g(\theta_{n})\|^{2}}{S_{n-1}}\right].
\end{align}
Applying \cref{lem8} and \cref{lem23}, we have
\[\sum_{n=1}^{T}\bigg(\frac{\|\nabla g(\theta_{n})\|^{2}}{S_{n-1}}\bigg)\le \frac{1}{\sqrt{S_{0}}}\sum_{n=1}^{T}\bigg(\frac{\|\nabla g(\theta_{n})\|^{2}}{\sqrt{S_{n-1}}}\bigg)=\mathcal{O}(\ln T),\]\[ \sum_{n=1}^{T}\bigg(\frac{g(\theta_{n})\|\nabla g(\theta_{n})\|^{2}}{S_{n-1}}\bigg)\le \frac{1}{\sqrt{S_{0}}}\sum_{n=1}^{T}\bigg(\frac{g(\theta_{n})\|\nabla g(\theta_{n})\|^{2}}{\sqrt{S_{n-1}}}\bigg)=\mathcal{O}(\ln^{2} T),\] which induces that
\begin{equation}\nonumber\begin{aligned}
\sum_{n=1}^{T}\Expect[\hat{X}_{n}]^{2}\le \frac{\alpha_{0}^{2}\sigma_{1}^{2}}{2S_{0}}\ln\Expect[S_{T}]+\mathcal{O}(\ln^{2}T).
\end{aligned}\end{equation}
Substituting the above estimation of $\sum_{n=1}^{T}\Expect[\hat{X}_{n}]^{2}$ into \cref{zxc_t54}, we have
\begin{equation}\label{zxc_t54'}\begin{aligned}
\sum_{n=1}^{T}\Expect\|\nabla g(\theta_{n})\|^{2}&\le C_1\sqrt{\Expect{S_{T}}}+ \left(C_2 + \frac{\alpha_0\sigma_1}{\sqrt{2S_0}}\right)\sqrt{\Expect[S_{T}]\cdot \ln \Expect[S_{T}]}+\mathcal{O}({\ln T})\cdot\sqrt{\Expect S_{T}}.
\end{aligned}\end{equation} 
Note that by the affine noise variance condition, we have 
\begin{align*}
\E(S_T - S_0)=\E\left[\sum_{n=1}^{T}\left\|\nabla g(\theta_n,\xi_n) \right\|^2\right]= \sum_{n=1}^{T}\E\left[\left\|\nabla g(\theta_n,\xi_n) \right\|^2\right] \leq \sigma_0\sum_{n=1}^{T} \E\left[\left\|\nabla g(\theta_n) \right\|^2\right] + \sigma_1 T, 
\end{align*} 
that is 
\[\sum_{n=1}^{T}\Expect\|\nabla g(\theta_{n})\|^{2}\ge \frac{1}{\sigma_{0}}\Expect[S_{T}]-\frac{\sigma_{1}}{\sigma_{0}}T-\frac{S_{0}}{\sigma_{0}}.\]
Combing the inequality with \cref{zxc_t54'} gives
\begin{equation}\nonumber\begin{aligned}
\Expect[S_{T}]\le \sigma_0C_1\sqrt{\Expect{S_{T}}}+ \sigma_0\left(C_2 + \frac{\alpha_0\sigma_1}{\sqrt{2S_0}}\right)\sqrt{\Expect[S_{T}]\cdot \ln \Expect[S_{T}]}+\mathcal{O}({\ln T})\cdot\sqrt{\Expect S_{T}}+ \sigma_1 T.
\end{aligned}\end{equation}
By treating $\E[S_T]$ as the variable of a function, to estimate $\E[S_T]$ is equivalent to solve
\begin{align}
x \leq \sigma_0C_1\sqrt{x}+ \sigma_0\left(C_2 + \frac{\alpha_0\sigma_1}{\sqrt{2S_0}}\right)\sqrt{x\cdot \ln(x)}+\mathcal{O}({\ln T})\cdot\sqrt{x}+ \sigma_1 T
\end{align}
for any $T \geq 1$. This concludes  
$$\Expect[S_{T}] \leq \mathcal{O}(T), $$ 
where the hidden term of $\mathcal{O}$  depends only on $\theta_{1}$, $S_{0}$, $\alpha_{0}$, $ L$,$\sigma_{0},$ and $\sigma_{1}.$  
\end{proof}

\begin{thm}\label{convergence_rate_0}
Under \cref{ass_g_poi}~\ref{ass_g_poi:i}$\sim$\ref{ass_g_poi:i2} and  \cref{ass_noise}~\ref{ass_noise:i}$\sim$ \ref{ass_noise:i2}, consider the sequence $\{\theta_{n}\}$ generated by AdaGrad-Norm. For any $\theta_{1}\in \mathbb{R}^{d}$ and $S_{0}>0,$ we have
$$\frac{1}{T}\sum_{n=1}^{T}{\Expect\big\|\nabla g(\theta_{n})\big\|^{2}} \leq \mathcal{O}\bigg(\frac{\ln T}{\sqrt{T}}\bigg),\ \ \text{and}\ \ \min_{1\le n\le T}\Expect\big[\|\nabla g(\theta_{n})\|^{2}\big] \leq \mathcal{O}\bigg(\frac{\ln T}{\sqrt{T}}\bigg).$$ 
\end{thm}
\begin{proof}%(of \cref{convergence_rate_0})
By applying the estimation of $\E(S_T)$ in \cref{lem:st} to \cref{zxc_t54'}, we have
\begin{align*}
\frac{1}{T}\sum_{n=1}^{T}\Expect\|\nabla g(\theta_{n})\|^{2}&\le \frac{C_1\sqrt{\sigma_1}}{\sqrt{T}} + \left(C_2 + \frac{\alpha_0\sigma_1}{\sqrt{2S_0}}\right)\frac{\sqrt{\sigma_1}\sqrt{\ln(T)}}{\sqrt{T}}+\frac{\mathcal{O}({\ln T})\sqrt{\sigma_1}}{\sqrt{T}}. 
\end{align*}  
\end{proof}
Note that in Theorem~\ref{convergence_rate_0}, we do not need \cref{ass_g_poi:i3} of \cref{ass_g_poi}  and \cref{ass_noise:i2} of \cref{ass_noise}.  This theorem demonstrates that under smoothness and affine noise variance conditions, AdaGrad-Norm can achieve a near-optimal rate, i.e., $\mathcal{O}\big(\frac{\ln T}{\sqrt{T}}\big).$ It is worth mentioning that the complexity results in Theorem~\ref{convergence_rate_0} is in the expectation sense, rather than in the high probability sense as presented in most of the prior works~\citep{li2020high,defossez2020simple,kavis2022high,liu2022convergence,faw2022power,wang2023convergence}. Our assumptions align with those in \cite{faw2022power,wang2023convergence}, while our result in Theorem~\ref{convergence_rate_0} is stronger compared to the results presented in these works (as denoted in the below corollary).
Meanwhile, we do not impose the restrictive requirement that $\|\nabla g(\theta_{n},\xi_{n})\|$ is almost-surely uniformly bounded, which was assumed in~\cite{ward2020adagrad}. 

Furthermore, Theorem \ref{convergence_rate_0} directly leads to the following stronger high-probability convergence rate result.

\begin{cor}\label{coro:rate}
Under \cref{ass_g_poi}~\ref{ass_g_poi:i}$\sim$\ref{ass_g_poi:i2} and  \cref{ass_noise}~\ref{ass_noise:i}$\sim$ \ref{ass_noise:i2}, consider the sequence $\{\theta_{n}\}$ generated by AdaGrad-Norm. For any initial point $\theta_{1}\in \mathbb{R}^{d}$ and $S_{0}>0,$ we have with probability at least $1-\delta,$
$$\frac{1}{T}\sum_{k=1}^{T}{\big\|\nabla g(\theta_{n})\big\|^{2}} \leq \mathcal{O}\bigg(\frac{1}{\delta}\cdot\frac{\ln T}{\sqrt{T}}\bigg),\ \ \text{and}\ \min_{1\le k\le n}\|\nabla g(\theta_{n})\|^{2} \leq \mathcal{O}\bigg(\frac{1}{\delta}\cdot\frac{\ln T}{\sqrt{T}}\bigg).$$ 
\end{cor}
\begin{proof}%(of \cref{coro:rate})
Applying \emph{Markov's inequality} into \cref{convergence_rate_0} concludes the high probability convergence rate for AdaGrad-Norm. 
\end{proof}
The high-probability results in Corollary~\ref{coro:rate} have a linear dependence on $1/\delta,$ which is better than the quadratic dependence $1/\delta^{2}$ in prior works \citep{faw2022power,wang2023convergence}.

\section{Extension of the Analysis to RMSProp}\label{sec:AdaGrad:coordinate}
In this section, we will employ the proof techniques outlined in \cref{sec:asympt:result} to establish the asymptotic convergence of the coordinated RMSProp algorithm. RMSprop, proposed by~\cite{RMSProp}, is a widely recognized adaptive gradient method. It has attracted much attention with several follow-up studies \citep{xu2021convergence,shi2021rmsprop}. The per-dimensional formula of the coordinated RMSProp is provided below. %,xu2021convergence,shi2021rmsprop}:
\begin{align}\label{RMSProp}
   v_{n,i} &= \beta_n v_{n-1,i} + (1 - \beta_n) (\nabla_i g(\theta_{n},\xi_{n}))^2, \notag\\
   %\eta_{n,i}&=\frac{\alpha_{n}}{\sqrt{v_{n,i}} + \epsilon}, \notag\\
   \theta_{n+1,i} &= \theta_{n,i} - \frac{\alpha_{n}}{\sqrt{v_{n,i}} + \epsilon}\nabla_i g(\theta_{n}, \xi_{n}),
\end{align}
where \( \epsilon > 0 \) is a small number, \( \beta_{n} \in (0,1) \) is a parameter, and \( \alpha_{n} \) is the global learning rate. Here \( \nabla_{i}g(\theta_{n}, \xi_{n}) \) and \( \nabla_{i} g(\theta_{n}) \) denote the \( i \)-th component of the stochastic gradient and the gradient, respectively. We use \( v_{n} := [v_{n,1}, \dots, v_{n,d}]^{\top} \) to denote the corresponding vectors where each component is $v_{n,i}$ (with the initial value \( v_{0} := [v, v, \ldots, v]^{\top} \)), where $v > 0$. In our analysis, we define the variable $\eta_{t,i} = \frac{\alpha_n}{\sqrt{v_{t,i}} + \epsilon}$ and the vector $\eta_t = [\eta_{t,1} \cdots \eta_{t,d}]^T$. We utilize the symbol \( \circ \) to represent the Hadamard product. Consequently, the RMSProp algorithm can be expressed in vector form as: $\theta_{n+1} = \theta_n - \eta_t \circ \nabla g(\theta_{n}, \xi_{n})$. 
%With near-optimal parameter settings, RMSProp is configured as $\alpha_{n}:=\frac{1}{\sqrt{n}}, \beta_{n}:=1-\frac{1}{n}$ (of course, $\frac{1}{\sqrt{n}}$ and $\frac{1}{n}$ can be replaced with their corresponding asymptotically equivalent infinitesimal terms; here, for simplicity of analysis, we use these specific parameters). 

The work in \cite{zou2019sufficient} demonstrated that the RMSProp algorithm can achieve a near-optimal convergence rate of $\mathcal{O}(\ln n/\sqrt{n})$ with high probability under the boundedness of the second-order moment of stochastic gradient and the parameter settings 
\begin{align}\label{near_optimal_para}
\alpha_{n}:=\frac{1}{\sqrt{n}}, \ \  \beta_{n}:=1-\frac{1}{n}\ (\forall\ n\ge 2)\ \text{with}\,\, \beta_{1} \in (0,1).
\end{align}
Furthermore, \cite{zou2019sufficient,chen2022towards} noted that RMSprop can be seen as a coordinate-based version of AdaGrad under these ``near-optimal'' parameter settings. Our analysis of AdaGrad-Norm naturally extends to RMSProp due to the structural similarities with coordinated AdaGrad under this parameter setting of ~\cref{near_optimal_para}.%Next, we give two properties under the near-optimal parameter settings to verify this point. 
%\begin{prop}\label{property_0}
%Under \cref{optimal:rmsprop}, each component \(\eta_{n,i}\) of \(\eta_{n}  = [\eta_{n,1}, \eta_{n,2}, \ldots, \eta_{n,d}]^{\top}\) is monotonically decreasing (w.r.t $n$).  
%\end{prop}
%\begin{proof}
%Recall RMSProp in \cref{RMSProp}, let $\beta_n = 1-1/n$, for each dimension $i$ we have
%\[v_{n,i}=\beta_{n}v_{n-1,i}+(1-\beta_{n}) (\nabla_i g(\theta_n,\xi_n))^2=\Big(1-\frac{1}{n}\Big)v_{n-1,i}+\frac{1}{n}(\nabla_i g(\theta_n,\xi_n))^2,\]
%which implies that $
%nv_{n,i}=(n-1)v_{n-1,i}+ (\nabla_i g(\theta_n,\xi_n))^2\ge (n-1)v_{n-1,i}. $
%We can conclude that \(nv_{n,i}\) is monotonically non-decreasing w.r.t. $n$. Then
%\begin{align*}
%\eta_{n,i}=\frac{\alpha_{n}}{\sqrt{v_{n,i}}+\epsilon}=\frac{\sqrt{n}\alpha_{n}}{\sqrt{nv_{n,i}}+\sqrt{n}\epsilon}=\frac{1}{\sqrt{nv_{n,i}}+\sqrt{n}\epsilon}.
%\end{align*} %%The numerator of the above inequality is monotonically decreasing and greater than 0, while the denominator is monotonically non-increasing and greater than 0. 
%Therefore, we can deduce the monotonic non-increasing property of \(\eta_{n}\). 
%\end{proof}   
%\begin{prop}\label{property_1}
%Under \cref{optimal:rmsprop}, each component \(\eta_{n,i}\) of \(\eta_{n}  = [\eta_{n,1}, \eta_{n,2}, \ldots, \eta_{n,d}]^{\top}\) satisfies that:
%\[\eta_{n,i}\le \frac{1}{\sqrt{S_{n,i}}},\]
%where we define $S_{n}:=\sum_{i=1}^{d}S_{n,i},$ $S_{n,i}:=v+\epsilon+\sum_{k=1}^{n} g_{k,i}^{2}\ (\forall\ n\ge 1),$ and $S_{0,i}:=v+\epsilon.$
%\end{prop}
% \red{From the above two examples, it can be seen that this is completely equivalent to the coordinate-based AdaGrad. }

To analyze RMSprop, we will need to assume variants of  \cref{ass_g_poi}~\ref{ass_g_poi:i3} and \cref{ass_noise}~\ref{ass_noise:i2}~\ref{ass_noise:i3} to be the coordinate-wise versions respectively. 
\begin{assumpt}\label{coordinate_0}
$g(\theta)$ is not asymptotically flat in each coordinate, i.e., there exists $\eta>0,$ for any $i\in[d],$ such that $\liminf_{\|\theta\|\rightarrow+\infty} (\nabla_{i} g(\theta))^{2}>\eta.$ 
\end{assumpt}
\begin{assumpt}\label{coordinate}
The stochastic gradient $\nabla g(\theta_n, \xi_n)$ satisfies 
\begin{enumerate}[label=\textnormal{(\roman*)},leftmargin=*]
\item\label{coordinate_i_1}  
Each coordinate of $\nabla g(\theta_{n},\xi_{n})$  satisfies that \(\Expect[\nabla g_{i}(\theta_{n},\xi_{n})^{2} \mid \mathscr{F}_{n-1}]  \leq \sigma_{0}(\nabla g_{i}(\theta_{n}))^{2} + \sigma_{1}.\)
 \item\label{coordinate_i_2}  For any $i\in[d],$ any $\theta_n$ satisfying $(\nabla_{i} g(\theta_{n}))^{2}<D_{0}$, we have $(\nabla_{i} g(\theta_{n},\xi_{n}))^{2}<D_{1}\ \ \text{a.s.}$ for some constants $D_0, D_1 >0$.
\end{enumerate}
\end{assumpt}
The coordinate-wise affine noise variance condition in ~\cref{coordinate}~\ref{coordinate_i_1} was adopted in \cite{wang2023convergence} when extending the high-probability result of AdaGrad-Norm to coordinated AdaGrad. 
Note that the coordinate affine noise variance condition is less stringent than the typical bounded variance assumption, i.e., \(\mathbb{E} [\|\nabla g(\theta_{n},\xi_{n}) - \nabla g(\theta_{n})\|^{2} \mid \mathscr{F}_{n-1}] < \sigma^{2}.\)

First, we establish the coordinate-wise sufficient descent lemma for RMSProp, as detailed in~\cref{sufficient:lem'},  with the complete proof provided in~\cref{proof:lem:rmsprop}. For simplicity, we define the Lyapunov function
\begin{align}\hat{g}(\theta_t) = g(\theta_{t})+\sum_{i=1}^{d}\zeta_{i}(t)+\frac{\sigma_{1}}{2}\sum_{i=1}^{d}\eta_{t-1,i},
\end{align}
where $\zeta_{i}(t) := (\nabla_{i}g(\theta_{t}))^{2}\eta_{t-1,i}$. In the analysis, we make the special handling for \(v_n\) and then introduce the auxiliary variables $S_{t,i}:=v+\sum_{k=1}^{t}(\nabla_{i}g(\theta_{k},\xi_{k}))^{2}$ and $S_{t} := \sum_{i=1}^{d}S_{t,i}$ to transform RMSProp into a form that aligns with AdaGrad, which allow us to leverage the similar analytical approach.  %Compared to the sufficient lemma in \cref{sufficient:lem} for AdaGrad-Norm, there is additional term $\frac{\sigma_{1}^{2}}{2}\sum_{i=1}^{d}\eta_{t-1,i}$
\begin{lem}\label{sufficient:lem'}
Under \cref{ass_g_poi} \ref{ass_g_poi:i}$\sim$\ref{ass_g_poi:i2}, \cref{ass_noise} \ref{ass_noise:i}, \cref{coordinate} \ref{coordinate_i_1}, consider the sequence $\{\theta_{t}\}$ generated by RMSProp, we have the following sufficient decrease inequality. 
\begin{align}\label{jh_30}
\hat{g}(\theta_{t+1})-\hat{g}(\theta_t) & \le-\frac{3}{4}\sum_{i=1}^{d}\zeta_{i}(t)+\left(\frac{ L}{2}+\frac{(2\sigma_{0}+1) L^{2}}{\sqrt{v}}\right)\|\eta_{t}\circ \nabla g(\theta_{t},\xi_{t})\|^{2}+M_{t},
\end{align} 
where $M_{t}:=M_{t,1}+M_{t,2}+M_{t,3}$ is a martingale difference sequence with \( M_{t,1} \) defined in \cref{adam_0} and \( M_{t,2} \), \( M_{t,3} \) defined in  \cref{adam_2}.
\end{lem}
The first key result for RMSProp is the stability of the function value, which is described in the following theorem. The full proof of \cref{bound''} for RMSProp follows a similar approach to that of AdaGrad, which we defer to \cref{stability:proof:rmsprop}.
\begin{thm}\label{bound''}
Suppose that \cref{ass_g_poi} \ref{ass_g_poi:i}$\sim$\ref{ass_g_poi:i2}, \cref{ass_noise} \ref{ass_noise:i}, \cref{coordinate_0}, \cref{coordinate} \cref{coordinate_i_1} hold. Consider RMSProp. We have
\[\Expect\left[\sup_{n\ge 1}g(\theta_{n})\right] < +\infty.\]
%     The constant hidden in \(\mathcal{O}\) only depends on the constants in the assumptions.
\end{thm}
Building on the stability,  several auxiliary lemmas from  \cref{proof:lem:rmsprop}, and then applying \cref{pros:assump:loss}, we conclude the almost sure convergence for RMSProp. This is the first almost sure convergence for RMSProp to the best of our knowledge. The full proof is provided in~\cref{sec:proof:thm1:rmsprop}.
\begin{thm}\label{convergence_1.0'}
Suppose that \cref{ass_g_poi} \ref{ass_g_poi:i}$\sim$\ref{ass_g_poi:i2}, \cref{ass_noise} \ref{ass_noise:i}, \cref{coordinate_0,coordinate,extra} hold. Consider RMSProp. We have 
\begin{equation}\nonumber\begin{aligned}
\lim_{n\rightarrow\infty}\|\nabla g(\theta_{n})\|=0\ \ \text{a.s.}
\end{aligned}\end{equation}
\end{thm}
By combining the stability in \cref{bound''} with almost sure convergence in \cref{convergence_1.0'}, we apply Lebesgue’s dominated convergence theorem to obtain the mean-square convergence result for RMSProp. 
\begin{thm}\label{convergence_2.0'}
Suppose that \cref{ass_g_poi} \ref{ass_g_poi:i}$\sim$\ref{ass_g_poi:i2}, \cref{ass_noise} \ref{ass_noise:i}, \cref{coordinate_0,coordinate,extra} hold. Consider RMSProp. We have
\begin{equation}\nonumber\begin{aligned}
\lim_{n\rightarrow\infty}\Expect\|\nabla g(\theta_{n})\|^{2}=0.
\end{aligned}\end{equation}
\end{thm}
\begin{proof}
Based on the function value's stability in \cref{bound''}, we can derive the following inequality:
\[\Expect\Big[\sup_{n\ge 1}\|\nabla g(\theta_{n})\|^{2}\Big] \mathop{\le}^{\text{\cref{loss_bound}}} 2 L\Expect\Big[\sup_{n\ge 1}g(\theta_{n})\Big]<+\infty.\]
Then, by the almost sure convergence  from \cref{convergence_1.0'} and \emph{Lebesgue's dominated convergence} theorem, the mean-square convergence result, i.e., \(\lim_{n\rightarrow\infty}\Expect\|\nabla g(\theta_{n})\|^{2}=0\) follows.  
\end{proof}
It is worth mentioning that our approach for establishing the non-asymptotic convergence rate of AdaGrad-Norm can be directly applied to RMSProp under the hyperparameters setting in \cref{near_optimal_para}, which implies \(\frac{1}{T} \sum_{t=1}^{T} \mathbb{E}\|\nabla g(\theta_{n})\|^2 \leq \mathcal{O}(\ln T/\sqrt{T})\). %Since the methods are almost identical, and to avoid extending the length of the paper unnecessarily, we omit the non-asymptotic results and proof here.
%\end{rem}}

\section{Conclusion}\label{sec:conclusion}
%This work addressed several limitations of the theoretical analysis of AdaGrad-Norm. Specifically, we have proposed novel techniques that avoid the no saddle points assumption used in previous works and subsequently establish the asymptotic convergence in both almost surely and mean square senses. Meanwhile, we demonstrate the near-optimal non-asymptotic convergence rate in the expectation sense. Moreover, we alleviate the uniform boundedness assumption on stochastic gradients, which is commonly used in high-probability convergence analysis. 

This study offers a comprehensive analysis of the norm version of AdaGrad and addresses significant gaps in its theoretical framework, particularly regarding asymptotic convergence and non-asymptotic convergence rates in non-convex optimization. By introducing a novel stopping time technique from probabilistic theory, we are the first to establish AdaGrad-Norm stability under mild conditions. Our findings encompass two forms of asymptotic convergence, namely almost sure convergence and mean-square convergence. Additionally, we provide a more precise estimation for $\E[S_T]$ and establish a near-optimal non-asymptotic convergence rate based on expected average squared gradients. The techniques we derived in the proof might be of broader interest to the optimization community. We justify this by applying the techniques to RMSProp with a specific parameter configuration, which provides new insights into the stability and asymptotic convergence of RMSProp. This new perspective reinforces existing findings and paves the way for further exploration of other adaptive optimization techniques, such as Adam. The community might benefit from these new understandings of adaptive methods in optimization in stochastic algorithms, online learning methods, deep learning methods, and beyond.

\bibliographystyle{plainnat}
\bibliography{adagrad_ref}

% \acks{All acknowledgements go at the end of the paper before appendices and references.
% Moreover, you are required to declare funding (financial activities supporting the
% submitted work) and competing interests (related financial activities outside the submitted work).
% More information about this disclosure can be found on the JMLR website.}
\appendix

\newpage

\tableofcontents

\newpage

\section{Appendix: Auxiliary Lemmas for the Theoretical Results}\label{sec:lem:appendix}
\begin{lem}\label{loss_bound} (Lemma 10 of \cite{jin2022convergence})
Suppose that $g(x)$ is differentiable and lower bounded $ g^{\ast} = \inf_{x\in \ \mathbb{R}^{d}}g(x) >-\infty$ and $\nabla g(x)$ is Lipschitz continuous with parameter $ L > 0$, then $\forall \ x\in \ \mathbb{R}^{d}$, we have
\begin{align*}
\big\|\nabla g(x)\big\|^{2}\le {2 L}\big(g(x)-g^{*}\big).
\end{align*}
\end{lem}

	\begin{lem} (Theorem 4.2.1 in \cite{Guo2005}) \label{lem_summation_MDS}	
		Suppose that $\{Y_{n}\}\in \mathbb{R}^{d}$ is a $ L_2$ martingale difference sequence, and $(Y_{n},\mathscr{F}_{n})$ is an adaptive process. Then it holds that $\sum_{k=0}^{+\infty}Y_k<+\infty \ a.s.,$ if there exists $p\in (0,2)$ such that 
		\begin{align*}
  \sum_{n=1}^{+\infty}\Expect[\|Y_{n}\|^{p}]<+\infty,	 \quad   \text{ or }  \quad  \sum_{n=1}^{+\infty}\Expect\big[\|Y_{n}\|^{p}\big|\mathscr{F}_{n-1}\big]<+\infty. \quad \text{a.s.}
		\end{align*}
	\end{lem}		
	%		This lemma is a fundamental conclusion on the weighted sums of the martingale difference \cite{miao2013almost}.	
\begin{lem}\label{lem_summation} (Lemma 6 in \cite{jin2022convergence})
Suppose that $\{Y_n\}\in \mathbb{R}^{d}$ is a non-negative sequence of random variables, then it holds that $\sum_{n=0}^{+\infty} Y_n<+\infty \ a.s.,$ if\ $\sum_{n=0}^{+\infty}\Expect\big[Y_n\big]<+\infty.$   	\end{lem}
\begin{lem}\label{Guo_Lei} (Lemma 4.2.13 in \cite{Guo2005}) Let $\{Y_{n},\mathscr{F}_{n}\}$ be a martingale difference sequence, where $Y_n$ can be a matrix. Let $(U_{n},\mathscr{F}_{n})$ be an adapted process, where $U_{n}$ can be a matrix, and $\|U_{n}\|<+\infty$ almost surely for all $n$. If $\sup_{n}\Expect[\|Y_{n+1}\||\mathscr{F}_{n}]<+\infty\ \ a.s.,$ then we have
$$\sum_{k=0}^{n}U_{n}Y_{n+1}=\mathcal{O}\bigg(\bigg(\sum_{k=0}^{n}\|U_{n}\|\bigg)\ln^{1+\sigma}\Bigg(\bigg(\sum_{k=0}^{n}\|U_{n}\|\bigg)+e\Bigg)\Bigg)\ \ (\forall\ \sigma>0)\ \ \text{a.s.}$$
\end{lem}

\begin{lem}\label{vital2}(Burkholder's inequality)
Let \( \{X_n\}_{n \geq 0} \) be a real-valued martingale difference sequence for a filtration \( \{\mathscr{F}_n\}_{n \geq 0} \), and let \( s\le t<+\infty \) be two stopping time with respect to the same filtration \( \{\mathscr{F}_n\}_{n \geq 0} \). Then for any \( p > 1 \), there exist positive constants \( C_p \) and \( C_p' \) (depending only on \( p \)) such that 
\[C_p \mathbb{E}\left[\bigg(\sum_{n=s}^{t}|X_{n}|^{2}\bigg)^{p/2}\right] \leq \mathbb{E}\left[\sup_{s \leq n \leq t} \bigg|\sum_{k=s}^{n}X_k\bigg|^{p} \right] \leq C_p' \mathbb{E}\left[\bigg(\sum_{n=s}^{t}|X_{n}|^{2}\bigg)^{p/2}\right].\]   
\end{lem}

\begin{lem}\label{vital1}(Doob's stopped theorem)
For an adapted process $(Y_{n}, \mathscr{F}_{n})$, if there exist two bounded stopping times $s \leq t < +\infty\ a.s.$, and if $[s= n] \in \mathscr{F}_{n-1}$ and $[t=n]\in \mathscr{F}_{n-1}$ for all $n > 0$, then the following equation holds.
\[
\Expect\left[\sum_{n=s}^{t}Y_{n}\right] = \Expect\left[\sum_{n=s}^{t}\Expect[Y_{n}|\mathscr{F}_{n-1}]\right].
\]
\end{lem}
If the upper index of the summation is less than the lower index, we define the summation to be zero, i.e., $\sum_{s}^{t}(\cdot) \equiv -\sum_{t}^{s}(\cdot)\ (\forall\ t<s)$. The above equation remains true. 
\begin{lem}\label{sum:expect:ab}
For an adapted process $(Y_{n}, \mathscr{F}_{n})$, and finite stopping times \(a-1,\) \(a\) and \(b\), i.e., $a,\ b<+\infty\ a.s.$
the following equation holds.
\begin{align*}
\E\left[\sum_{n=a}^b Y_n \right] = \E\left[\sum_{n=a}^{b}\Expect[ Y_n|\mathscr{F}_{n-1}] \right].
\end{align*}
\end{lem}
\begin{proof}(of \cref{sum:expect:ab})
\begin{align*}
\E\left[\sum_{n=a}^b Y_n \right]&=\Expect\left[\sum_{n=1}^{b}Y_{n}-\sum_{n=1}^{a-1}Y_{n}\right]=\Expect\left[\sum_{n=1}^{b}Y_{n}\right]-\Expect\left[\sum_{n=1}^{a-1}Y_{n}\right]\\&\mathop{=}^{(a)}\Expect\left[\sum_{n=1}^{b}\Expect\left[Y_{n}|\mathscr{F}_{n-1}\right]\right]-\Expect\left[\sum_{n=1}^{a-1}\Expect\left[Y_{n}|\mathscr{F}_{n-1}\right]\right]\\&=\E\left[\sum_{n=a}^b \Expect[Y_n|\mathscr{F}_{n-1}] \right],
\end{align*}
where in \emph{$(a)$}, we apply \emph{Doob's stopped} theorem, i.e., for any stopping times \(s < +\infty\ \ a.s.\), we have \(\mathbb{E}\left[\sum_{n=1}^{s} Y_n\right] = \mathbb{E}\left[\sum_{n=1}^{s} \Expect[Y_n|\mathscr{F}_{n-1}]\right].\) 
\end{proof}

\begin{lem}\label{lem_S_{T}} 
Consider the AdaGrad-Norm algorithm in  \cref{AdaGrad_Norm} and suppose that   \cref{ass_g_poi}~\ref{ass_g_poi:i}$\sim$\ref{ass_g_poi:i2} and  \cref{ass_noise}~\ref{ass_noise:i}$\sim$ \ref{ass_noise:i2} hold. For any initial point $\theta_{1}\in \mathbb{R}^{d}, S_{0}>0$, and $T \geq 1$, let $\zeta =\sqrt{S_{0}}+\sum_{n=1}^{\infty}\|\nabla g(\theta_{n},\xi_{n})\|^{2}/n^2$. The following results hold.
\begin{itemize}
\item[(a)]  $\E(\zeta)$ is uniformly upper bounded by a constant, which depends on $\theta_1, \sigma_0, \sigma_1, \alpha_0,  L, S_0$.
\item[(b)] $S_T$ is upper bounded by $(1+\zeta)^2T^4$.
\end{itemize}
\end{lem}
\begin{proof}(of \cref{lem_S_{T}})
Recalling the sufficient decrease inequality in \cref{sufficient:lem}
\begin{align*}
\hat{g}(\theta_{n+1})-\hat{g}(\theta_{n}) & \le -\frac{\alpha_{0}}{4}\zeta(n)+C_{\Gamma,1}\cdot \Gamma_n  +C_{\Gamma,2}\frac{\Gamma_{n}}{\sqrt{S_{n}}}  + \alpha_0 \hat{X}_{n}.
\end{align*}
Dividing both sides of the inequality by $n^2\alpha_0/4$, we obtain
\begin{align}\label{zxc_09}
\frac{1}{n^{2}} \zeta(n) &\le \frac{4}{\alpha_0n^{2}}\big(\hat{g}(\theta_{n})-\hat{g}(\theta_{n+1})\big)+\frac{4C_{\Gamma,1}}{\alpha_0}\cdot\frac{\Gamma_n}{n^{2}}  +\frac{4C_{\Gamma,2}}{\alpha_0}\frac{\|\nabla g(\theta_{n},\xi_{n})\|^{2}}{n^2 S_{n}^{\frac{3}{2}}}+\frac{4 \hat{X}_{n}}{n^{2}}.
\end{align}
For the second term on the RHS of \cref{zxc_09}, we use \emph{Young's inequality} and $S_n \geq S_{n-1}$:
\begin{equation}\nonumber\begin{aligned}
 \frac{4C_{\Gamma,1}}{\alpha_0}\cdot\frac{\Gamma_n}{n^{2}} &\le \frac{\|\nabla g(\theta_{n},\xi_{n})\|^{2}}{2n^{2}\sqrt{S_{n}}}+\frac{16C_{\Gamma,1}^2}{\alpha_0^2}\frac{\|\nabla g(\theta_{n},\xi_{n})\|^{2}}{2n^2S_{n}^{\frac{3}{2}}}  \le \frac{\|\nabla g(\theta_{n},\xi_{n})\|^{2}}{2n^{2}\sqrt{S_{n-1}}}+\frac{16C_{\Gamma,1}^2}{\alpha_0^2}\frac{\|\nabla g(\theta_{n},\xi_{n})\|^{2}}{2n^2S_{n}^{\frac{3}{2}}}.
\end{aligned}\end{equation}
Substituting the above inequality into \cref{zxc_09} gives
\begin{align*}
\frac{\zeta(n)}{2n^{2}}  &\le \frac{4}{\alpha_0n^{2}}\big(\hat{g}(\theta_{n})-\hat{g}(\theta_{n+1})\big)+\left(\frac{4C_{\Gamma,2}}{\alpha_0} + \frac{8C_{\Gamma,1}^2}{\alpha_0^2}\right)\frac{\|\nabla g(\theta_{n},\xi_{n})\|^{2}}{n^2 S_{n}^{\frac{3}{2}}}    +\frac{4 \hat{X}_{n}}{n^{2}}.
\end{align*}
Telescoping the indices \( n \) from 1 to \( T \) over the above inequality, we have
\begin{align}\label{inequ:sequence:1}
\sum_{n=1}^{T} \frac{1}{2n^2}\zeta(n) &\le \sum_{n=1}^{T} \frac{4}{\alpha_0n^{2}}\big(\hat{g}(\theta_{n})-\hat{g}(\theta_{n+1})\big) +\mathcal{C}_1 \sum_{n=1}^T\frac{\|\nabla g(\theta_{n},\xi_{n})\|^{2}}{n^2 S_{n}^{\frac{3}{2}}}  + 4\sum_{n=1}^{T}\frac{\hat{X}_{n}}{n^{2}},
\end{align}
where we use $\mathcal{C}_1$ to denote the coefficient constant factor of $\frac{\|\nabla g(\theta_{n},\xi_{n})\|^{2}}{n^2 S_{n}^{\frac{3}{2}}}$ to simplify the expression. For the first term of RHS of \cref{inequ:sequence:1}, since $\hat{g}(\theta_n) = g(\theta_n) + \sigma_0\alpha_0\zeta(n)/2 \geq 0$ for all $n \geq 1$, we have 
\begin{align}
& \sum_{n=1}^{T}\frac{1}{n^{2}}\big(\hat{g}(\theta_{n})-\hat{g}(\theta_{n+1})\big) =  \sum_{n=1}^{T} \frac{\hat{g}(\theta_{n})}{n^2}-\frac{\hat{g}(\theta_{n+1})}{(n+1)^2} + \frac{\hat{g}(\theta_{n+1})}{(n+1)^{2}} - \frac{\hat{g}(\theta_{n+1})}{n^{2}} \notag \\
=  &\sum_{n=1}^{T} \frac{\hat{g}(\theta_{n})}{n^2}-\frac{\hat{g}(\theta_{n+1})}{(n+1)^2} - \frac{\hat{g}(\theta_{n+1})(2n+1)}{(n+1)^{2}n^2}  \leq  \hat{g}(\theta_1).
\end{align}
For the second term of RHS of \cref{inequ:sequence:1}, we utilized the series-integral result
\[\sum_{n=1}^{T}\frac{\|\nabla g(\theta_{n},\xi_{n})\|^{2}}{n^2 S_{n}^{\frac{3}{2}}} \leq \sum_{n=1}^{T}\frac{\|\nabla g(\theta_{n},\xi_{n})\|^{2}}{S_{n}^{\frac{3}{2}}}<\int_{S_{0}}^{+\infty}\frac{1}{x^{\frac{3}{2}}}\text{d}x=\frac{2}{\sqrt{S_{0}}}.\]
Applying the above estimations into \cref{inequ:sequence:1} and taking the mathematical expectation on both sides, we have $\forall\ n\ge 1,$
\begin{align}\label{inequ:sequence:expect}
\sum_{n=1}^{T}\frac{\E\left[\zeta(n)\right]}{2n^2} &\le \frac{4}{\alpha_0} \hat{g}(\theta_1) + \frac{2}{\sqrt{S_0}}\mathcal{C}_1   + 4\sum_{n=1}^{T}\frac{\E[\hat{X}_{n}]}{n^{2}} = \frac{4}{\alpha_0} \hat{g}(\theta_1) + \frac{2}{\sqrt{S_0}}\mathcal{C}_1,
\end{align}
since $\lbrace \hat{X}_{n}, \mathscr{F}_{n-1}\rbrace$ is a martingale difference sequence. According to \emph{the affine noise variance condition}, we obtain:
\begin{align}\label{adagrad_100-1}
\sum_{n=1}^{T}\frac{\E\left[\zeta(n)\right]}{2n^2}\ge \sum_{n=1}^{T}\frac{\E\left[\|\nabla g(\theta_{n},\xi_{n})\|^{2}\right]}{2\sigma_{0}n^2}-\frac{\sigma_{1}}{2\sigma_{0}}\sum_{n=1}^{T}\frac{1}{n^{2}}\mathop{\ge}^{(a)} \sum_{n=1}^{T}\frac{\E\left[\|\nabla g(\theta_{n},\xi_{n})\|^{2}\right]}{2\sigma_{0}n^2}-\frac{\sigma_{1}\pi^{2}}{12\sigma_{0}}.
\end{align}
Here, $(a)$ uses the inequity \[\sum_{n=1}^{T}\frac{1}{n^{2}}<\sum_{n=1}^{+\infty}\frac{1}{n^{2}}=\frac{\pi^{2}}{6}.\] 
Combining \cref{inequ:sequence:expect} with \cref{adagrad_100-1}, we obtain
\begin{align*}
\Expect\bigg[\sum_{n=1}^{T}\frac{\|\nabla g(\theta_{n},\xi_{n})\|^{2}}{2\sigma_{0}n^{2}}\bigg]=\sum_{n=1}^{T}\frac{\E\left[\|\nabla g(\theta_{n},\xi_{n})\|^{2}\right]}{2\sigma_{0}n^2}\le \frac{\sigma_{1}\pi^{2}}{12\sigma_{0}}+\frac{4}{\alpha_0} \hat{g}(\theta_1) + \frac{2}{\sqrt{S_0}}\mathcal{C}_1.
\end{align*}
By \emph{Lebesgue monotone convergence} theorem, we further get that $\zeta=\sqrt{S_{0}}+\sum_{n=1}^{+\infty}\|\nabla g(\theta_{n},\xi_{n})\|^{2}\big/n^{2}<+\infty\ \ a.s.,$ and
\begin{equation}\label{zxc_10}\begin{aligned}
\Expect[\zeta]=\sqrt{S_{0}}+\Expect\bigg[\sum_{n=1}^{T}\frac{\|\nabla g(\theta_{n},\xi_{n})\|^{2}}{n^{2}}\bigg] \leq   \sqrt{S_{0}}+\frac{\sigma_{0}\sigma_{1}\pi^{2}}{6\sigma_{0}}+\frac{16\sigma_{0}}{\alpha_0} \hat{g}(\theta_1) + \frac{8\sigma_{0}}{\sqrt{S_0}}\mathcal{C}_1.
\end{aligned}\end{equation}
Next, we derive the relationship of $S_T$ and the $\zeta$. Note that $\forall\ T\ge 1,$
\[\sum_{n=1}^{T}\frac{\|\nabla g(\theta_{n},\xi_{n})\|^{2}}{n^{2}\sqrt{S_{n-1}}}>\frac{1}{T^{2}\sqrt{S_{T}}}\sum_{n=1}^{T}\|\nabla g(\theta_{n},\xi_{n})\|^{2}=\frac{S_T -S_0}{T^{2}\sqrt{S_{T}}} .\] We have
\begin{equation}\nonumber\begin{aligned}
\sqrt{S_T} &\leq \bigg(\sum_{n=1}^{T}\frac{\|\nabla g(\theta_{n},\xi_{n})\|^{2}}{n^{2}\sqrt{S_{n-1}}}\bigg) \cdot T^2+ \sqrt{S_0}  \leq \bigg(\sum_{n=1}^{T}\frac{\|\nabla g(\theta_{n},\xi_{n})\|^{2}}{n^{2}\sqrt{S_{n-1}}} + \sqrt{S_0} \bigg)\cdot T^2=\zeta\cdot T^{2}\\&<(1+\zeta)\cdot T^{2},
\end{aligned}\end{equation}
as we desired. 
\end{proof}
%  \begin{lem}\label{lem_S_{T}} 
% Under items $(i)\sim(ii)$ of \cref{ass_g_poi} and  items $(i)\sim(ii)$ of \cref{ass_noise}, consider the sequence $\{\theta_{n}\}_{n\ge 1}$ generated by AdaGrad-Norm in Eq. (\ref{AdaGrad_Norm}). Then $\forall \theta_{1}\in \mathbb{R}^{d},$ $\forall\ T\ge 1,$ $\exists\ \zeta\in \mathscr{F}_{\infty},$ such that $\zeta\ge 1,$ $\Expect(\ln^{2}(\zeta))<+\infty,$ and \(S_{T}<\zeta\cdot T^{4},\)
% where the upper bound of $\Expect(\ln^{2}(\zeta))$ only depends on $\theta_{1},$ $S_{0},$ $\alpha_{0},$ $\sigma_{0},$ $ L,$ $\sigma_{1}.$
% \end{lem} 

% Manual newpage inserted to improve layout of sample file - not
% needed in general before appendices/bibliography.

\section{Appendix: Additional Proofs in~\cref{sec:asympt:result}}\label{appendix:add:proof}
\subsection{Proofs of Lemmas in \cref{subsec:stability}}
\begin{proof}(of \cref{lem:estimation:supg})
For any $T\geq 1,$ we calculate \(\E\left(\sup_{n \ge 1} g(\theta_{n})\right)\) based on the segment of $g$ on the stopping time
\begin{align}\label{start}
  &\quad \Expect\Big[\sup_{1\le n<T}g(\theta_{n})\Big] \notag \\&\le \Expect\Big[\sup_{1\le n< \tau_{1,T}}g(\theta_{n})\Big]  +\Expect\Big[\sup_{\tau_{1,T}\le n<T}g(\theta_{n})\Big]\notag\\&=\Expect\Big[\I_{[\tau_{1,T}=1]}\sup_{1\le n< \tau_{1,T}}g(\theta_{n})\Big]+\underbrace{\Expect\Big[\I_{[\tau_{1,T}>1]}\sup_{1\le n< \tau_{1,T}}g(\theta_{n})\Big]}_{\Pi_{1, T}} +\underbrace{\Expect\Big[\sup_{\tau_{1,T}\le n< T}g(\theta_{n})\Big]}_{\Pi_{2,T}}\notag\\&\mathop{\le}^{(a)}0+\Delta_{0}+\Pi_{2,T},
\end{align}
where we define $\tau_{t,T}:=\tau_{t}\wedge T.$ 
To make the inequality consistent, we let $\sup_{a \le t < b}(\cdot) = 0 \ (\forall\ a \ge b).$ For $(a)$ in \cref{start}, since $\tau_{1,T} \geq 1$, we have  $\Expect\Big[\I_{[\tau_{1,T}=1]}\sup_{1\le n< \tau_{1,T}}g(\theta_{n})\Big]=0$ and
\begin{align*}
\Pi_{1, T}&=\Expect\Big[\I_{[\tau_{1,T}>1]}\sup_{1\le n< \tau_{1,T}}g(\theta_{n})\Big]\le\Expect\Big[\I_{[\tau_{1}>1]}\sup_{1\le n< \tau_{1,T}}g(\theta_{n})\Big]\le\Delta_{0}.%+\Expect\
\end{align*}
Next, we focus on  $\Pi_{2,T}$. Specifically, we have:
\begin{align}\label{inequ:xt2}
\Pi_{T,2}&=\Expect\Big[\sup_{\tau_{1,T}\le n<T}g(\theta_{n})\Big]=\Expect\bigg[\sup_{i\ge 1}\Big(\sup_{\tau_{3i-2,T}\le n<\tau_{3i+1,T}}g(\theta_n)\Big)\bigg]\notag\\&\le\underbrace{\Expect\bigg[\Big(\sup_{\tau_{1,T}\le n<\tau_{4,T}}g(\theta_n)\Big)\bigg]}_{\Pi_{2,T}^1}+\underbrace{\Expect\bigg[\sup_{i\ge 2}\Big(\sup_{\tau_{3i-2,T}\le n<\tau_{3i+1,T}}g(\theta_n)\Big)\bigg]}_{\Pi_{2,T}^2}.
\end{align}
We decompose $\Pi_{2,T}$ into $\Pi_{2,T}^1$ and $\Pi_{2,T}^2$ and estimate them separately. For the term $\Pi_{2,T}^1$ we have
\begin{align}\label{inequ:xt21}
    \Pi_{2,T}^1& = \Expect\bigg[\Big(\sup_{\tau_{1,T}\le n<\tau_{3,T}}g(\theta_n)\Big)\bigg]+\Expect\bigg[\Big(\sup_{\tau_{3,T}\le n<\tau_{4,T}}g(\theta_n)\Big)\bigg] \notag \\
    & \mathop{\le}^{\text{\cref{fab_2}}} \Expect\bigg[\Big(\sup_{\tau_{1,T}\le n<\tau_{3,T}}g(\theta_n)\Big)\bigg]+\Delta_{0}\notag\\&=\Expect[g(\theta_{\tau_{1,T}})]+\Expect\bigg[\Big(\sup_{\tau_{1,T}\le n<\tau_{3,T}}(g(\theta_n)-g(\theta_{\tau_{1,T}}))\Big)\bigg]+\Delta_{0}\notag\\
&=\Expect[\I_{[\tau_{1}=1]}g(\theta_{\tau_{1}})]+\Expect[\I_{[\tau_{1}>1]}g(\theta_{\tau_{1}})]+\Expect\bigg[\Big(\sup_{\tau_{1,T}\le n<\tau_{3,T}}(g(\theta_n)-g(\theta_{\tau_{1,T}}))\Big)\bigg]+\Delta_{0}\notag\\&\mathop{\le}^{(a)}g(\theta_{1})+\Big(\Delta_{0}+\alpha_{0}\sqrt{2 L\Delta_{0}}+\frac{ L\alpha_{0}^{2}}{2}\Big)+\Expect\bigg[\Big(\sup_{\tau_{1,T}\le n<\tau_{3,T}}(g(\theta_n)-g(\theta_{\tau_{1,T}}))\Big)\bigg] + \Delta_{0}\notag\\&\mathop{\le}^{(b)}g(\theta_{1})+2\Delta_{0}+\alpha_{0}\sqrt{2 L\Delta_{0}}+\frac{ L\alpha_{0}^{2}}{2} + C_{\Pi,1} \Expect\Bigg[\sum_{n=\tau_{1,T}}^{\tau_{3,T}-1} \zeta(n)\Bigg],
\end{align}
where $C_{\Pi,1}$ is a constant and is defined in \cref{inqu:diff:g:tau1}.
For $(a)$ of \cref{inequ:xt21}, we follow the fact that $\Expect\Big[\I_{[\tau_{1,T}>1]}g(\theta_{\tau_{1,T}-1})\Big]\le \Delta_{0}$ and  get that
\begin{align*}
\Expect[\I_{[\tau_{1}>1]}g(\theta_{\tau_{1,T}})]&=\Expect[\I_{[\tau_{1}>1]}g(\theta_{\tau_{1,T}-1})]+\Expect[\I_{[\tau_{1}>1]}g(\theta_{\tau_{1,T}})-g(\theta_{\tau_{1,T}-1})]\\&\mathop{\le}^{\text{\cref{inequ:g:adj}}} \Delta_{0}+\alpha_{0}\sqrt{2 L\Delta_{0}}+\frac{ L\alpha_{0}^{2}}{2}.
\end{align*}
For (b) we use the one-step iterative formula on $g$ 
   \begin{align}\label{inequ:one:g:adj}
       g(\theta_{n+1})-g(\theta_{n})&\le \nabla g(\theta_{n})^{\top}(\theta_{n+1}-\theta_{n})+\frac{ L}{2}\|\theta_{n+1}-\theta_{n}\|^{2} \notag \\&\le\frac{\alpha_{0}\|\nabla g(\theta_{n})\|\|\nabla g(\theta_{n},\xi_{n})\|}{\sqrt{S_{n}}}+\frac{ L\alpha_{0}^{2}}{2}\frac{\|\nabla g(\theta_{n},\xi_{n})\|^{2}}{S_{n}} \notag \\
       & \leq \frac{\alpha_{0}\|\nabla g(\theta_{n})\|}{\sqrt{S_{n-1}}} \|\nabla g(\theta_{n},\xi_{n})\|+\frac{ L\alpha_{0}^{2}}{2}\frac{\|\nabla g(\theta_{n},\xi_{n})\|^{2}}{\sqrt{S_0}\sqrt{S_{n-1}}},
   \end{align}
   which induces that (recall that $\zeta_n = \|\nabla g(\theta_{n},\xi_{n})\|^{2}/\sqrt{S_{n-1}}$)
\begin{align}\label{inqu:diff:g:tau1}
&\quad \Expect\bigg[\Big(\sup_{\tau_{1,T}\le n<\tau_{3,T}}(g(\theta_n)-g(\theta_{\tau_{1,T}}))\Big)\bigg]
\le  \Expect\Bigg[\sum_{n=\tau_{1,T}}^{\tau_{3,T}-1}|g(\theta_{n+1})-g(\theta_{n})|\Bigg] \notag\\& \le \Expect\Bigg[\sum_{n=\tau_{1,T}}^{\tau_{3,T}-1}\frac{\alpha_{0}\|\nabla g(\theta_{n})\|\cdot\|\nabla g(\theta_{n},\xi_{n})\|}{\sqrt{S_{n-1}}}\Bigg]+\Expect\Bigg[\sum_{n=\tau_{1,T}}^{\tau_{3,T}-1}\frac{ L\alpha^{2}_{0}\|\nabla g(\theta_{n},\xi_{n})\|^{2}}{2\sqrt{S_0}\sqrt{S_{n-1}}}\Bigg] \notag \\&\mathop{=}^{(a)}\Expect\left[\sum_{n=\tau_{1,T}}^{\tau_{3,T}-1}\frac{\alpha_{0}\|\nabla g(\theta_{n})\|}{\sqrt{S_{n}}}\E\left(\|\nabla g(\theta_{n},\xi_{n})\| \mid \mathscr{F}_{n-1}\right)+\frac{ L\alpha_{0}^{2}}{2\sqrt{S_{0}}}\sum_{n=\tau_{1,T}}^{\tau_{3,T}-1}\frac{\E \left(\|\nabla g(\theta_{n},\xi_{n})\|^{2} \mid \mathscr{F}_{n-1}\right) }{\sqrt{S_{n-1}}}\right]
\notag 
\\&\mathop{\le}^{(*)}\bigg(\alpha_{0}\Big(\sqrt{\sigma_{0}}+\sqrt{\frac{\sigma_{1}}{\eta}}\Big)+\frac{ L\alpha_{0}^{2}}{2\sqrt{S_{0}}}\Big(\sigma_{0}+\frac{\sigma_{1}}{\eta}\Big)\bigg)\Expect\Bigg[\sum_{n=\tau_{1,T}}^{\tau_{3,T}-1}\zeta(n)\Bigg] := C_{\Pi,1}\Expect\Bigg[\sum_{n=\tau_{1,T}}^{\tau_{3,T}-1}\zeta(n)\Bigg],
\end{align}
where (a) uses \cref{sum:expect:ab}.
If $\tau_{1,T} > \tau_{3,T} - 1$, inequality $(*)$ trivially holds since $\sum_{n=\tau_{1,T}}^{\tau_{3,T}-1}\cdot=0$. Moving forward we will exclusively examine the case $\tau_{1,T} \le \tau_{3,T} - 1$. By the definition of $\tau_t$, we have $\hat{g}(\theta_n) > \Delta_0 \geq \hat{C}_g$ for any $n \in [\tau_{1,T}, \tau_{3,T})$. Consequently, upon applying \cref{pro_0}, we deduce that $\|\nabla g(\theta_{n})\|^{2} > \eta$ for any $n \in [\tau_{1,T}, \tau_{3,T})$. Combined with the affine noise variance condition, we further achieve the subsequent inequalities that for any $n\in[\tau_{1,T},\tau_{3,T})$
\begin{align}\label{asddsaasd}
    \Expect[\|\nabla g(\theta_{n},\xi_{n})\|^{2}|\mathscr{F}_{n-1}]&\le {\sigma_{0}}\|\nabla g(\theta_{n})\|^{2}+{\sigma_{1}} 
     < \Big({\sigma_{0}}+{\frac{\sigma_{1}}{{\eta}}}\Big)\cdot\|\nabla g(\theta_{n})\|^{2}
\end{align}
and
\begin{align}\label{asddsa}
    &\quad \Expect[\|\nabla g(\theta_{n},\xi_{n})\||\mathscr{F}_{n-1}]  \leq \left(\Expect[\|\nabla g(\theta_{n},\xi_{n})\|^2|\mathscr{F}_{n-1}) \right]^{1/2} \leq \Big({\sigma_{0}}\|\nabla g(\theta_{n})\|^{2}+{\sigma_{1}}\Big)^{1/2} \notag \\
    &\le \sqrt{\sigma_{0}}\|\nabla g(\theta_{n})\|+\sqrt{\sigma_{1}} <\Big(\sqrt{\sigma_{0}}+\sqrt{\frac{\sigma_{1}}{{\eta}}}\Big)\cdot\|\nabla g(\theta_{n})\|.
\end{align}
Next, we turn to estimate $\Pi_{2,T}^2$.
\begin{align}\label{qaswert}
   \Pi_{2,T}^2&={\Expect\bigg[\sup_{i\ge 2}\Big(\sup_{\tau_{3i-2,T}\le n<\tau_{3i+1,T}}g(\theta_n)\Big)\bigg]}\notag\\&\le \Expect\bigg[\sup_{i\ge 2}\Big(\sup_{\tau_{3i-2,T}\le n<\tau_{3i-1,T}}g(\theta_n)\Big)\bigg]+\Expect\bigg[\sup_{i\ge 2}\Big(\sup_{\tau_{3i-1,T}\le n<\tau_{3i,T}}g(\theta_n)\Big)\bigg]\notag\\&\quad +\Expect\bigg[\sup_{i\ge 2}\Big(\sup_{\tau_{3i,T}\le n<\tau_{3i+1,T}}g(\theta_n)\Big)\bigg]\notag\\&\mathop{\le}^{(a)} 2\Delta_{0}+\Expect\bigg[\sup_{i\ge 2}\Big(\sup_{\tau_{3i-1,T}\le n<\tau_{3i,T}}g(\theta_n)\Big)\bigg]+\Delta_{0}\notag\\&\le 3\Delta_{0}+\Expect\Big[\sup_{n=\tau_{3i-1,T}}{g}(\theta_{n})\Big]+\Expect\Big[\sup_{i\ge 2}\sup_{\tau_{3i-1,T}\le n\le \tau_{3i,T}}(g(\theta_{n})-g(\theta_{\tau_{3i-1,T}}))\Big]\notag\\&\mathop{\le}^{(b)}3\Delta_{0}+\Big(2\Delta_0+2\alpha_{0}\sqrt{ L\Delta_0}+\frac{ L\alpha_{0}^{2}}{2}\Big) + C_{\Pi,1}\E\Bigg[\sum_{i=2}^{+\infty}\sum_{\tau_{3i-1,T}}^{\tau_{3i,T}-1}\zeta(n)\Bigg],
\end{align}
where $(a)$ follows from \cref{fab_2} and \cref{fabulous}. To derive $(b)$, we first use the following estimation of $g(\theta_n)$ at the stopping time $\tau_{3i-1,T}$  
\begin{align*}
\sup_{n=\tau_{3i-1,T}}{g}(\theta_{n})&=\sup_{n=\tau_{3i-1,T}}{g}(\theta_{n-1})+\sup_{n=\tau_{3i-1,T}}({g}(\theta_{n})-g(\theta_{n-1}))\\&\mathop{\le}^{\text{\cref{inequ:g:adj}}} 2\Delta_{0}+2\alpha_{0}\sqrt{ L\Delta_{0}}+\frac{ L\alpha_{0}^{2}}{2}.
\end{align*}
Then, since the objective $g(\theta_n)$ in the interval $ n \in [\tau_{3i-1,T},\tau_{3i,T})$ has similar properties as the interval $[\tau_{1,T},\tau_{3,T})$, we follow the same procedure as \cref{inqu:diff:g:tau1}  to estimate the supremum of $g(\theta_{n})-g(\theta_{\tau_{3i-1,T}})$ on the interval $ n \in [\tau_{3i-1,T}, \tau_{3i,T})$, it achieves that 
\begin{align}\label{inqu:diff:g:tau2}
&\quad \E \left[\sup_{i\ge 2}\sup_{\tau_{3i-1,T}\le n\le \tau_{3i,T}}(g(\theta_{n})-g(\theta_{\tau_{3i-1,T}}))\right] \leq \E \left[\sum_{i=2}^{+\infty}\sup_{\tau_{3i-1,T}\le n\le \tau_{3i,T}}(g(\theta_{n})-g(\theta_{\tau_{3i-1,T}}))\right]  \notag \\
& \leq \bigg(\alpha_{0}\Big(\sqrt{\sigma_{0}}+\sqrt{\frac{\sigma_{1}}{\eta}}\Big)+\frac{ L\alpha_{0}^{2}}{2\sqrt{S_{0}}}\Big(\sigma_{0}+\frac{\sigma_{1}}{\eta}\Big)\bigg)\Expect\Bigg[\sum_{i=2}^{+\infty}\sum_{n=\tau_{3i-1,T}}^{\tau_{3i,T}-1}\zeta(n)\Bigg].
\end{align}
By substituting the estimations of $\Pi_{2,T}^1$ and $\Pi_{2,T}^2$ from \cref{inequ:xt21} and \cref{qaswert} respectively into \cref{inequ:xt2}, we achieve the estimation for $\Pi_{2,T}$. Then, substituting the result for $\Pi_{2,T}$ into \cref{start} gives
\begin{align}\label{wxm_200}
\Expect\Big[\sup_{1\le n<T}g(\theta_{n})\Big]\le C_{\Pi,0}+C_{\Pi, 1}\underbrace{\Expect\Bigg[\sum_{n=\tau_{1,T}}^{\tau_{3,T}-1}\zeta(n)+\sum_{i=2}^{+\infty}\sum_{\tau_{3i-1,T}}^{\tau_{3i,T}-1}\zeta(n)\Bigg]}_{\Pi_{3,T}},
\end{align}
where
\begin{align}\label{constant_Pi}
C_{\Pi,0}&=g(\theta_{1})+6\Delta_{0}+5\alpha_{0}\sqrt{ L\Delta_{0}}+\frac{3 L\alpha_{0}^{2}}{2}, C_{\Pi,1}=\alpha_{0}\Big(\sqrt{\sigma_{0}}+\sqrt{\frac{\sigma_{1}}{\eta}}\Big)+\frac{ L\alpha_{0}^{2}}{2\sqrt{S_{0}}}\Big(\sigma_{0}+\frac{\sigma_{1}}{\eta}\Big).
\end{align}
Next, we turn to find an upper bound for $\Pi_{3,T}$ which is independent of $T$. Recall the sufficient decrease inequality in \cref{sufficient:lem}
\begin{align*}
\hat{g}(\theta_{n+1})-\hat{g}(\theta_{n}) & \le -\frac{\alpha_{0}}{4}\zeta_{n}+C_{\Gamma,1}\cdot \Gamma_n  +C_{\Gamma,2}\frac{\Gamma_{n}}{\sqrt{S_{n}}} + \alpha_0 \hat{X}_{n}.
\end{align*}
First, we estimate the first term of $\Pi_{3,T}$. Telescoping the above inequality over $n$ from the interval \(I_{1,\tau}:=[\tau_{1,T},\tau_{3,T}-1]\)  gives
\begin{align*}%\label{inequ:stab:2'}
\frac{\alpha_{0}}{4}\sum_{n \in I_{1,\tau}}\zeta(n)&\le \hat{g}(\theta_{\tau_{1,T}})-\hat{g}(\theta_{\tau_{3,T}})+C_{\Gamma,1}\sum_{n \in I_{1,\tau}} \Gamma_n+C_{\Gamma,2}\sum_{n \in I_{1,\tau}}\frac{\Gamma_{n}}{\sqrt{S_{n}}}+\alpha_0 \sum_{n \in I_{1,\tau}}\hat{X}_{n}.
\end{align*} 
Taking the expectation on both sides of the above inequality, we have
\begin{align*}
\frac{\alpha_{0}}{4}\Expect\left[\sum_{n \in I_{1,\tau}}\zeta(n)\right]&\le\Expect\big[\hat{g}(\theta_{\tau_{1,T}})\big]+C_{\Gamma,1}\Expect\bigg[\sum_{n \in I_{1,\tau}} \Gamma_n \bigg]+C_{\Gamma,2}\Expect\Bigg[\sum_{n \in I_{1,\tau}}\frac{\Gamma_{n}}{\sqrt{S_{n}}}\Bigg]+\alpha_0\Expect\Bigg[\sum_{n \in I_{1,\tau}}\hat{X}_{n}\Bigg]\notag\\&\mathop{\le}^{(a)} \Expect\big[\hat{g}(\theta_{\tau_{1,T}})\big]+C_{\Gamma,1}\Expect\bigg[\sum_{n \in I_{1,\tau}} \Expect[\Gamma_n|\mathscr{F}_{n-1}]\bigg]+C_{\Gamma,2}\Expect\Bigg[\sum_{n \in I_{1,\tau}}\frac{\Gamma_{n}}{\sqrt{S_{n}}}\Bigg]+0,
\end{align*}
where for (a), we use \text{\emph{Doob's Stopped} theorem} (see \cref{vital1}) since the stopping times $\tau_{1,T}\le \tau_{3,T}-1$ and $\hat{X}_n$ is a martingale sequence. For the first term of the RHS of the above inequality,
\begin{align*}
\Expect\big[\hat{g}(\theta_{\tau_{1,T}})\big]&=\Expect\big[\I_{[\tau_{1}=1]}\hat{g}(\theta_{1})\big]+\Expect\big[\I_{\tau_{1}>1}\hat{g}(\theta_{\tau_{1,T}})\big] \\&\le \hat{g}(\theta_{1})+ \Expect\big[\I_{\tau_{1}>1}\hat{g}(\theta_{\tau_{1,T}-1})\big]+\Expect\big[\I_{\tau_{1}>1}(\hat{g}(\theta_{\tau_{1,T}})-\hat{g}(\theta_{\tau_{1,T}-1}))\big]\\&\mathop{\le}^{\text{\cref{lem:adj:ghat}}} \hat{g}(\theta_{1})+\Delta_{0}+h(\Delta_{0})<\hat{g}(\theta_{1})+\frac{3\Delta_{0}}{2}.
\end{align*}
We thus conclude that 
\begin{align}\label{inequ:stat:1'}
\frac{\alpha_{0}}{4}\Expect\left[\sum_{n \in I_{\tau,1}}\zeta(n)\right]&\le\hat{g}(\theta_{1})+\frac{3\Delta_{0}}{2}+C_{\Gamma,1}\Expect\bigg[\sum_{n \in I_{\tau,i}} \Expect[\Gamma_n|\mathscr{F}_{n-1}] \bigg]+C_{\Gamma,2}\Expect\Bigg[\sum_{n \in I_{\tau,i}}\frac{\Gamma_{n}}{\sqrt{S_{n}}}\Bigg].
\end{align}
For the second term of $\Pi_{3,T}$, we telescope the sufficient decrease inequality in \cref{sufficient:lem} over $n$ from the interval \( I_{i,\tau}^{'}:= [\tau_{3i-1,T}, \tau_{3i,T}-1 ]\ (\forall\ i\ge 2)\)
\begin{align}\label{inequ:stab:2}
\frac{\alpha_{0}}{4}\sum_{n \in I_{i,\tau}^{'}}\zeta(n)&\le \hat{g}(\theta_{\tau_{3i-1,T}})-\hat{g}(\theta_{\tau_{3i,T}})+C_{\Gamma,1}\sum_{n \in I_{i,\tau}^{'}} \Gamma_n+C_{\Gamma,2}\sum_{n \in I_{i,\tau}^{'}}\frac{\Gamma_{n}}{\sqrt{S_{n}}}+\alpha_0 \sum_{n \in I_{i,\tau}^{'}}\hat{X}_{n}.
\end{align} 
Recalling the definition of the stopping time $\tau_t$, we know that \(\tau_{3i,T} \geq \tau_{3i-1,T}\) always holds. In particular, \(\tau_{3i,T} = \tau_{3i-1,T}\) implies that \(\tau_{3i,T}-1 < \tau_{3i-1,T}\). Since $\sum_{n=a}^b(\cdot) = 0$ for $b  <  a$, we have  $\sum_{n=\tau_{3i-1,T}}^{\tau_{3i,T}-1}(\cdot) = 0 $ and $\hat{g}(\theta_{\tau_{3i,T}}) = \hat{g}(\theta_{\tau_{3i-1,T}})$, then LHS and RHS of \cref{inequ:stab:2} are both zero and \cref{inequ:stab:2} holds. Taking the expectation on both sides and noting the equation of \cref{sum:expect:ab} gives
\begin{align}\label{inequ:stat:1}
\frac{\alpha_{0}}{4}\Expect\left[\sum_{n \in I_{i,\tau}^{'}}\zeta(n)\right]&\le \Expect\big[\hat{g}(\theta_{\tau_{3i-1,T}})-\hat{g}(\theta_{\tau_{3i,T}})\big]+C_{\Gamma,1}\Expect\bigg[\sum_{n \in I_{i,\tau}^{'}} \Expect[\Gamma_n|\mathscr{F}_{n-1}] \bigg]\notag\\&+C_{\Gamma,2}\Expect\Bigg[\sum_{n \in I_{i,\tau}^{'}}\frac{\Gamma_{n}}{\sqrt{S_{n}}}\Bigg]+0.
\end{align}
If \(\tau_{3i-1, T} <  \tau_{3i, T}\),  for any $n \in I_{i,\tau}^{'} = [\tau_{3i-1,T}, \tau_{3i,T}-1]$, by applying \cref{lem:adj:ghat} we have
\begin{align*}
\hat{g}(\theta_{\tau_{3i-1, T}}) - \hat{g}(\theta_{\tau_{3i,T}})&<\hat{g}(\theta_{\tau_{3i-1,T}}) <  \hat{g}(\theta_{\tau_{3i-1,T}-1})+h(\hat{g}(\theta_{\tau_{3i-1,T}-1})).
\end{align*}
Based on the properties of the stopping time \(\tau_{3i-1},\) we have $\hat{g}(\theta_{\tau_{3i-1,T}-1}) \leq 2\Delta_{0}.$ Based on the above inequality, we further estimate the first term of  \cref{inequ:stat:1} and achieve that
\begin{align}\label{wxm_90}
\frac{\alpha_{0}}{4}\Expect\Bigg[\sum_{n=I_{i,\tau}^{'}}\zeta(n)\Bigg]\notag&\le C_{\Delta_{0}}\Expect\big[\I_{\{\tau_{3i-1,T}<\tau_{3i,T} \}}\big]+C_{\Gamma,1}\Expect\bigg[\sum_{n=I_{i,\tau}^{'}} \Expect[\Gamma_n|\mathscr{F}_{n-1}] \bigg] \\&+C_{\Gamma,2}\Expect\Bigg[\sum_{n=I_{i,\tau}^{'}}\frac{\Gamma_{n}}{\sqrt{S_{n}}}\Bigg]  ,
\end{align}
where
\begin{align}\label{inequ:C:Delta0}
C_{\Delta_{0}}: =2\Delta_{0}+ \sqrt{2 L}\left(1 + \frac{\sigma_0 L}{2\sqrt{S_0}}\right)\alpha_0\sqrt{2\Delta_0} + \left(1  + \frac{\sigma_0\alpha_0 L}{2\sqrt{S_0}}\right)\frac{ L\alpha_0^2}{2}.
\end{align}
Telescoping \cref{wxm_90} over $i$ from \(2\) to \(+\infty\) to estimate the second part of $\Pi_{3,T}$, we have
\begin{align}\label{pqwe_1}
\frac{\alpha_{0}}{4}\Expect\left[\sum_{i=2}^{+\infty}\sum_{n=I_{i,\tau}^{'}}\zeta(n)\right]\le& C_{\Delta_{0}}\cdot\sum_{i=2}^{+\infty}\Expect\big[\I_{\tau_{3i-1,T}<\tau_{3i,T}}\big]+C_{\Gamma,1}\sum_{i=2}^{+\infty}\Expect\bigg[\sum_{n=I_{i,\tau}^{'}} \Expect[\Gamma_n|\mathscr{F}_{n-1}]\bigg] \notag \\&+C_{\Gamma,2}\sum_{i=2}^{+\infty}\Expect\Bigg[\sum_{n=I_{i,\tau}^{'}}\frac{\Gamma_{n}}{\sqrt{S_{n}}}\Bigg].
\end{align}
Note that the stopping time \(\tau_{t}\) is truncated for any finite time \(T\). For a specific $T$, the sum \(\sum_{i=2}^{+\infty}\) has only finite non-zero terms, thus we can interchange the order of summation and expectation \(\mathbb{E}\left(\sum_{i=2}^{+\infty} (\cdot) \right) = \sum_{i=2}^{+\infty} \left(\mathbb{E}(\cdot) \right).\) Substituting \cref{pqwe_1} and \cref{inequ:stat:1'} into \cref{wxm_200} gives
\begin{align}\label{wxm_110}
&\Expect\Big[\sup_{1\le n< T}g(\theta_{n})\Big]\notag \\
\le&\ \overline{C}_{\Pi,0}+ C_{\Pi,1}C_{\Delta_{0}}\cdot\sum_{i=2}^{+\infty}\underbrace{\Expect\big[\I_{\tau_{3i-1,T}<\tau_{3i,T}}\big]}_{\Psi_{i,1}}+C_{\Pi,1}C_{\Gamma,1}\underbrace{\Expect\left[\bigg(\sum_{I_{1,\tau}}+\sum_{i=2}^{+\infty}\sum_{n=I_{i,\tau}^{'}} \bigg)\Expect[\Gamma_n|\mathscr{F}_{n-1}]\right]}_{\Psi_{2}} \notag \\&\quad +C_{\Pi,1}C_{\Gamma,2}\underbrace{\Expect\Bigg[\bigg(\sum_{n=I_{1,\tau}}+\sum_{i=2}^{+\infty}\sum_{n=I_{i,\tau}^{'}} \bigg)\frac{\Gamma_{n}}{\sqrt{S_{n}}}\Bigg]}_{\Psi_{3}},
\end{align}
where $\overline{C}_{\Pi,0}:=\hat{g}(\theta_{1})+\frac{3\Delta_{0}}{2}+C_{\Pi,0}.$ 
\end{proof}

\begin{proof}(of \cref{lem_su})
Due to \cref{sufficient:lem}, we know
\begin{align}\label{jh_20}
\hat{g}(\theta_{n+1})-\hat{g}(\theta_{n}) & \le -\frac{\alpha_{0}}{4}\zeta(n)+C_{\Gamma, 1}\cdot \Gamma_n  + C_{\Gamma, 2}\frac{\Gamma_{n}}{\sqrt{S_{n}}} + \alpha_0 \hat{X}_{n},
\end{align}
Then we define an auxiliary variable $ y_{n} := \frac{1}{\sqrt{S_{n-1}}}$. Multiplying both sides of \cref{jh_20} by this auxiliary variable, we obtain
\begin{align*}
y_{n}\hat{g}(\theta_{n+1})-y_{n}\hat{g}(\theta_{n}) & \le  -\frac{\alpha_{0}}{4}y_{n}\zeta(n)+C_{\Gamma, 1}\cdot y_{n}\Gamma_n  + C_{\Gamma, 2}y_{n}\frac{\Gamma_{n}}{\sqrt{S_{n}}} + \alpha_0 y_{n}\hat{X}_{n}\notag.
\end{align*}
By transposing the above inequality, and note that 
\(y_{n}g(\theta_{n+1}) - y_{n}g(\theta_{n}) = y_{n+1}g(\theta_{n+1}) - y_{n}g(\theta_{n}) + (y_{n} - y_{n+1})g(\theta_{n+1}),\) we obtain
\begin{equation}\nonumber\begin{aligned}
\frac{\alpha_{0}}{4}y_{n}\zeta(n)&\le \big(y_{n}\hat{g}(\theta_{n})-y_{n+1}\hat{g}(\theta_{n+1})\big)+{(y_{n+1}-y_{n})\hat{g}(\theta_{n+1})}+C_{\Gamma, 1}\cdot y_{n}\Gamma_n\\&+C_{\Gamma,2}y_{n}\frac{\Gamma_{n}}{\sqrt{S_{n}}}+\alpha_0 y_{n}\hat{X}_{n}.
\end{aligned}\end{equation}
For any positive number \(T \ge 0\), we telescope the terms indexed by \(n\) from 1 to \(T\), and take the mathematical expectation, yielding
\begin{align}\label{ada_q_1}
\frac{\alpha_{0}}{4}\Expect\Bigg[\sum_{n=1}^{T}y_{n}\zeta_{n}\Bigg]&{\leq} y_{1}\hat{g}(\theta_{1})+\Expect\Bigg[\underbrace{\sum_{n=1}^{T}(y_{n+1}-y_{n})\hat{g}(\theta_{n+1})}_{\Theta_{1}}\Bigg]+C_{\Gamma,1}\cdot\underbrace{\sum_{n=1}^{T}y_{n}\Gamma_{n}}_{\Theta_{2}}+C_{\Gamma,2}\cdot\underbrace{\sum_{n=1}^{T}y_{n}\frac{\Gamma_{n}}{\sqrt{S_{n}}}}_{\Theta_{3}}+0.
\end{align}
Our objective is to prove that the RHS of the above inequality has an upper bound independent of $T.$ To this end, we bound $\Theta_{1}$, $\Theta_{2}$, and $\Theta_{3}$ separately. For $\Theta_{2}$, we have
\begin{align}\label{jh_02}
\Theta_{1}&=\sum_{n=1}^{T}(y_{n+1}-y_{n})\hat{g}(\theta_{n+1})=\sum_{n=1}^{T}\Big(\frac{1}{\sqrt{S_{n+1}}}-\frac{1}{\sqrt{S_{n}}}\Big)\hat{g}(\theta_{n+1})\le 0.
\end{align}
Then for term $\Theta_{2}$ in \cref{jh_02}, we have
\begin{align}\label{jh_03}
\Theta_{2}=\sum_{n=1}^{T}y_{n}\Gamma_{n}&\le \sum_{n=1}^{T}\frac{\Gamma_{n}}{\sqrt{S_{n-1}}}=\sum_{n=1}^{T}y_{n}\Gamma_{n}\le \sum_{n=1}^{T}\frac{\Gamma_{n}}{\sqrt{S_{n}}}+\sum_{n=1}^{T}\Gamma_{n}\bigg(\frac{1}{\sqrt{S_{n-1}}}-\frac{1}{\sqrt{S_{n}}}\bigg)\notag\\&\mathop{\le}^{(a)} \int_{S_{0}}^{+\infty}\frac{1}{x^{\frac{3}{2}}}\text{d}x+\frac{1}{\sqrt{S_{0}}}=\frac{3}{\sqrt{S_{0}}}.
\end{align}
In step $(a)$, we apply the series-integral inequality and the fact that \(\|\nabla g(\theta_{n})\|/\sqrt{S_{n}}\le 1.\)
Finally for term \(\Theta_{3}\), we only need to use the series-integral inequality to get
\begin{align}\label{jh_04}
\Theta_{3}=\sum_{n=1}^{T}y_{n}\frac{\Gamma_{n}}{\sqrt{S_{n}}}\le \frac{1}{\sqrt{S_{0}}}\int_{S_{0}}^{+\infty}\le \frac{2}{S_{0}}.
\end{align}
Subsequently, we substitute the estimates for $\Theta_{1}$, $\Theta_{2},$ and $\Theta_{3}$ from \cref{jh_02}, \cref{jh_03}, and \cref{jh_04} back into \cref{ada_q_1}, resulting in the following inequality
\begin{align*}
\frac{\alpha_{0}}{4}\Expect\Bigg[\sum_{n=1}^{T}y_{n}\zeta_{n}\Bigg]&{\leq} y_{1}\hat{g}(\theta_{1})+0+\frac{3C_{\Gamma,1}}{\sqrt{S_{0}}}+\frac{2C_{\Gamma,2}}{S_{0}}<+\infty.
\end{align*}
The right-hand side of the above inequality is independent of \(T\). Therefore, by applying the \emph{ Lebesgue's monotone convergence} theorem, we obtain
\begin{align*}
\frac{\alpha_{0}}{4}\Expect\Bigg[\sum_{n=1}^{+\infty}y_{n}\zeta_{n}\Bigg]&{\leq} y_{1}\hat{g}(\theta_{1})+\frac{3C_{\Gamma,1}}{\sqrt{S_{0}}}+\frac{2C_{\Gamma,2}}{S_{0}}<+\infty.
\end{align*}
Then,
\begin{align*}
\Expect\Bigg[\sum_{n=1}^{+\infty}\frac{\|\nabla g(\theta_{n})\|^{2}}{S_{n-1}}\Bigg]\le M:=\hat{g}(\theta_{1})+\frac{3C_{\Gamma,1}}{\sqrt{S_{0}}}+\frac{2C_{\Gamma,2}}{S_{0}}<+\infty,
\end{align*}
where $M$ is a constant. For any $\nu > 0$,
combined with the affine noise variance condition, we further achieve the subsequent inequality
\begin{align}
    \I_{\|\nabla g(\theta_{n})\|^{2}> \nu}\Expect[\|\nabla g(\theta_{n},\xi_{n})\|^{2}|\mathscr{F}_{n-1}]&\le \I_{\|\nabla g(\theta_{n})\|^{2}>\nu}( {\sigma_{0}}\|\nabla g(\theta_{n})\|^{2}+{\sigma_{1}})\notag\\&= \I_{\|\nabla g(\theta_{n})\|^{2}>\nu}\Big({\sigma_{0}}+\frac{{\sigma_{1}}}{\|\nabla g(\theta_{n})\|^{2}}\Big)\|\nabla g(\theta_{n})\|^{2} \notag \\
    & <  \I_{\|\nabla g(\theta_{n})\|^{2}>\nu}\Big({\sigma_{0}}+{\frac{\sigma_{1}}{{\nu}}}\Big)\cdot\|\nabla g(\theta_{n})\|^{2}\notag\\&\le \Big({\sigma_{0}}+{\frac{\sigma_{1}}{\nu}}\Big)\cdot\|\nabla g(\theta_{n})\|^{2}.
\end{align}
Then, we obtain
\begin{align*}
    \Expect\Bigg[ \sum_{n=1}^{+\infty}\I_{\|\nabla g(\theta_{n})\|^{2}>\nu}\frac{\|\nabla g(\theta_{n},\xi_{n})\|^{2}}{S_{n}}\Bigg]&\le\Expect\Bigg[ \sum_{n=1}^{+\infty}\I_{\|\nabla g(\theta_{n})\|^{2}>\nu}\frac{\|\nabla g(\theta_{n},\xi_{n})\|^{2}}{S_{n-1}}\Bigg]\notag\\&\le\Big({\sigma_{0}}+{\frac{\sigma_{1}}{\nu}}\Big)\cdot\Expect\Bigg[\sum_{n=1}^{+\infty}\frac{\|\nabla g(\theta_{n})\|^{2}}{S_{n-1}}\Bigg] \\&<\Big({\sigma_{0}}+{\frac{\sigma_{1}}{{\nu}}}\Big)\cdot M.
\end{align*}
This completes the proof. 
\end{proof}

\begin{proof}(of \cref{lem:psi:i1})
We start by observing the inequality
\[\Psi_{i,1}=\Expect[\I_{\tau_{3i-1, T} < \tau_{3i, T}}]=\pro(\tau_{3i-1, T} < \tau_{3i, T}).\]
What we need to consider is the probability of the event $\tau_{3i-1, T} < \tau_{3i, T}$ occurring. In the case we consider $\tau_{3i-1, T} < \tau_{3i, T}$ which implies that $\hat{g}(\theta_{3i-1, T}) \geq 2\Delta_0$. On the other hand, according to the definition of the stopping time $\tau_{3i-2,T}$, we have $\hat{g}(\tau_{{3i-2, T}-1}) \leq \Delta_0$. Then
\begin{align*}
\hat{g}(\theta_{\tau_{3i-2,T}}) <  \hat{g}(\theta_{\tau_{3i-2,T}-1})+h(\hat{g}(\theta_{\tau_{3i-2,T}-1})) \leq \Delta_0 + h(\Delta_0) < \frac{3}{2}\Delta_0.
\end{align*}
Since \(\Delta_0 > C_{0}\), we know that \( h(\Delta_0) < \frac{1}{2}\Delta_0\) by \cref{lem:adj:ghat}.  Then, by \cref{sufficient:lem}),
\begin{align*}
\frac{\Delta_0}{2} &=2\Delta_0-\frac{3\Delta_0}{2}\le \hat{g}(\theta_{\tau_{3i-1,T}})-\hat{g}(\theta_{\tau_{3i-2,T}})\le \sum_{n=\tau_{3i-2,T}}^{\tau_{3i-1,T}-1}(\hat{g}(\theta_{n+1})-\hat{g}(\theta_{n}))\\&\le C_{\Gamma,1}\cdot \sum_{n=\tau_{3i-2,T}}^{\tau_{3i-1,T}-1}\Gamma_n  +C_{\Gamma,2}\sum_{n=\tau_{3i-2,T}}^{\tau_{3i-1,T}-1}\frac{\Gamma_{n}}{\sqrt{S_{n}}} + \alpha_0\Bigg| \sum_{n=\tau_{3i-2,T}}^{\tau_{3i-1,T}-1}\hat{X}_{n}\Bigg|\\& \mathop{\le}^{\text{Young's inequality}} C_{\Gamma,1}\cdot \sum_{n=\tau_{3i-2,T}}^{\tau_{3i-1,T}-1}\Gamma_n  +C_{\Gamma,2}\sum_{n=\tau_{3i-2,T}}^{\tau_{3i-1,T}-1}\frac{\Gamma_{n}}{\sqrt{S_{n}}}+\frac{\alpha_0^{2}}{\Delta_0}\Bigg( \sum_{n=\tau_{3i-2,T}}^{\tau_{3i-1,T}-1}\hat{X}_{n}\Bigg)^{2}  + \frac{\Delta_0}{4},
\end{align*} 
which further induces that 
\begin{equation}\label{power_00}\begin{aligned}
\frac{\Delta_0}{4}&\le  C_{\Gamma,1}\cdot \sum_{n=\tau_{3i-2,T}}^{\tau_{3i-1,T}-1}\Gamma_n  +C_{\Gamma,2}\sum_{n=\tau_{3i-2,T}}^{\tau_{3i-1,T}-1}\frac{\Gamma_{n}}{\sqrt{S_{n}}}  + \frac{\alpha_0^{2}}{\Delta_0}\Bigg( \sum_{n=\tau_{3i-2,T}}^{\tau_{3i-1,T}-1}\hat{X}_{n}\Bigg)^{2}.
\end{aligned}\end{equation}
Based on the above analysis, we can obtain the following sequence of event inclusions
\begin{align*}
\{\tau_{3i-1,T}<\tau_{3i,T}\}&\subset \{\hat{g}(\theta_{3i-1,T})>2\Delta_{0}\}\subset\Big\{\frac{\Delta_{0}}{2}\le \hat{g}(\theta_{\tau_{3i-1,T}})-\hat{g}(\theta_{\tau_{3i-2,T}})\Big\}\\&\subset\{\text{\cref{power_00} holds}\}.
\end{align*}
Thus, we have the following probability inequality
\begin{align*}
\Expect[\I_{\tau_{3i-1,T}<\tau_{3i,T}}]=\pro(\tau_{3i-1,T}<\tau_{3i,T})\le \pro(\text{\cref{power_00} holds}).
\end{align*}
Then, according to \emph{Markov's inequality}, we obtain
\begin{align*}
\Pro(\text{\cref{power_00} holds})&\le \frac{4}{\Delta_0}C_{\Gamma,1}\cdot \Expect\Bigg[\sum_{n=\tau_{3i-2,T}}^{\tau_{3i-1,T}-1}\Gamma_n\Bigg]  \\&+ \frac{4C_{\Gamma,2}}{\Delta_0}\Expect\Bigg[\sum_{n=\tau_{3i-2,T}}^{\tau_{3i-1,T}-1}\frac{\Gamma_{n}}{\sqrt{S_{n}}}\Bigg]  + \frac{4\alpha_0^{2}}{\Delta_0^{2}}\Expect\Bigg[\sum_{n=\tau_{3i-2,T}}^{\tau_{3i-1,T}-1}\hat{X}_{n}\Bigg]^{2}\\&\mathop{=}^{\text{\cref{sum:expect:ab}}} \frac{4C_{\Gamma,1}}{\Delta_0}\cdot \Expect\Bigg[\sum_{n=\tau_{3i-2,T}}^{\tau_{3i-1,T}-1}\Expect[\Gamma_n|\mathscr{F}_{n-1}]\Bigg]  + \frac{4C_{\Gamma,2}}{\Delta_0}\Expect\Bigg[\sum_{n=\tau_{3i-2,T}}^{\tau_{3i-1,T}-1}\frac{\Gamma_{n}}{\sqrt{S_{n}}}\Bigg]\\&  + \frac{4\alpha_0^{2}}{\Delta_0^{2}}\Expect\Bigg[\sum_{n=\tau_{3i-2,T}}^{\tau_{3i-1,T}-1}\hat{X}_{n}^{2}\bigg].
\end{align*}
This completes the proof. 
\end{proof}

\subsection{Proofs of Lemmas in \cref{subsec:almost:sure}}

\begin{proof}(of \cref{step size})
Firstly, when $\lim_{n\rightarrow+\infty}S_{n} < +\infty,$ we clearly have \[\sum_{n=1}^{+\infty}\frac{1}{\sqrt{S_{n}}} = +\infty.\] We then only need to prove that this result also holds for the case $\lim_{n\rightarrow+\infty}S_{n} = +\infty.$ That is, we define the event $S$ \[\mathcal{S} := \left\{\sum_{n=1}^{+\infty}\frac{1}{\sqrt{S_{n}}} < +\infty,\ \text{and}\ \lim_{n\rightarrow+\infty}S_{n}=+\infty\right\}\] and desire to prove that \(\Pro(\mathcal{S}) = 0.\) 

According to the stability of $g(\theta_n)$ in \cref{stable}, the following result holds almost surely on the event $\mathcal{S}$.
\begin{equation}\label{wxm_03}\begin{aligned}
\sum_{n=1}^{+\infty}\frac{\|\nabla g(\theta_{n+1})\|^{2}}{\sqrt{S_{n}}} \mathop{\leq}^{\text{\cref{loss_bound}}} 2 L\Big(\sup_{n\ge 1}g(\theta_{n})\Big)\cdot\sum_{n=1}^{+\infty}\frac{1}{\sqrt{S_{n}}}<+\infty\ \text{a.s.}
\end{aligned}\end{equation}
On the other hand, by the affine noise variance condition  \( \Expect\big[\|\nabla g(\theta_{n+1};\xi_{n+1})\|^{2} \big| \mathscr{F}_{n}\big] \leq \sigma_0\|\nabla g(\theta_{n+1})\|^{2} + \sigma_{1}\), it induces that 
\begin{align}\label{inequ:sum:gradS}
\sum_{n=1}^{+\infty}\frac{\|\nabla g(\theta_{n+1})\|^{2}}{\sqrt{S_{n}}}&\ge \frac{1}{\sigma_{0}}\sum_{n=1}^{+\infty}\frac{\Expect[\|\nabla g(\theta_{n+1},\xi_{n+1})\|^{2}|\mathscr{F}_{n}]}{\sqrt{S_{n}}}-\sum_{n=1}^{+\infty}\frac{\sigma_{1}}{\sigma_{0}\sqrt{S_{n}}} \notag \\ &=\frac{1}{\sigma_{0}}\underbrace{\sum_{n=1}^{+\infty}\frac{\|\nabla g(\theta_{n+1},\xi_{n+1})\|^{2}}{\sqrt{S_{n}}}}_{\Xi_1}-\underbrace{\sum_{n=1}^{+\infty}\frac{\sigma_{1}}{\sigma_{0}\sqrt{S_{n}}}}_{\Xi_2} \notag \\&+\underbrace{\sum_{n=1}^{+\infty}\frac{\Expect[\|\nabla g(\theta_{n+1},\xi_{n+1})\|^{2}|\mathscr{F}_{n}]-\|\nabla g(\theta_{n+1},\xi_{n+1})\|^{2}}{\sqrt{S_{n}}}}_{\Xi_3}.
\end{align}
Next, we determine whether the RHS of \cref{inequ:sum:gradS} converges the event $\mathcal{S}.$ For the term $\Xi_1$, using the series-integral comparison test, the following result holds on the event $\mathcal{S}$:
\[\Xi_1 = \lim_{n \rightarrow \infty} \int_{S_{0}}^{S_{n}}\frac{1}{\sqrt{x}}\text{d}x = \lim_{n \rightarrow \infty}\sqrt{S_n} - \sqrt{S_0}=+\infty.\] The second term $\Xi_2$ clearly converges on $\mathcal{S}.$ Since the last term $\Xi_3$ is the sum of a martingale sequence, we only need to determine the convergence of the following series on the set $\mathcal{S}$
\begin{align*}
&\sum_{n=1}^{+\infty}\Expect\Bigg[\bigg|\frac{\|\nabla g(\theta_{n+1},\xi_{n+1})\|^{2}-\Expect[\|\nabla g(\theta_{n+1},\xi_{n+1})\|^{2}|\mathscr{F}_{n}]}{\sqrt{S_{n}}}\bigg| \mid \mathscr{F}_{n}\Bigg]\\ 
& \leq 2\sum_{n=1}^{+\infty}\Expect\Bigg[\frac{\|\nabla g(\theta_{n+1},\xi_{n+1})\|^{2}}{\sqrt{S_{n}}} \mid \mathscr{F}_{n}\Bigg]  \mathop{<}^{(a)}2(2 L\sigma_{0}\sup_{n\ge 1}g(\theta_{n})+\sigma_{1})\sum_{n=1}^{+\infty}\frac{1}{\sqrt{S_{n}}}<+\infty\  \,\, a.s.,
\end{align*}
where $(a)$ uses the affine noise variance condition $\Expect[\|\nabla g(\theta_{n},\xi_{n})\|^{2}|\mathscr{F}_{n-1}]\le \sigma_{0}\|\nabla g(\theta_{n})\|^{2}+\sigma_{1},$ and \cref{loss_bound} that $ \|\nabla g(\theta)\|^{2}\le 2 Lg(\theta) $ for $\forall\ \theta\in\mathbb{R}^{d}.$
We conclude that the last term $\Xi_3$ converges almost surely. Therefore, combining the above estimations for $\Xi_1, \Xi_2, \Xi_3$, we prove that the following relation holds on the event $\mathcal{S}$: \[\sum_{n=1}^{+\infty}\frac{\|\nabla g(\theta_{n+1})\|^{2}}{\sqrt{S_{n}}}=+\infty\ \text{a.s.}\] 
However, in \cref{wxm_03} we know that the series $\sum_{n=1}^{+\infty}\frac{\|\nabla g(\theta_{n+1})\|^{2}}{\sqrt{S_{n}}}$ converges almost surely on the event $\mathcal{S}$. Thus, we can claim that if and only if the event $\mathcal{S}$ is a set of measure zero, that is $\Pro(\mathcal{S}) = 0.$ We complete the proof. 
\end{proof}

%\subsection{Proof of }
%\label{lem:proof:su}

\section{Appendix: Proofs of Lemmas in \cref{sec:nonasympt}}
\begin{proof}(of \cref{lem8})
Recalling the sufficient decrease inequality in \cref{sufficient:lem}, we have
\begin{align*}
\hat{g}(\theta_{n+1})-\hat{g}(\theta_{n}) & \le -\frac{\alpha_{0}}{4}\zeta(n)+C_{\Gamma,1}\cdot \Gamma_n  +C_{\Gamma,2}\frac{\Gamma_{n}}{\sqrt{S_{n}}} + \alpha_0 \hat{X}_{n}.
\end{align*}
We take the mathematical expectation
\begin{align}\label{12sqdr}
\Expect\big[\hat{g}(\theta_{n+1})\big]-\Expect\big[\hat{g}(\theta_{n})\big] &  \le -\frac{\alpha_{0}}{4}\E\left[\zeta(n)\right]+C_{\Gamma,1}\cdot \E\left[\Gamma_n \right]+C_{\Gamma,2}\E\left[\frac{\Gamma_{n}}{\sqrt{S_{n}}} \right] + \alpha_0 \E\left[\hat{X}_{n},
\right]\end{align}
since $\hat{X}_{n}$ is a martingale such that $\Expect\left[\hat{X}_{n} \mid \mathscr{F}_{n-1}\right] = 0$. Telescoping the above inequality from $n=1$ to $T$ gives
\begin{align}\label{jrn_1}
\sum_{n=1}^{T}\E\left[\zeta(n)\right] & \le \frac{4}{\alpha_0}\Expect\big[\hat{g}(\theta_{1})\big]+\frac{4C_{\Gamma,1}}{\alpha_0}\sum_{n=1}^{T}\E\left[\Gamma_n \right] + \frac{4C_{\Gamma,2}}{\alpha_0}\sum_{n=1}^{T}\E\left[\frac{\Gamma_{n}}{\sqrt{S_{n}}} \right].
\end{align}
Note that \begin{align}
& \sum_{n=1}^{T} \E\left[\Gamma_n \right]  =  \E\left[\sum_{n=1}^{T} \frac{\|\nabla g(\theta_{n},\xi_{n})\|^{2}}{S_{n}} \right] \leq \E\left[\int_{S_{0}}^{S_{T}}\frac{1}{x}\text{d}x \right]\leq \E\left[\ln (S_T/S_0)\right] \leq \E(\ln S_T) - \ln S_0 \notag \\
& \Expect\bigg[\sum_{n=1}^{T}\frac{\|\nabla g(\theta_{n},\xi_{n})\|^{2}}{S_{n}^{\frac{3}{2}}}\bigg] \leq \E\left[\int_{S_{0}}^{S_T} \frac{1}{x^{\frac{3}{2}}} \mathrm{d} x \right] \le \frac{2}{\sqrt{S_{0}}}<+\infty. \notag
\end{align}
Substituting the above results into \cref{jrn_1}, we have
\begin{align}\label{inqu:lem:v1:1}
\sum_{n=1}^{T}\E\left[\zeta(n)\right] & \le \left(\frac{4}{\alpha_0}\Expect\big[\hat{g}(\theta_{1})\big] -\frac{4C_{\Gamma,1}}{\alpha_0}\ln S_0  \right)+ \frac{4C_{\Gamma,1}}{\alpha_0}\E\left[\ln S_T \right] + \frac{4C_{\Gamma,2}}{\alpha_0}\frac{2}{\sqrt{S_0}}.
\end{align}
By \cref{lem_S_{T}} (b), we know that $$S_T \leq \left(\sum_{n=1}^{\infty}\frac{\zeta(n)}{n^{2}} + \sqrt{S_0} \right)^2T^4.$$
Combing \cref{lem_S_{T}} (a), we have 
\begin{align}
\E\left[\ln S_T \right] & \leq 2\E\left[\sum_{n=1}^{\infty}\frac{\zeta(n)}{n^{2}} + \sqrt{S_0}\right] + 4\ln T = 2 \sum_{n=1}^{\infty}\frac{\E\left[\zeta(n)\right]}{n^{2}} + 4\ln T + 2\sqrt{S_0} \notag \\
&  \leq 4\ln T + \mathcal{O}(1). \notag
\end{align}
Then for any $T \geq 1$
\begin{equation}\nonumber\begin{aligned}
\sum_{n=1}^{T}\E\left[\zeta(n)\right] & \le  \frac{16C_{\Gamma,1}}{\alpha_0}  \ln T  + \mathcal{O}(1).
\end{aligned}\end{equation}
The proof is complete. 
\end{proof}

%\subsection{Proof of \cref{lem23}}\label{sec:lem23}
\begin{proof}(of \cref{lem23})
Applying the $ L$-smoothness of $g$ and the iterative formula of AdaGrad-Norm, we have 
\begin{align}\label{inequ:smoothness:g}
g(\theta_{n+1}) \leq g(\theta_n)- \alpha_0 \frac{\nabla g(\theta_n)^{T}\nabla g(\theta_n,\xi_n)}{\sqrt{S_n}} + \frac{ L\alpha_0^2}{2}\frac{\nabla g(\theta_n; \xi_n)^2}{S_n}.
\end{align}
Then combined with $g^{2}(\theta_{n+1})-g^{2}(\theta_{n}) = \left(g(\theta_{n+1})-g(\theta_{n}) \right)\left(g(\theta_{n+1})+g(\theta_{n}) \right)$ we have
\begin{align}\label{inequ:sufifient:v2:1}
& \quad g^{2}(\theta_{n+1})-g^{2}(\theta_{n}) \notag \\ &\le -\frac{2\alpha_{0}g(\theta_{n})\nabla g(\theta_{n})^{\top}\nabla g(\theta_{n},\xi_{n})}{\sqrt{S_{n}}}+ \frac{\alpha_0^2\left(\nabla g(\theta_{n})^{\top}\nabla g(\theta_{n},\xi_{n})\right)^2}{S_n} \notag 
\\& \quad +\left(g(\theta_{n}) - \frac{\alpha_0\nabla g(\theta_{n})^{\top}\nabla g(\theta_{n},\xi_{n})}{\sqrt{S_{n}}}\right) L\alpha_{0}^{2}\frac{\big\|\nabla g(\theta_{n},\xi_{n})\big\|^{2}}{S_{n}} +  \frac{ L^2\alpha_0^4}{4}\frac{\big\|\nabla g(\theta_{n},\xi_{n})\big\|^{4}}{S_{n}^2}  \notag \\
& \mathop{\leq}^{(a)} -\frac{2\alpha_{0}g(\theta_{n})\nabla g(\theta_{n})^{\top}\nabla g(\theta_{n},\xi_{n})}{\sqrt{S_{n}}} + g(\theta_n)\left(2+\alpha_0^2 \right) L\cdot\Gamma_n   + \frac{\alpha_0^2}{2}\left\| \nabla g(\theta_n)\right\|^2\Gamma_n  + \frac{3\alpha_0^4 L^2}{4}\Gamma_n  \notag \\
& \leq -\frac{2\alpha_{0}g(\theta_{n})\nabla g(\theta_{n})^{\top}\nabla g(\theta_{n},\xi_{n})}{\sqrt{S_{n}}} + \left((2+2\alpha_0^2) Lg(\theta_n) + \frac{3\alpha_0^4 L^2}{4} \right)\Gamma_n 
\end{align} 
Here we inherit the notation $\Gamma_n = \left\|\nabla g(\theta_n, \xi_n) \right\|^2/S_n$ in \cref{inequ:sufficient:decrease}. For $(a)$ we use some common inequalities, the facts that  $S_n \geq \left\|\nabla g(\theta_{n},\xi_{n}) \right\|^2$, \cref{loss_bound} such that
\begin{align}
\frac{\left(\nabla g(\theta_{n})^{\top}\nabla g(\theta_{n},\xi_{n})\right)^2}{S_n} &  
\leq \frac{\left\|\nabla g(\theta_{n}) \right\|^2\left\|\nabla g(\theta_{n},\xi_{n}) \right\|^2}{S_n} \leq \frac{2 Lg(\theta_n)\left\|\nabla g(\theta_{n},\xi_{n}) \right\|^2}{S_n} \notag \\
- \frac{\alpha_0\nabla g(\theta_{n})^{\top}\nabla g(\theta_{n},\xi_{n})}{\sqrt{S_{n}}} 
& \leq \frac{1}{2 L}\left\|\nabla g(\theta_{n})\right\|^2 + \frac{\alpha_0^2 L}{2} \frac{\left\|\nabla g(\theta_{n},\xi_{n}) \right\|^2}{S_n} \leq \frac{1}{2 L}\left\|\nabla g(\theta_{n})\right\|^2 + \frac{\alpha_0^2 L}{2}  \notag \\
\frac{\big\|\nabla g(\theta_{n},\xi_{n})\big\|^{4}}{S_{n}^2} &  \leq \frac{\big\|\nabla g(\theta_{n},\xi_{n})\big\|^{2}}{S_{n}}.
\end{align}
and for the last inequality we use \cref{loss_bound} that $\left\|\nabla g(\theta_n) \right\|^2 \leq 2 Lg(
\theta_n).$
For the first term of RHS of \cref{inequ:sufifient:v2:1}, we let $\Delta_{S, n}$ denote  $1/\sqrt{S_n} -1/\sqrt{S_{n-1}}$ and inherit the notation $\zeta(n) = \left\|\nabla g(\theta_{n})\right\|^2/\sqrt{S_{n-1}}$ in \cref{inequ:sufficient:decrease}:
\begin{align}\label{inqu:sufficient:v2:2}
&\frac{g(\theta_{n})\nabla g(\theta_{n})^{\top}\nabla g(\theta_{n},\xi_{n})}{\sqrt{S_{n}}}  = \frac{g(\theta_{n})\nabla g(\theta_{n})^{\top}\nabla g(\theta_{n},\xi_{n})}{\sqrt{S_{n-1}}} + g(\theta_{n})\nabla g(\theta_{n})^{\top}\nabla g(\theta_{n},\xi_{n})\Delta_{S,n} \notag \\ 
& = g(\theta_{n}) \zeta(n)  + \frac{g(\theta_{n})\nabla g(\theta_{n})^{\top}\left( \nabla g(\theta_{n},\xi_{n})- g(\theta_{n})\right)}{\sqrt{S_{n-1}}} + g(\theta_{n})\nabla g(\theta_{n})^{\top}\nabla g(\theta_{n},\xi_{n})\Delta_{S,n}.
\end{align}
We then substitute \cref{inqu:sufficient:v2:2} into \cref{inequ:sufifient:v2:1} and achieve that
\begin{align}\label{inequ:sufficient:v2:3}
g^{2}(\theta_{n+1})-g^{2}(\theta_{n}) 
&\le - 2\alpha_0 g(\theta_{n})\zeta(n) + \left((2+2\alpha_0^2) Lg(\theta_n) + \frac{3\alpha_0^4 L^2}{4} \right)\Gamma_n\notag \\
& + 2\alpha_0g(\theta_{n})\E\left[\nabla g(\theta_{n})^{\top}\nabla g(\theta_{n},\xi_{n})\Delta_{S,n} \mid \mathscr{F}_{n-1}\right] + 2\alpha_0\hat{Y}_n,
\end{align}
where $\hat{Y}_n$ is a martingale different sequence and defined below
\begin{equation}\nonumber\begin{aligned}
 \hat{Y}_n &:=\frac{g(\theta_{n})\nabla g(\theta_{n})^{\top}(\nabla g(\theta_{n})-\nabla g(\theta_{n},\xi_{n}))}{\sqrt{S_{n-1}}}\\&+g(\theta_{n})\nabla g(\theta_{n})^{\top}\nabla g(\theta_{n},\xi_{n})\Delta_{S, n}-g(\theta_{n})\Expect\bigg[\nabla g(\theta_{n})^{\top}\nabla g(\theta_{n},\xi_{n})\Delta_{S, n}\bigg|\mathscr{F}_{n-1}\bigg].\end{aligned}\end{equation}
For the second to last term of RHS of \cref{inequ:sufficient:v2:3} we have
\begin{equation}\nonumber\begin{aligned}
&\quad 2\alpha_{0}g(\theta_{n})\Expect\bigg[\nabla g(\theta_{n})^{\top}\nabla g(\theta_{n},\xi_{n})\Delta_{S,n}\bigg|\mathscr{F}_{n-1}\bigg]\\& \mathop{\le}^{(a)} \alpha_{0}g(\theta_{n})\|\nabla g(\theta_{n})\|^{2} \Delta_{S,n}+4\alpha_{0}g(\theta_{n})\Expect^{2}\bigg[\nabla g(\theta_{n},\xi_{n})\sqrt{\Delta_{S,n}}\bigg|\mathscr{F}_{n-1}\bigg] \notag \\
& \mathop{\le}^{(b)}\frac{\alpha_{0}g(\theta_{n})\|\nabla g(\theta_{n})\|^{2}}{\sqrt{S_{n-1}}}+4\alpha_{0}g(\theta_{n})\Expect[\|\nabla g(\theta_{n},\xi_{n})\|^{2}|\mathscr{F}_{n-1}]\cdot\Expect\bigg[\Delta_{S,n}\bigg|\mathscr{F}_{n-1}\bigg]\\&\mathop{\le}^{(c)} \frac{\alpha_{0}g(\theta_{n})\|\nabla g(\theta_{n})\|^{2}}{\sqrt{S_{n-1}}}+4\alpha_{0}g(\theta_{n})\Expect\bigg[(\sigma_{0}\|\nabla g(\theta_{n})\|^{2}+\sigma_{1})\Delta_{S,n}\bigg|\mathscr{F}_{n-1}\bigg]\\&\mathop{\le}^{(d)} \alpha_{0}g(\theta_{n})\zeta(n)+4 L\alpha_{0}\sigma_{0}g^{2}(\theta_{n})\Expect\bigg[ \Delta_{S,n}\bigg|\mathscr{F}_{n-1}\bigg]+4\alpha_{0}\sigma_{1}g(\theta_{n})\Expect\bigg[\Delta_{S,n}\bigg|\mathscr{F}_{n-1}\bigg],
\end{aligned}\end{equation} 
where $(a)$ follows from mean inequality, $(b)$ uses Cauchy-Schwartz inequality, $(c)$ applies the affine noise variance condition, and $(d)$ follows from \cref{loss_bound} which states $\|\nabla g(\theta)\|^2 \leq 2 Lg(\theta)$. We then substitute the above estimation into \cref{inequ:sufficient:v2:3}
\begin{align}\label{inequ:sufficient:key}
{g^{2}(\theta_{n+1})}-{g^{2}(\theta_{n})} 
& \le- \alpha_0 g(\theta_{n}) \zeta(n)+4 L\alpha_{0}\sigma_{0}g^{2}(\theta_{n})\Expect\left[\Delta_{S,n}\mid\mathscr{F}_{n-1}\right]+4\alpha_{0}\sigma_{1}g(\theta_{n})\Expect\left[\Delta_{S,n}\mid\mathscr{F}_{n-1}\right] \notag \\& +\left((2+2\alpha_0^2) Lg(\theta_n) + \frac{3\alpha_0^4 L^2}{4} \right)\Gamma_n +2\alpha_0 \hat{Y}_n.
\end{align}
Next, for any stopping time $\tau$ that satisfies $[\tau=i] \in \mathscr{F}_{i-1}\ (\forall\ i>0)$, telescoping the index $n$ from $1$ to $\tau\wedge T-1$ in \cref{inequ:sufficient:key} and taking expectation on the above inequality yields
\begin{equation}\label{jrn_4}\begin{aligned}
&\Expect\big[{g^{2}(\theta_{\tau\wedge T})}\big]-\Expect\big[{g^{2}(\theta_{1})}\big]\le- \alpha_0 \Expect\bigg[\sum_{n=1}^{\tau\wedge T-1}g(\theta_{n}) \zeta(n)\bigg]\\&+4 L\alpha_{0}\sigma_{0}\Expect\bigg[\sum_{n=1}^{\tau\wedge T-1}g^{2}(\theta_{n})\Expect\bigg[\Delta_{S,n}\bigg|\mathscr{F}_{n-1}\bigg]\bigg]+4\alpha_{0}\sigma_{1}\Expect\bigg[\sum_{n=1}^{\tau\wedge T-1}g(\theta_{n})\Expect\bigg[\Delta_{S,n}\bigg|\mathscr{F}_{n-1}\bigg]\bigg]\\&+\Expect\bigg[\sum_{n=1}^{\tau\wedge T-1}\left((2+2\alpha_0^2) Lg(\theta_n) + \frac{3\alpha_0^4 L^2}{4} \right) \Gamma_n\bigg]+2\alpha_0\Expect\bigg[\sum_{n=1}^{\tau\wedge T-1} \hat{Y}_n\bigg]. 
\end{aligned}\end{equation} 
We further use \emph{Doob's stopped} theorem that $\Expect\big[\sum_{n=1}^{\tau\wedge T-1}\Expect(\cdot|\mathscr{F}_{n-1})\big]=\Expect\big[\sum_{n=1}^{\tau\wedge T-1}\cdot\big]$ to simplify \cref{jrn_4} and achieve that 
\begin{align}\label{jrn_5}
& \quad \Expect\big[{g^{2}(\theta_{\tau\wedge T})}\big]-\Expect\big[{g^{2}(\theta_{1})}\big] \notag \\
&\le- \alpha_0 \Expect\bigg[\sum_{n=1}^{\tau\wedge T-1}g(\theta_{n}) \zeta(n)\bigg]+4 L\alpha_{0}\sigma_{0}\Expect\bigg[\sum_{n=1}^{\tau\wedge T-1}g^{2}(\theta_{n})\Delta_{S,n}\bigg]+4\alpha_{0}\sigma_{1}\Expect\bigg[\sum_{n=1}^{\tau\wedge T-1}g(\theta_{n})\Delta_{S,n}\bigg] \notag \\&\quad +\Expect\bigg[\sum_{n=1}^{\tau\wedge T-1}\left((2+2\alpha_0^2) Lg(\theta_n) + \frac{3\alpha_0^4 L^2}{4} \right) \Gamma_n\bigg] + 0.
\end{align}
For the second term on the RHS of the aforementioned inequality, we have the following estimation
\begin{equation}\label{jrn_6}\begin{aligned}
&\quad \Expect\bigg[\sum_{n=1}^{\tau\wedge T-1}g^{2}(\theta_{n})\bigg(\Delta_{S, n}\bigg)\bigg] \notag \\
& =\Expect\bigg[\sum_{n=0}^{\tau\wedge T-2}\frac{g^{2}(\theta_{n+1})}{\sqrt{S_{n}}}-\sum_{n=1}^{\tau\wedge T-1}\frac{g^{2}(\theta_{n})}{\sqrt{S_{n}}}\bigg] \leq \Expect\bigg[\frac{g^{2}(\theta_{1})}{\sqrt{S_{0}}}\bigg]+\Expect\bigg[\sum_{n=1}^{\tau\wedge T-1}\frac{g^{2}(\theta_{n+1})-g^{2}(\theta_{n})}{\sqrt{S_{n}}}\bigg]\\& \mathop{\le}^{(a)} \Expect\bigg(\frac{g^{2}(\theta_{1})}{\sqrt{S_{0}}}\bigg)+ 2\alpha_0\Expect\bigg[\sum_{n=1}^{\tau\wedge T-1}\frac{g(\theta_{n})\|\nabla g(\theta_{n})\|\|\nabla g(\theta_{n},\xi_{n})\|}{S_{n}}\bigg] \\&\quad +\Expect\bigg[\sum_{n=1}^{\tau\wedge T-1}\left((2+2\alpha_0^2) Lg(\theta_n) + \frac{3\alpha_0^4 L^2}{4} \right)\frac{\big\|\nabla g(\theta_{n},\xi_{n})\big\|^{2}}{S_{n}^{\frac{3}{2}}}\bigg]\\& \mathop{\le}^{(b)} \Expect\bigg[\frac{g^{2}(\theta_{1})}{\sqrt{S_{0}}}\bigg]+\frac{\alpha_{0}\psi_1}{4}\Expect\bigg[\sum_{n=1}^{\tau\wedge T-1}\frac{g(\theta_{n})\|\nabla g(\theta_{n})\|^{2}}{\sqrt{S_{n-1}}}\bigg]+\frac{4\alpha_0}{\psi_1} \Expect\bigg[\sum_{n=1}^{\tau\wedge T-1}\frac{g(\theta_{n})\|\nabla g(\theta_{n},\xi_{n})\|^{2}}{S_{n}^{\frac{3}{2}}}\bigg] \notag \\
& \quad +\Expect\bigg[\sum_{n=1}^{\tau\wedge T-1}\left((2+2\alpha_0^2) Lg(\theta_n) + \frac{3\alpha_0^4 L^2}{4} \right)\frac{\big\|\nabla g(\theta_{n},\xi_{n})\big\|^{2}}{S_{n}^{\frac{3}{2}}}\bigg],
\end{aligned}\end{equation}
where for $(a)$ we use the upper bound of $g^2(\theta_{n+1}) - g^2(\theta_n)$ in \cref{inequ:sufifient:v2:1} and \emph{the Cauchy-Schwartz inequality}, and for $(b)$ we use \emph{Young inequality} and let $\psi_1 = \frac{1}{4 L\sigma_0\alpha_0}$. 
Similarly, we can estimate the third term on the RHS of \cref{jrn_5} as follows.
\begin{equation}\label{jrn_6:v1}\begin{aligned}
&\quad\Expect\bigg[\sum_{n=1}^{\tau\wedge T-1}g(\theta_{n})\bigg(\Delta_{S, n}\bigg)\bigg] \notag \\
& =\Expect\bigg[\sum_{n=0}^{\tau\wedge T-2}\frac{g(\theta_{n+1})}{\sqrt{S_{n}}}-\sum_{n=1}^{\tau\wedge T-1}\frac{g(\theta_{n})}{\sqrt{S_{n}}}\bigg] \leq \Expect\bigg[\frac{g(\theta_{1})}{\sqrt{S_{0}}}\bigg]+\Expect\bigg[\sum_{n=1}^{\tau\wedge T-1}\frac{g(\theta_{n+1})-g(\theta_{n})}{\sqrt{S_{n}}}\bigg]\\& \mathop{\le}^{(a)} \Expect\bigg[\frac{g(\theta_{1})}{\sqrt{S_{0}}}\bigg]+ \alpha_0\Expect\bigg[\sum_{n=1}^{\tau\wedge T-1}\frac{\|\nabla g(\theta_{n})\|\|\nabla g(\theta_{n},\xi_{n})\|}{S_{n}}\bigg] +\frac{\alpha_0^2 L}{2}\Expect\bigg[\sum_{n=1}^{\tau\wedge n-1}\frac{\big\|\nabla g(\theta_{n},\xi_{n})\big\|^{2}}{S_{n}^{\frac{3}{2}}}\bigg]\\& \mathop{\le}^{(b)} \Expect\bigg[\frac{g(\theta_{1})}{\sqrt{S_{0}}}\bigg]+\frac{\alpha_{0}\psi_2}{4}\Expect\bigg[\sum_{n=1}^{\tau\wedge n-1}\frac{\|\nabla g(\theta_{n})\|^{2}}{\sqrt{S_{n-1}}}\bigg]+\left(\frac{\alpha_0}{\psi_2}  + \frac{\alpha_0^2 L}{2}\right)\Expect\bigg[\sum_{n=1}^{\tau\wedge T-1}\frac{\|\nabla g(\theta_{n},\xi_{n})\|^{2}}{S_{n}^{\frac{3}{2}}}\bigg],
\end{aligned}\end{equation}
where for $(a)$ we use \cref{inequ:smoothness:g} and \emph{the Cauchy-Schwartz inequality} and for $(b)$ we use \emph{Young inequality} and let $\psi_2 = 1/(4\alpha_0\sigma_1)$. 
Substituting the above estimations into \cref{jrn_5} we have
\begin{align}\label{jrn_8}
&\Expect\big({g^{2}(\theta_{\tau\wedge T})}\big)-\Expect\big[{g^{2}(\theta_{1})}\big]\le- \frac{3\alpha_{0}}{4}\Expect\bigg[\sum_{n=1}^{\tau\wedge T-1}  g(\theta_{n}) \zeta(n)\bigg] +\frac{\alpha_{0}}{4}\Expect\bigg[] \zeta(n)\bigg] + \tilde{C}_1\Expect\bigg[\sum_{n=1}^{\tau\wedge T-1}\frac{g(\theta_{n})\Gamma_n}{\sqrt{S_{n}}}\bigg] \notag \\& +\tilde{C}_{2}\Expect\bigg[\sum_{n=1}^{\tau\wedge T-1} g(\theta_{n}) \Gamma_n\bigg] 
+\tilde{C}_3\Expect\bigg[\sum_{n=1}^{\tau\wedge T-1}\frac{\Gamma_n}{\sqrt{S_{n}}}\bigg]+\frac{3\alpha_{0}^2 L^2}{4}\Expect\bigg[\sum_{n=1}^{\tau\wedge T-1} \Gamma_n\bigg] + \mathcal{O}(1),
\end{align} where
\begin{align}
\tilde{C}_{1}&:=64\sigma_0^2\alpha_0^3 L^2 + 8\sigma_0\alpha_0(1+\alpha_0^2) L^2,\ \ \tilde{C}_{2}:= 2(1+\alpha_0^2) L,  \notag \\
\tilde{C}_{3} & :=4\alpha_0^3\sigma_1\left(4\sigma_1+\frac{ L}{2}\right)+3\sigma_0\alpha_{0}^{5} L^3. \notag 
\end{align}
We notice the following facts
\begin{align} 
&\sum_{n=1}^{\tau\wedge T-1} \Gamma_n \leq \sum_{n=1}^{T}\Gamma_n = \sum_{n=1}^{T}\frac{\|\nabla g(\theta_{n},\xi_{n})\|^{2}}{S_{n}}<\int_{S_{0}}^{S_{T}}\frac{1}{x}\text{d}x<\ln S_{T}-\ln S_{0}, \notag \\&
\sum_{n=1}^{\tau\wedge T-1}\frac{\Gamma_n}{\sqrt{S_n}}\le \sum_{n=1}^{+\infty}\frac{\|\nabla g(\theta_{n},\xi_{n})\|^{2}}{S_{n}^{\frac{3}{2}}}\le \int_{S_{0}}^{+\infty}x^{-\frac{3}{2}}\text{d}x\le \frac{2}{\sqrt{S_0}}, \notag \\& 
\Expect\bigg[\sum_{n=1}^{\tau\wedge T-1} \zeta(n)\bigg]\le \Expect\bigg[\sum_{n=1}^{T}\frac{\|\nabla g(\theta_{n})\|^{2}}{\sqrt{S_{n-1}}}\bigg]< \mathcal{O}(1)+ 2\left(\frac{\sigma_1}{\sqrt{S_0}} + \alpha_0 L\right)\Expect[\ln S_{T}],  \notag 
\end{align} 
where the last fact follows from \cref{inqu:lem:v1:1} of \cref{lem8}.
We then use these facts to simplify  \cref{jrn_8} as
\begin{align}\label{inequ:lem:v2:1}
&\quad\Expect\big[{g^{2}(\theta_{\tau\wedge T})}\big] \notag \\ 
&\le- \frac{3\alpha_{0}}{4}\Expect\bigg[\sum_{n=1}^{\tau\wedge T-1} g(\theta_{n}) \zeta(n)\bigg]+
2\left(\frac{\sigma_1}{\sqrt{S_0}} + \alpha_0 L\right)\Expect[\ln S_{T}]+\tilde{C}_1\Expect\left[\sup_{n\le T}g(\theta_{n})  \sum_{n=1}^{\tau\wedge T-1}\frac{\Gamma_n}{\sqrt{S_n}}\right]\notag \\
& \quad +\tilde{C}_2\Expect\left[\big(\sup_{n\le T}g(\theta_{n})\big) \cdot \sum_{n=1}^{\tau\wedge T-1}\Gamma_n\right] + \frac{2\tilde{C}_3}{\sqrt{S_0}} +\frac{3\alpha_{0}^2 L^2}{4}\E\left[\ln S_{T}\right] + \mathcal{O}(1)
\notag \\
&\mathop{\le}^{(a)} - \frac{3\alpha_{0}}{4}\Expect\bigg[\sum_{n=1}^{\tau\wedge T-1} g(\theta_{n}) \zeta(n)\bigg]+
2\left(\frac{\sigma_1}{\sqrt{S_0}} + \alpha_0 L\right)\Expect[\ln S_{T}]+\frac{2\tilde{C}_1}{\sqrt{S_0}}\Expect\left[\sup_{n\le T}g(\theta_{n})\right] \notag \\
&\quad +\tilde{C}_2\Expect\left[\sup_{n\le T}g(\theta_{n}) \cdot \ln(S_T)\right] + \frac{3\alpha_{0}^2 L^2}{4}\E\left[\ln S_{T}\right]+ \mathcal{O}(1).
\end{align}  
Then for any $\lambda>0,$ we define a stopping time $\tau^{(\lambda)}:=\min\Big\{n:{g^{2}(\theta_{n})}>\lambda\Big\}.$ For any$\ \lambda_{0}>0,$ we let $\tau = \tau^{(\ln T)\lambda_{0}}\wedge T\ (\forall\ T\ge 3)$ in  \cref{inequ:lem:v2:1} and use the \emph{ Markov's inequality}
\begin{align}\label{jrnjrn_1}
\Pro\Bigg(\frac{\sup_{1\le n\le T}{g^{\frac{3}{2}}(\theta_{n})}}{\ln^{\frac{3}{2}} T}>\lambda_{0}\Bigg)&=\Pro\Big(\sup_{1\le n\le T}{g^{2}(\theta_{n})}>\lambda^{\frac{4}{3}}_{0}\ln^{2} T\Big)=\Expect\left[\I_{\tau^{(\ln^{2}T)\lambda_{0}}\wedge T}\right]\notag \\
&\le \frac{1}{\lambda^{\frac{4}{3}}_{0}\ln^{2} T}\cdot\Expect\left[{g^{2}(\theta_{\tau^{(\ln^{2} T)\lambda_{0}}\wedge T})}\right] \notag \\
&\mathop{\le}^{(a)}\frac{\phi_{0}}{\lambda^{\frac{4}{3}}_{0} \ln T}\Bigg(\Expect\bigg[\frac{\sup_{1\le k\le n}{g^{\frac{3}{2}}(\theta_{n})}}{\ln^{\frac{3}{2}} T}\bigg]\Bigg)^{\frac{2}{3}}+\frac{\phi_{1}}{\lambda_{0}^{\frac{4}{3}}\ln^2 T},\end{align}
where $\phi_0 =\frac{2\tilde{C}_1}{\sqrt{S_0}} + \left(4\ln T + 2\sqrt{S_0} \right) + 2 \left(\Expect\ln^{3}(\zeta)\right)^{\frac{1}{3}}$ and $\phi_1 = 2\left(\frac{\sigma_1}{\sqrt{S_0}} + \alpha_0 L\right)\E\left[\ln S_{T}\right] + \mathcal{O}(1)$. The last inequality $(a)$ follows $\ln T>1\ (\forall\ T\ge 3)$, and since $g(x) = x^{3/2}$ is convex, by Jensen inequality 
\begin{align}
\Expect\left[\sup_{n\le T}g(\theta_{n}) \right]^{\frac{3}{2}} & \leq \Expect\left[\sup_{n\le T}g^{\frac{3}{2}}(\theta_{n}) \right] \notag 
 \end{align}
and by \emph{Holder inequality} and the upper bound of $S_T \leq \left(1+\zeta\right)^2T^4$ and $\zeta = \sqrt{S_{0}}+\sum_{n=1}^{\infty}\|\nabla g(\theta_{n},\xi_{n})\|^{2}/n^2$ is uniformly bounded in \cref{lem_S_{T}}. We have 
\begin{align}\label{inequ:glnST}
\Expect\left[\sup_{n\le T}g(\theta_{n}) \cdot \ln(S_T)\right] & \leq 4\ln T\Expect\left[\sup_{n\le T}g(\theta_{n}) \right] + 2\Expect\left[\sup_{n\le T}g(\theta_{n})\ln( 1+\zeta) \right] \notag \\
& \mathop{\leq}^{(a)} \left(4\ln T + 2\sqrt{S_0} \right)\left(\Expect\sup_{n\le T}g^{\frac{3}{2}}(\theta_{n})\right)^{\frac{2}{3}} +2\Expect\left[\sup_{n\le T}g^{\frac{3}{2}}(\theta_{n})\right]^{\frac{2}{3}}   \left(\Expect\ln^{3}(\zeta)\right)^{\frac{1}{3}}. 
\end{align}
In step $(a)$, we first used the common inequality $\ln(1+x)\le x\ (\forall\ x>-1)$, and then applied the \emph{Hölder's} inequality, i.e., $\Expect[XY]\le \Expect^{\frac{2}{3}}[\|X\|^{\frac{3}{2}}]\Expect^{\frac{1}{3}}[\|Y\|^{3}].$ Next, we bound the expectation of $\sup_{1\le n\le T}{g^{\frac{3}{2}}(\theta_{n})/\ln^{\frac{3}{2}} T}$
{\begin{align}\label{jrnjrn_0}
&\Expect\left[\frac{\sup_{1\le n\le T}{g^{\frac{3}{2}}(\theta_{n})}}{\ln^{\frac{3}{2}}T}\right]\notag\\
&=\Expect\left[\I_{\bigg(\frac{\sup_{1\le n\le T}{g^{\frac{3}{2}}(\theta_{n})}}{\ln^{\frac{3}{2}}n}\le 1\bigg)}\frac{\sup_{1\le n\le T}{g^{\frac{3}{2}}(\theta_{n})}}{\ln^{\frac{3}{2}}n}\right]+\Expect\left[\I_{\bigg(\frac{\sup_{1\le n\le T}{g^{\frac{3}{2}}(\theta_{n})}}{\ln^{\frac{3}{2}}n}>1\bigg)}\frac{\sup_{1\le n\le T}{g^{\frac{3}{2}}(\theta_{n})}}{\ln^{\frac{3}{2}} T}\right]\notag\\
&\le 1+\int_{1}^{+\infty}-\lambda\ \ \text{d} \Pro\Big(\frac{\sup_{1\le n\le T}{g^{\frac{3}{2}}(\theta_{n})}}{\ln^{\frac{3}{2}} T}>\lambda\Big)\notag\\
&=1+\int_{1}^{+\infty}\Pro\Big(\frac{\sup_{1\le n\le T}{g^{\frac{3}{2}}(\theta_{n})}}{\ln^{\frac{3}{2}} T}>\lambda\Big) \ \text{d} \lambda\notag\\
&\le 1+\int_{1}^{+\infty}\frac{1}{\lambda^{\frac{4}{3}}}\Bigg(\frac{\phi_{0}}{\ln T}\Bigg(\Expect\bigg[\frac{\sup_{1\le n\le T}{g^{\frac{3}{2}}(\theta_{n})}}{\ln^{\frac{3}{2}}n}\bigg]\Bigg)^{\frac{2}{3}}+\frac{\phi_{1}}{\ln^2 T}\Bigg)\text{d}\lambda \notag\\
&=1+\frac{3\phi_{0}}{\ln T}\Expect\Bigg[\frac{\sup_{1\le n\le T}{g^{\frac{3}{2}}(\theta_{n})}}{\ln^{\frac{3}{2}} T}\Bigg]^{\frac{2}{3}}+\frac{3\phi_{1}}{\ln^2 T}.
\end{align}}
for $T \geq 3$, we have $\ln T \geq 1$ and recall the upper bound of $S_T$ in \cref{lem_S_{T}}
\begin{align}
\E[\ln S_T] & \leq \E[2\ln(1+\zeta) + 4\ln T] \leq  \mathcal{O}(1) + 4\ln T \notag \\
\frac{\phi_0}{\ln T} & =  \frac{2\tilde{C}_1/\sqrt{S_0} + 4\ln T + 2\sqrt{S_0}}{ \ln T} + \frac{(\E[\ln^3 \zeta])^{1/3}}{\ln T}  
 = 4 + \frac{\mathcal{O}(1)}{\ln T}  + \frac{(\E[\ln^3 \zeta])^{1/3}}{\ln T} = 4 + \frac{\mathcal{O}(1)}{\ln T} \notag \\
\frac{\phi_{1}}{\ln^2 T} & =  2\left(\frac{\sigma_1}{\sqrt{S_0}} + \alpha_0 L\right)\frac{\E\left[\ln S_{T}\right]}{\ln^2 T} + \frac{\mathcal{O}(1)}{\ln T} \leq 2\left(\frac{\sigma_1}{\sqrt{S_0}} + \alpha_0 L\right)\frac{4\ln T}{\ln^2 T} + \frac{\mathcal{O}(1)}{\ln T} = \frac{\mathcal{O}(1)}{\ln T}, \notag 
\end{align}
where we use the fact that there exists $c_0 > 0$ such that $\ln^3(x) \leq \max(c_0, x)$ for all $x > 0$, then
\begin{align}
(\E[\ln^3 \zeta])^{1/3} & \leq   \max\left(c_0^{1/3},  \left( \E (\zeta)\right)^{1/3}\right) < +\infty. \notag
\end{align}
We treat $\E\left[\sup_{1\le n\le T}{g^{\frac{3}{2}}(\theta_{n})}/\ln^{\frac{3}{2}} T\right]$ as the variable. Then to solve \cref{jrnjrn_0} is equivalent to solve
\begin{align}
x \leq 1 + \left(4 + \frac{\mathcal{O}(1)}{\ln T} \right) x^{2/3} + \frac{\mathcal{O}(1)}{\ln T} \notag 
\end{align}
We have
\begin{equation}\label{adagrad:60.1}\begin{aligned}
&\Expect\Bigg[\frac{\sup_{1\le n\le T}{g^{\frac{3}{2}}(\theta_{n})}}{\ln^{\frac{3}{2}} T}\Bigg]\le\max\left\lbrace 1 + \frac{\mathcal{O}(1)}{\ln T},\left(4+ \frac{\mathcal{O}(1)}{\ln T}\right)^{3} \right\rbrace<+\infty.
\end{aligned}\end{equation} By Jensen inequality with the convex function $g(x) = x^{3/2}$, this also implies that 
\[\Expect\Big[\sup_{1\le n\le T}g(\theta_{n})\Big] \leq \Big(\Expect\sup_{1\le n\le T}g(\theta_{n})^{3/2}\Big)^{2/3}\le \mathcal{O}\left(\ln T\right).\]
We set the stopping time $\tau$ in  \cref{inequ:lem:v2:1} to be $n$ and combine \cref{inequ:glnST} and the estimation of $\E[\ln S_T]$
\begin{equation}\nonumber\begin{aligned}
\Expect\Bigg[\sum_{n=1}^{T-1}\frac{ g(\theta_{n})\|\nabla g(\theta_{n})\|^{2}}{\sqrt{S_{n-1}}}\Bigg] = \Expect\Bigg[\sum_{n=1}^{T-1} g(\theta_{n}) \zeta(n)\Bigg]\le \mathcal{O}(\ln^2 T).
\end{aligned}\end{equation} 
The lemma follows. 
\end{proof}

\section{Appendix: Proofs of RMSProp }\label{p_RMSProp}
%\subsection{Dependency Graph of Lemmas and Theorems}
In this section, we will provide the proofs of the lemmas and theorems related to RMSProp, as discussed in \cref{sec:AdaGrad:coordinate}. To facilitate a clear grasp of the concepts, we provide a dependency graph below to illustrate the relationships among these lemmas and theorems. 
\begin{figure}[ht]
\begin{tikzpicture}[
  node distance=1cm,
  every node/.style={draw, rectangle, minimum width=1.2cm, minimum height=0.5cm, text centered, font=\small},
  every comment/.style={rectangle, draw=none, font=\small},
  >=Stealth, 
  thick]
\node [label={sufficient decrease}](lemma31) {\cref{sufficient:lem'}}; % lemma 5.1

\node (lemma33) [below=of lemma31]{ \cref{vital_0}}; % lemma D.1
%\node (lemma33) [right= of lemma32]{Lemma \ref{vital_0}};
\node (lemma35) [label={stability}, right= of lemma31]{\cref{bound''}};  % Theorem 5.1

\node (lemma34) [below= of lemma35]{ \cref{RMSProp_0}}; % lemma D.2
\node (lemma36) [right= of lemma35]{\cref{bounded_v}}; % Lemma D.3

\node (lemma37) [below= of lemma36]{ \cref{RMSProp_9}}; % Lemma D.4 
\node (lemma38) [below= of lemma37]{ \cref{convergence_v}}; % Lemma D.5
\node (lemma39) [left= of lemma38]{ \cref{convergence_1.0'}}; % Theorem 5.2

\node (lemma41) [right= of lemma36]{ \cref{loss_bound}}; % Lemma A.1
\node (lemma35-1) [label=stability, below= of lemma41]{\cref{bound''}};  % Theorem 5.1
\node (lemma39-1) [label=almost-sure, below= of lemma35-1]{ \cref{convergence_1.0'}}; % Theorem 5.2
\node (lemma40) [label=mean-square, right= of lemma35-1, distance=8cm]{ \cref{convergence_2.0'}};
% \node (lemma42) [ below= of lemma40]{};

 % lemma 5.1

%\node (lemma33) [below =of lemma31] {Lemma \ref{RMSProp_0}};
%\node (lemma34) [below =of lemma33] {Lemma \ref{vital_1}};
%\node (lemma35) [below =of lemma34] {Lemma \ref{vital_0}};
%\node (lemma36) [below =of lemma35] {Theorem \ref{bound''}};

\draw[->] (lemma31) to (lemma35);
\draw[->] (lemma35) to (lemma36);
\draw[->] (lemma31) to (lemma34);
\draw[->] (lemma33) to (lemma34);
\draw[->] (lemma34) to (lemma35);
\draw[->] (lemma31) to (lemma35);

\draw[->] (lemma36) to (lemma37);
\draw[->] (lemma34) to (lemma37);
\draw[->] (lemma37) to (lemma38);
\draw[->] (lemma38) to (lemma39);
\draw[->] (lemma34) to (lemma39);
\draw[thick, dotted] ($(lemma36.north east)!0.5!(lemma41.north west)$) to ($(lemma38.south east)!0.5!(lemma39-1.south west)$); 
\draw[->] (lemma41) to (lemma40);
\draw[->] (lemma35-1) to (lemma40);
\draw[->] (lemma39-1) to node[draw=none, left=-3.6cm, font=\small, align=center]{+ Lebesgue's dominated\\theorem}(lemma40);
% -- node[midway, left=0.3cm, draw=none, text width=2cm, font=\scriptsize, rotate=90]{\text{\emph{Lebesgue's Dominated Convergence} theorem}}
%\draw[->] (lemma35) to[out=340, in=200] (lemma40);
%\draw[->] (lemma42) -- (lemma45);
%draw[->] (lemma31) -- (lemma45);
%\draw[->] (lemma31) to[out=315, in=45] (lemma36);
%\draw[->] (lemma31) to[out=315, in=45] (lemma37);
%\draw[->] (lemma33) to[out=315, in=45] (lemma37);
%\draw[->] (lemma43) -- (lemma50);
%\draw[->] (lemma50) -- (lemma44);

\end{tikzpicture}
\caption{The proof structure of RMSProp}
\end{figure}
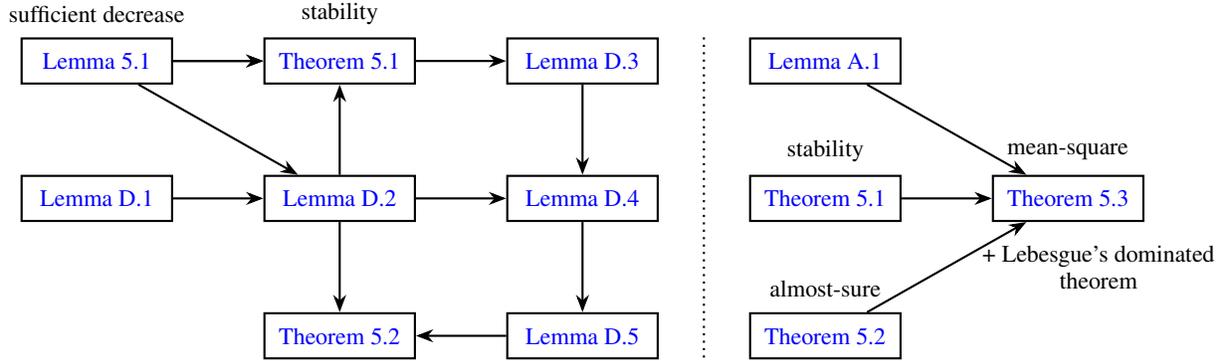

\subsection{Useful Properties of RMSProp}
\begin{property}\label{property_0}
The sequence \(\{\eta_t\}_{t\ge 1}\) is monotonically decreasing per coordinate with respect to \( t \).
\end{property}
\begin{proof}
By the iterative formula of RMSProp~in~\cref{RMSProp}, we know that for all $\ t\ge 1$
\[v_{t+1}=\beta_{2,t+1}v_{t}+(1-\beta_{2,t+1})(\nabla g(\theta_{t+1},\xi_{t+1}))^{\circ 2}=\Big(1-\frac{1}{t+1}\Big)v_{t}+\frac{1}{t+1}(\nabla g(\theta_{t+1},\xi_{t+1}))^{\circ 2},\]
which induces that 
\begin{align}\label{jh_100}
(t+1)v_{t+1,i}=\big((t+1)-1\big)v_{t,i}+ (\nabla_i g(\theta_{t+1},\xi_{t+1}))^{2}\ge tv_{t,i}.
\end{align}
This implies that \(tv_{t,i}\) is monotonically non-decreasing. Since
\begin{align*}
\eta_{t,i}=\frac{\alpha_t}{\sqrt{v_{t,i}}+\epsilon}=\frac{\sqrt{t}\alpha_t}{\sqrt{tv_{t,i}}+\sqrt{t}\epsilon}=\frac{1}{\sqrt{tv_{t,i}}+\sqrt{t}\epsilon},
\end{align*} 
where the global learning rate $\alpha_t = 1/\sqrt{t}$ and the denominator is monotonically non-increasing and greater than 0. Thus, the sequence \(\eta_t\) is monotonically decreasing at each coordinate with respect to $t$.  
\end{proof}
\begin{property}\label{property_1}
The sequence \(\{\eta_t\}_{t\ge 1}\) satisfies that for each coordinate $i$, \(tv_{t,i}\ge r_1 S_{t,i},\) where \(r_1:=\min\{\beta_1,1-\beta_1\},\) \(S_{t,i}:=v+\sum_{k=1}^{t}(\nabla_{i}g(\theta_{k},\xi_{k}))^{2}\) for all \( t\ge 1 \), and \(S_{0,i}:=v.\)
\end{property}
\begin{proof}
%First two equations are obvious; we primarily focus on deriving the second inequality. 
For $v_{1,i}$, we derive the following estimate
\begin{align*}
v_{1,i} &=\beta_{1}v_{0,i} + (1-\beta_{1})(\nabla_{i}g(\theta_{1},\xi_{1}))^{2} =\beta_{1}v+(1-\beta_1)(\nabla_{i}g(\theta_{1},\xi_{1}))^{2}.
%\\&=\beta_{1}v+(\nabla_{i}g(\theta_{1},\xi_{1}))^{2}-\beta_{1}(\nabla_{i}g(\theta_{1},\xi_{1}))^{2}.
\end{align*}
We observe that $ \min(\beta_1, 1-\beta_1) S_{1,i} \le v_{1,i}\le S_{1,i}.$ Recalling \cref{jh_100} that $kv_{k,i}\ge(k-1)v_{k-1,i}+ (\nabla_{i}g(\theta_{k},\xi_{k}))^{2}$ for $\forall\ k\ge 2$ and summing up it  for $2 \leq k \leq t$, we have $\forall\ t\ge 2,$ $$t v_{t,i}\ge v_{1,i}+\sum_{k=2}^{t}(\nabla_{i}g(\theta_{k},\xi_{k}))^{2}$$
Combining this with the estimate for $v_{1,i}$
\[tv_{t,i}\ge \beta_1 v+ (1-\beta_1)(\nabla_{i}g(\theta_{1},\xi_{1}))^{2}+\sum_{k=2}^{t}(\nabla_{i}g(\theta_{k},\xi_{k}))^{2},\]
we have \( tv_{t,i}\ge \min(\beta_1, 1-\beta_1) S_{t,i}.\) 
\end{proof}
% \red{\begin{remark}\label{S_t}
% To simplify the proofs of subsequent theorems, we define two auxiliary variables: \(\Sigma_{v_{t}} := \sum_{i=1}^{d}v_{t,i}\) and \(S_{t} := \sum_{i=1}^{d}S_{t,i}\). Additionally, for convenience in the subsequent proofs, we define a new initial parameter based on \(S_{0,i}\) as \(\eta_{v_{0},i} = S_{0,i}/\alpha_{1} = v/\alpha_{1}\).
% \end{remark}}

\subsection{Auxiliary Lemmas of RMSProp}\label{proof:lem:rmsprop}

\begin{proof}(of \cref{sufficient:lem'})
Recalling the $L$-smoothness of the function and substituting the formula of RMSProp gives
% \begin{align*}
% g(\theta_{t+1})-g(\theta_{t})\le \nabla g(\theta_{t})^{\top}(\theta_{t+1}-\theta_{t})+\frac{ L}{2}\|\theta_{t+1}-\theta_{t}\|^{2}.     
% \end{align*}
% Then, by substituting the iterative formula for \( \theta_{t} \) from \cref{RMSProp} into the above inequality, we obtain:
\begin{align}\label{adam_-1}
g(\theta_{t+1})-g(\theta_{t})& \mathop{\le}^{(a)}\underbrace{-\sum_{i=1}^{d}\eta_{t,i}\nabla_{i}g(\theta_{t})\nabla_{i}g(\theta_{t},\xi_{t})}_{\Theta_{t,1}}+\frac{ L}{2}\sum_{i=1}^{d}\eta_{t,i}^{2}\nabla_{i}g(\theta_{t},\xi_{t})^{2}.
\end{align}
%\red{In inequality $(a)$, we employ the identity that $\nabla_{i} g(u_{t})=\nabla_{i} g(\theta_{t}) + \nabla_{i} g(u_{t}) - \nabla_{i} g(\theta_{t}).$} 
Using the following identity, we decompose $\Theta_{t,1}$ into a negative term $-\sum_{i=1}^{d}\zeta_{i}(t)$, an error term $\Theta_{t,1,1}$, and a martingale difference term $M_{t,1}$.
\begin{align}\label{adam_0}
&\Theta_{t,1} \notag \\
&=-\sum_{i=1}^{d}\eta_{t,i}\nabla_{i}g(\theta_{t})\nabla_{i}g(\theta_{t},\xi_{t})=-\sum_{i=1}^{d}\eta_{{t-1},i}\nabla_{i}g(\theta_{t})\nabla_{i}g(\theta_{t},\xi_{t})+\sum_{i=1}^{d}\Delta_{t,i}\nabla_{i}g(\theta_{t})\nabla_{i}g(\theta_{t},\xi_{t})\notag\\&=-\sum_{i=1}^{d}\underbrace{\eta_{{t-1},i}(\nabla_{i}g(\theta_{t}))^{2}}_{\zeta_{i}(t)}+\underbrace{\sum_{i=1}^{d}\Delta_{t,i}\nabla_{i}g(\theta_{t})\nabla_{i}g(\theta_{t},\xi_{t})}_{\Theta_{t,1,1}}+\underbrace{\sum_{i=1}^{d}\eta_{{t-1},i}\nabla_{i}g(\theta_{t})(\nabla_{i} g(\theta_{t})-\nabla_{i}g(\theta_{t},\xi_{t}))}_{M_{t,1}},
\end{align}
where $\Delta_{t} = \eta_{t-1} - \eta_t$ and $\Delta_{t,i}$ represents the $i$-th component of $\Delta_{t}$. We further bound the error term $\Theta_{t,1,1}$
\begin{align}\label{adam_2}
\Theta_{t,1,1} &=\sum_{i=1}^{d}\Expect\big[\Delta_{t,i}\nabla_{i}g(\theta_{t})\nabla_{i}g(\theta_{t},\xi_{t})\mid\mathscr{F}_{t-1}\big] \notag \\
& \quad +\underbrace{\sum_{i=1}^{d}\big(\Delta_{t,i}\nabla_{i}g(\theta_{t})\nabla_{i}g(\theta_{t},\xi_{t})-\Expect\big[\Delta_{t,i}\nabla_{i}g(\theta_{t})\nabla_{i}g(\theta_{t},\xi_{t})\mid\mathscr{F}_{t-1}\big]}_{M_{t,2}}\big)\notag\\&\mathop{<}^{(a)}\sum_{i=1}^{d}\sqrt{\eta_{{t-1}}}\nabla_{i}g(\theta_{t})\Expect\big[\sqrt{\Delta_{t,i}}\sqrt{\nabla_{i}g(\theta_{t},\xi_{t})}\mid\mathscr{F}_{t-1}\big]+M_{t,2}\notag\\&\mathop{\le}^{(b)} \frac{1}{2}\sum_{i=1}^{d}\eta_{{t-1}}(\nabla_{i} g(\theta_{t}))^{2}+\frac{1}{2}\sum_{i=1}^{d}\Expect^{2}\big[\sqrt{\Delta_{t,i}}{\nabla_{i}g(\theta_{t},\xi_{t})}\mid\mathscr{F}_{t-1}\big]+M_{t,2}\notag\\&\mathop{\le}^{(c)} \frac{1}{2}\sum_{i=1}^{d}\zeta_{i}(t)+\frac{1}{2}\sum_{i=1}^{d}\Expect[(\nabla_{i}g(\theta_{t},\xi_{t}))^{2}\mid\mathscr{F}_{t-1}]\cdot\Expect[\Delta_{t,i}\mid\mathscr{F}_{t-1}]+M_{t,2}\notag\\&\le \frac{1}{2}\sum_{i=1}^{d}\zeta_{i}(t)+\frac{1}{2}\sum_{i=1}^{d}\Expect[(\nabla_{i}g(\theta_{t},\xi_{t}))^{2}\mid\mathscr{F}_{t-1}]\cdot\Delta_{t,i} +M_{t,2}\notag\\&\quad +\underbrace{\frac{1}{2}\Bigg(\sum_{i=1}^{d}\Big(\Expect[(\nabla_{i}g(\theta_{t},\xi_{t}))^{2}\mid\mathscr{F}_{t-1}]\cdot\Expect[\Delta_{t,i}\mid\mathscr{F}_{t-1}]-\Expect[(\nabla_{i}g(\theta_{t},\xi_{t}))^{2}\mid\mathscr{F}_{t-1}]\cdot\Delta_{t,i}\Big)\Bigg)}_{M_{t,3}} \notag\\&\mathop{\le}^{(d)} \frac{1}{2}\sum_{i=1}^{d}\zeta_{i}(t)+\frac{\sigma_{0}}{2}\underbrace{\sum_{i=1}^{d}(\nabla_{i}g(\theta_{t}))^{2}\cdot\Delta_{t,i}}_{\Theta_{t,1,1,1}}+\frac{\sigma_{1}}{2}\sum_{i=1}^{d}\Delta_{t,i}+M_{t,2}+M_{t,3}.
\end{align}
In the above derivation, step $(a)$ utilizes the property of conditional expectation that for the random variables \(X \in \mathscr{F}_{n-1}\) and \(Y \in \mathscr{F}_{n}\), \(\Expect[XY|\mathscr{F}_{n-1}] = X\Expect[Y|\mathscr{F}_{n-1}]\). Note that \(\Delta_{t,i} = \sqrt{\Delta_{t,i}}\sqrt{\Delta_{t,i}} < \sqrt{\eta_{{t-1}}}\sqrt{\Delta_{t,i}}\) (due to  \cref{property_0}, each element of $\eta_t$ is non-increasing, we have \(\Delta_{t,i} \ge 0\), thus the square root of $\Delta_{t,i}$ is well-defined). In step $(b)$, we employed the \emph{AM-GM} inequality that \(ab \le \frac{a^{2} + b^{2}}{2}\). In step $(c)$, we used the \emph{Cauchy-Schwarz} inequality for conditional expectations that \(\Expect[XY|\mathscr{F}_{n-1}] \le \sqrt{\Expect[X^{2}|\mathscr{F}_{n-1}]\Expect[Y^{2}|\mathscr{F}_{n-1}]}.\) For step $(d)$, we used the coordinate-wise affine noise variance assumption stated in \cref{coordinate}~\ref{coordinate_i_1}.
Next, we estimate the second term \(\Theta_{t,1,1,1}\) of RHS of \cref{adam_2}
\begin{align*}
\Theta_{t,1,1,1}&=\sum_{i=1}^{d}(\nabla_{i}g(\theta_{t}))^{2}\cdot\Delta_{t,i}=\sum_{i=1}^{d}(\nabla_{i}g(\theta_{t}))^{2}\cdot\eta_{{t-1},i}-\sum_{i=1}^{d}(\nabla_{i}g(\theta_{t}))^{2}\cdot\eta_{t,i}\\&\le \sum_{i=1}^{d}(\nabla_{i}g(\theta_{t}))^{2}\eta_{{t-1},i}-\sum_{i=1}^{d}(\nabla_{i}g(\theta_{t+1}))^{2}\eta_{t,i}+\sum_{i=1}^{d}\big((\nabla_{i}g(\theta_{t+1}))^{2}-(\nabla_{i}g(\theta_{t}))^{2}\big)\eta_{t,i}\\&=\sum_{i=1}^{d}\zeta_{i}(t)-\sum_{i=1}^{d}\zeta_{i}(t+1)+\sum_{i=1}^{d}\big((\nabla_{i}g(\theta_{t+1}))^{2}-(\nabla_{i}g(\theta_{t}))^{2}\big)\eta_{t,i}\\&\le \sum_{i=1}^{d}\zeta_{i}(t)-\sum_{i=1}^{d}\zeta_{i}(t+1)+\sum_{i=1}^{d}\big((\nabla_{i}g(\theta_{t+1}))^{2}-(\nabla_{i}g(\theta_{t}))^{2}\big)\eta_{t,i}\\&\mathop{\le}^{(a)} \sum_{i=1}^{d}\zeta_{i}(t)-\sum_{i=1}^{d}\zeta_{i}(t+1)+\frac{1}{2\sigma_{0}}\sum_{i=1}^{d}\zeta_{i}(t)+\frac{(2\sigma_{0}+1) L^{2}}{\sqrt{v}}\|\eta_t\circ \nabla g(\theta_{t},\xi_{t})\|^{2}.
\end{align*}
In step $(a)$, we utilized the following inequality
\begin{align*}
(\nabla_{i}g(\theta_{t+1}))^{2}-(\nabla_{i}g(\theta_{t}))^{2}&= (\nabla_{i}g(\theta_{t})+\nabla_{i}g(\theta_{t+1})-\nabla_{i}g(\theta_{t}))^{2}-(\nabla_{i}g(\theta_{t}))^{2}\\&\le 2|\nabla_{i}g(\theta_{t})||\nabla_{i}g(\theta_{t+1})-\nabla_{i}g(\theta_{t})|+(\nabla_{i}g(\theta_{t+1})-\nabla_{i}g(\theta_{t}))^{2}\\&\le \frac{1}{2\sigma_{0}}(\nabla_{i}g(\theta_{t}))^{2}+(2\sigma_{0}+1)(\nabla_{i}g(\theta_{t+1})-\nabla_{i}g(\theta_{t}))^{2}.
\end{align*}
Furthermore, we have
\begin{align*}
&\quad \sum_{i=1}^{d}\big((\nabla_{i}g(\theta_{t+1}))^{2}-(\nabla_{i}g(\theta_{t}))^{2}\big)\eta_{t,i} \notag \\
& =\sum_{i=1}^{d}\big(2\nabla_{i} g(\theta_{t})^{\top}(\nabla_i g(\theta_{t+1})-\nabla_i g(\theta_{t}))+(\nabla_{i}g(\theta_{t+1})-\nabla_{i}g(\theta_{t}))^{2}\big)\eta_{t,i}\\
& \leq \sum_{i=1}^{d}\left(\frac{1}{2\sigma_0}\nabla_{i} g(\theta_{t})^{2} + 2\sigma_0(\nabla_i g(\theta_{t+1})-\nabla_i g(\theta_{t}))^2+(\nabla_{i}g(\theta_{t+1})-\nabla_{i}g(\theta_{t}))^{2} \right)\eta_{t,i} \\
&\mathop{\le}^{\text{$\eta_{t,i}\le \frac{1}{\sqrt{v}}$}}
 \frac{1}{2\sigma_{0}}\sum_{i=1}^{d}\zeta_{i}(t)+\frac{2\sigma_{0}+1}{\sqrt{v}}\|\nabla g(\theta_{t+1})-\nabla g(\theta_{t})\|^{2}\notag\\&\le\frac{1}{2\sigma_{0}}\sum_{i=1}^{d}\zeta_{i}(t)+\frac{(2\sigma_{0}+1) L^{2}}{\sqrt{v}}\|\theta_{t+1}-\theta_{t}\|^{2}\notag\\&\le \frac{1}{2\sigma_{0}}\sum_{i=1}^{d}\zeta_{i}(t)+\frac{(2\sigma_{0}+1) L^{2}}{\sqrt{v}}\|\eta_t\circ \nabla g(\theta_{t},\xi_{t})\|^{2}\notag,
\end{align*}
where since each component of $\eta_{t}$ is monotonically non-increasing in \cref{property_0}, we have $\eta_{t,i} \leq \eta_{0,i} \leq 1/\sqrt{v}$. 
We substitute the estimate of $\Theta_{t,1,1,1}$ into \cref{adam_2} and then substitute the estimation of $\Theta_{t,1,1}$ into \cref{adam_0}, which obtains
\begin{align}\label{adam_3}
\Theta_{t,1}&=-\frac{3}{4}\sum_{i=1}^{d}\zeta_{i}(t)+\sum_{i=1}^{d}\zeta_{i}(t)-\sum_{i=1}^{d}\zeta_{i}(t+1)+\frac{(2\sigma_{0}+1) L^{2}}{\sqrt{v}}\|\eta_t\circ \nabla g(\theta_{t},\xi_{t})\|^{2}\notag\\&+\frac{\sigma_{1}}{2}\sum_{i=1}^{d}\Delta_{t,i}+\underbrace{M_{t,1}+M_{t,2}+M_{t,3}}_{M_{t}}.
\end{align}
Then we apply the estimation of $\Theta_{t,1}$ into \cref{adam_-1}
%\begin{align}\label{adam_-1}
%g(\theta_{t+1})-g(\theta_{t})& \mathop{\le}^{(a)}\underbrace{-\sum_{i=1}^{d}\eta_{t,i}\nabla_{i}g(\theta_{t})\nabla_{i}g(\theta_{t},\xi_{t})}_{\Theta_{t,1}}+{ L}\sum_{i=1}^{d}\eta_{t,i}^{2}\nabla_{i}g(\theta_{t},\xi_{t})^{2}.
%\end{align}
\begin{align}\label{adam_16}
g(\theta_{t+1})-g(\theta_{t})& \leq -\frac{3}{4}\sum_{i=1}^{d}\zeta_{i}(t)+\sum_{i=1}^{d}\zeta_{i}(t)-\sum_{i=1}^{d}\zeta_{i}(t+1)+\left(\frac{ L}{2}+\frac{(2\sigma_{0}+1) L^{2}}{\sqrt{v}}\right)\|\eta_t\circ \nabla g(\theta_{t},\xi_{t})\|^{2}\notag\\&\quad +\frac{\sigma_{1}}{2}\sum_{i=1}^{d}\Delta_{t,i}+M_{t}.
\end{align}
We define the Lyapunov function $\hat{g}(\theta_{t}) = g(\theta_{t})+\sum_{i=1}^{d}\zeta_{i}(t)+\frac{\sigma_{1}}{2}\sum_{i=1}^{d}\eta_{{t-1},i}$. Then the above inequality can be re-written as
\begin{align}
\hat{g}(\theta_{t+1}))- \hat{g}(\theta_{t})) \le-\frac{3}{4}\sum_{i=1}^{d}\zeta_{i}(t)+\left(\frac{ L}{2}+\frac{(2\sigma_{0}+1) L^{2}}{\sqrt{v}}\right)\|\eta_t\circ \nabla g(\theta_{t},\xi_{t})\|^{2}+M_{t},
\end{align}
%where
%\begin{align}\label{adam_20}
%C_{1}:=\frac{(2\sigma_{0}+1) L^{2}}{\sqrt{v}},\ C_{2}:=\frac{7}{8}\frac{\beta_{1}^{2} L^{2}}{(1-\beta_{1})^{2}},\ C_{4}:=2+\frac{ L}{\sqrt{v}}\cdot \Big(\frac{\beta_{1}}{1-\beta_{1}}\Big)^{2}
%\end{align}
as we desired. 
\end{proof}

%{\begin{lem}\label{vital_1}
%Consider the RMSProp in \cref{RMSProp} and suppose that \cref{ass_g_poi} %\ref{ass_g_poi:i}$\sim$\ref{ass_g_poi:i2}, \cref{ass_noise} \ref{ass_noise:i}, %\cref{coordinate} \ref{coordinate_i_1} hold, then for any initial point and $\forall\ %\phi>0$, we have for any $T \geq 1,$ the following inequality:
%\begin{align}\label{adam_-11.5}
%\frac{\sqrt{S_{T}}}{(T+1)^{\phi}}\le  {\sqrt{vd}}+\sum_{t=1}^{T}\Lambda_{\phi,t}.
%\end{align}
%where 
%\[\Lambda_{\phi,t}:=\frac{\|\nabla g(\theta_{t},\xi_{t})\|^{2}}{(t+1)^{\phi}\sqrt{S_{t-%1}}},\] and $S_{T}$ is defined in \cref{S_t}.
%\end{lem}}
%\begin{proof}
%    For any $\phi > \mathbb{R}$, we consider $\frac{\sqrt{S_{T}}}{(T+1)^{\phi}}$, and we obtain:
%\begin{align*}
%\frac{\sqrt{S_{T}}}{(T+1)^{\phi}}&=\frac{{S_{T}}}%{(T+1)^{\phi}\sqrt{S_{T}}}=\frac{S_{0}+\sum_{t=1}^{T}\|\nabla %g(\theta_{t},\xi_{t})\|^{2}}{(T+1)^{\phi}\sqrt{S_{T}}} =\frac{S_{0}}%{(T+1)^{\phi}\sqrt{S_{T}}} +\sum_{t=1}^{T}\frac{\|\nabla g(\theta_{t},\xi_{t})\|^{2}}%{(T+1)^{\phi}\sqrt{S_{T}}}\\&\le \frac{S_{0}}{(T+1)^{\phi}\sqrt{S_{T}}} %+\sum_{t=1}^{T}\frac{\|\nabla g(\theta_{t},\xi_{t})\|^{2}}{(T+1)^{\phi}\sqrt{S_{T}}}\le %{\sqrt{S_{0}}}+\sum_{t=1}^{T}\frac{\|\nabla g(\theta_{t},\xi_{t})\|^{2}}%{(t+1)^{\phi}\sqrt{S_{t-1}}}\\&={\sqrt{dv}}+\sum_{t=1}^{T}\frac{\|\nabla %g(\theta_{t},\xi_{t})\|^{2}}{(t+1)^{\phi}\sqrt{S_{t-1}}}.
%\end{align*}
%With this, we complete the proof.

%\end{proof}

\begin{lem}\label{vital_0}
Under \cref{ass_g_poi} \ref{ass_g_poi:i}$\sim$\ref{ass_g_poi:i2}, \cref{ass_noise} \ref{ass_noise:i}, \cref{coordinate} \ref{coordinate_i_1}, we consider RMSProp with any initial point and $T \geq 1$. There exists a random variable $\zeta$ such that the following results hold
\begin{itemize}
\item[(a)] the random variable $0 \leq \zeta < +\infty$ a.s., and its expectation $\mathbb{E}(\zeta)$ is uniformly bounded above.
\item[(b)] $\sqrt{S_{T}} \leq (T+1)^{4} \zeta$  where $S_T = [S_{T,1}, S_{T,2}, \cdots, S_{T,d}]^T$ and each element $S_{T,i}$ is defined in \cref{property_1}
\end{itemize}
\end{lem}
\begin{proof}
For any $\phi > 0$, we estimate $\frac{\sqrt{S_{T}}}{(T+1)^{\phi}}$ as follows
\begin{align}\label{Lambda}
\frac{\sqrt{S_{T}}}{(T+1)^{\phi}}&=\frac{{S_{T}}}{(T+1)^{\phi}\sqrt{S_{T}}}=\frac{S_{0}+\sum_{t=1}^{T}\|\nabla g(\theta_{t},\xi_{t})\|^{2}}{(T+1)^{\phi}\sqrt{S_{T}}} =\frac{S_{0}}{(T+1)^{\phi}\sqrt{S_{T}}} +\sum_{t=1}^{T}\frac{\|\nabla g(\theta_{t},\xi_{t})\|^{2}}{(T+1)^{\phi}\sqrt{S_{T}}} \notag  \\&\le \frac{S_{0}}{(T+1)^{\phi}\sqrt{S_{T}}} +\sum_{t=1}^{T}\frac{\|\nabla g(\theta_{t},\xi_{t})\|^{2}}{(T+1)^{\phi}\sqrt{S_{T}}}\le {\sqrt{S_{0}}}+\underbrace{\sum_{t=1}^{T}\frac{\|\nabla g(\theta_{t},\xi_{t})\|^{2}}{(t+1)^{\phi}\sqrt{S_{t-1}}}}_{\sum_{t=1}^{T}\Lambda_{\phi,t}},
\end{align}
where $S_0=vd$. We set $\phi = 4$ in \cref{Lambda} and bound the expectation of the sum $\sum_{t=1}^{T}\Lambda_{4,t}$
\begin{align}\label{dsacdqewed}
\Expect\Bigg[\sum_{t=1}^{T}\Lambda_{4,t}\Bigg]&=\sum_{t=1}^{T}\Expect[\Lambda_{4,t}]=\sum_{t=1}^{T}\Expect\Bigg[\frac{\|\nabla g(\theta_{t},\xi_{t})\|^{2}}{(t+1)^{4}\sqrt{S_{t-1}}}\Bigg]=\sum_{t=1}^{T}\Expect\Bigg[\frac{\Expect[\|\nabla g(\theta_{t},\xi_{t})\|^{2}|\mathscr{F}_{t-1}]}{(t+1)^{4}\sqrt{S_{t-1}}}\Bigg]\notag\\&\mathop{\le}_{\text{\cref{loss_bound}}}^{\text{\cref{coordinate}\ref{coordinate_i_1} }}\sum_{t=1}^{T}\Expect\Bigg[\frac{2 L\sigma_{0}g(\theta_{t})+\sigma_{1}}{(t+1)^{4}\sqrt{S_{t-1}}}\Bigg]\le 2 L\sigma_{0}\sum_{t=1}^{T}\frac{\Expect\left[g(\theta_{t})\right]}{(t+1)^{4}}+\sigma_{1}\sum_{t=1}^{T}\frac{1}{(t+1)^{4}}.
\end{align}
%where $$C_{3}:=\frac{A+2L_{f}B}{\sqrt{S_{0}}},\ C_{4}:=\frac{C}{\sqrt{S_{0}}}.$$
Based on the sufficient descent inequality  in \cref{sufficient:lem'}, we estimate 
\begin{align*}
\Expect\left[g(\theta_{t})\right] \leq \mathcal{O}\left(\sum_{k=1}^{t}\Expect\|\eta_{k}\circ \nabla g(\theta_{k},\xi_{k})\|^{2}\right)+\mathcal{O}(1) = \mathcal{O}\left(\sum_{k=1}^{t}\Expect \left\| \theta_{t+1} - \theta_t \right\|^2\right)+\mathcal{O}(1) \leq \mathcal{O}(t).
\end{align*}
Substituting the above result into \cref{dsacdqewed}, and since $\sum_{t=1}^{T}\frac{1}{(t+1)^{p}}\le\sum_{t=1}^{T}\frac{1}{(t+1)^{2}}=\frac{\pi^{2}}{6},$ for any $p \geq 2$, we have
\[\Expect\left[\sum_{t=1}^{T}\Lambda_{4,t}\right] \leq \mathcal{O}(1).\]
where the RHS term is independent of $T$. According to the \emph{Lebesgue's Monotone Convergence} theorem, we have 
\[\sum_{t=1}^{T}\Lambda_{4,t} \to \sum_{t=1}^{+\infty}\Lambda_{4,t} \,\, \text{a.s.}, \quad\,\text{ and } \, 
\Expect\left[\sum_{t=1}^{+\infty}\Lambda_{4,t}\right] = \lim_{T \to \infty} \Expect\Bigg[\sum_{t=1}^{T}\Lambda_{4,t}\Bigg] = \lim_{T \to \infty} \sum_{t=1}^{T} \Expect[\Lambda_{4,t}] =\mathcal{O}(1).
\]
Next, we combine \cref{Lambda} and define $\zeta :=  \sqrt{vd} + \sum_{t=1}^{+\infty}\Lambda_{4,t},$
then
\begin{align}%\label{power}
\sqrt{S_{T}}\le  (T+1)^{4}\zeta, \quad 
\Expect[\zeta]=\sqrt{vd}+\Expect\Bigg[\sum_{t=1}^{+\infty}\Lambda_{4,t}\Bigg] \leq \mathcal{O}(1).
\end{align}
%Then through Eq. \ref{power}, we have
%\begin{align*}
%\frac{1}{2}\ln\left(\frac{S_{T}}{\sqrt{v}}\right)&\le 4\ln(T+1)+\ln\left(\zeta\right)\\&\le4\ln(T+1)+\ln\left(\max\left\{e,\zeta\right\}\right)\\&\le 4\ln(T+1)\left(1+\frac{\ln\left(\max\left\{e,\zeta\right\}\right)}{4\ln(T+1)}\right)\\&\mathop{\le}^{\ln(T+1)\ge 1/2}4\ln(T+1)\left(1+\frac{1}{2}{\ln\left(\max\left\{e,\zeta\right\}\right)}\right).
%\end{align*}
%Next, we note that
%\[
%\ln\bigg(\max\Big\{e,\frac{\sqrt{2}\zeta}{\sqrt{\sqrt{v}}}\Big\}\bigg) \geq 1, \quad \text{and} \quad \frac{1}{2} \leq \ln(T+1),
%\]
%from which we obtain:
%\begin{align*}
%\frac{1}{2}\ln(S_{T})&\le2\ln\bigg(\max\Big\{e,\frac{\sqrt{2}\zeta}{\sqrt{\sqrt{v}}}\Big\}\bigg)\bigg(\frac{\ln(T+1)}{\ln\big(\max\big\{e,\frac{\sqrt{2}\zeta}{\sqrt{\sqrt{v}}}\big\}\big)}+\frac{1}{2}\bigg)\\&\le 4\ln\bigg(\max\Big\{e,\frac{\sqrt{2}\zeta}{\sqrt{\sqrt{v}}}\Big\}\bigg)\ln(T+1).
%\end{align*}
%that is,
%\begin{align}\label{adam_-22}
%\frac{\sqrt{S_{T}}}{\alpha_{1}T^{2+\delta}}&\le \alpha_{1}(\hat{f}(u_{1})+|f^*|)+\alpha_{1}\Big(C_{3}+\frac{\mu}{\alpha_{1}}+\frac{|\mu^{2}-d\alpha_{0}v|}{\alpha_{1}\sqrt{d\alpha_{0}v}}\Big)\sum_{t=1}^{T}\frac{1}{t^{2}}\notag\\&+\sum_{t=1}^{T}\Big(\frac{}{t^{2}}M_{t}+M_{t,4}\Big)\notag\\&+\frac{5A}{16B}\sum_{t=1}^{T} \frac{1}{t^{2}}\frac{\eta_{t}\sum_{k=1}^{t}\Pi_{\Delta,k+1}M_{k}}{\sqrt{\Sigma_{v_{t-1}}}+\mu},
%\end{align} 
\end{proof}

\begin{lem}\label{RMSProp_0}
Under \cref{ass_g_poi} \ref{ass_g_poi:i}$\sim$\ref{ass_g_poi:i2}, \cref{ass_noise} \ref{ass_noise:i}, \cref{coordinate} \ref{coordinate_i_1}, consider RMSProp. We have $\forall\ 0<\delta\le 1/2$
\[\sum_{t=1}^{+\infty}\sum_{i=1}^{d}\Expect\left[\frac{\zeta_{i}(t)}{t^{\delta}}\right] \leq \mathcal{O}(1).\]
%The constant hidden within the \(\mathcal{O}\) notation depends only on the constants in the assumptions and \(1/\delta\).
\end{lem}
\begin{proof}
First, we recall the sufficient descent inequality in \cref{sufficient:lem'}
\begin{align*}
&\hat{g}(\theta_{t+1})-\hat{g}(\theta_{t})\notag \le-\frac{3}{4}\sum_{i=1}^{d}\zeta_{i}(t)+\left(\frac{ L}{2}+\frac{(2\sigma_{0}+1) L^{2}}{\sqrt{v}}\right)\|\eta_t\circ \nabla g(\theta_{t},\xi_{t})\|^{2}+M_{t}.
\end{align*}
For any $0<\delta\le 1/2$, dividing both sides of the above inequality by \(t^{\delta}\) and noting that \(t^{\delta} < (t+1)^{\delta}\), we have
\begin{align*}
 &\frac{\hat{g}(\theta_{t+1})}{(t+1)^{\delta}}-\frac{\hat{g}(\theta_{t})}{t^{\delta}}\notag \le-\frac{3}{4}\sum_{i=1}^{d}\frac{\zeta_{i}(t)}{t^{\delta}}+\left(\frac{ L}{2}+\frac{(2\sigma_{0}+1) L^{2}}{\sqrt{v}}\right)\frac{\|\eta_t\circ \nabla g(\theta_{t},\xi_{t})\|^{2}}{t^{\delta}}+\frac{M_{t}}{t^{\delta}}.   
\end{align*}
Since $M_t$ is a martingale difference sequence with $\E[M_t] = 0$, we take the expectation on both sides of the above inequality
\begin{align*}
    \Expect\left[\frac{\hat{g}(\theta_{t+1})}{(t+1)^{\delta}}\right]-\Expect\left[\frac{\hat{g}(\theta_{t})}{t^{\delta}}\right]\le -\frac{3}{4}\sum_{i=1}^{d}\Expect\left[\frac{\zeta_{i}(t)}{t^{\delta}}\right]+\left(\frac{ L}{2}+\frac{(2\sigma_{0}+1) L^{2}}{\sqrt{v}}\right)\Expect\left[\frac{\|\eta_t\circ \nabla g(\theta_{t},\xi_{t})\|^{2}}{t^{\delta}}\right]+0.
\end{align*}
Telescoping both sides of the above inequality for \(t\) from \(1\) to \(T\) gives
\begin{align}\label{rms_0}
    \frac{3}{4}\sum_{t=1}^{T}\sum_{i=1}^{d}\Expect\left[\frac{\zeta_{i}(t)}{t^{\delta}}\right]\le \hat{g}(\theta_{1})+\left(\frac{ L}{2}+\frac{(2\sigma_{0}+1) L^{2}}{\sqrt{v}}\right)\sum_{t=1}^{T}\Expect\left[\frac{\|\eta_t\circ \nabla g(\theta_{t},\xi_{t})\|^{2}}{t^{\delta}}\right].
\end{align}
Next, we focus on estimating \(\sum_{t=1}^{T}\Expect\left[\frac{\|\eta_t\circ \nabla g(\theta_{t},\xi_{t})\|^{2}}{t^{\delta}}\right]\)
\begin{align*}
&\sum_{t=1}^{T}\Expect\left[\frac{\|\eta_t\circ \nabla g(\theta_{t},\xi_{t})\|^{2}}{t^{\delta}}\right]=\sum_{t=1}^{T}   \sum_{i=1}^{d}\frac{1}{t^{\delta}}\Expect\left[\eta_{t,i}^{2}(\nabla_{i}g(\theta_{t},\xi_{t}))^{2}\right]\mathop{\le}^{\text{\cref{property_1}}}\frac{1}{r_1}\sum_{t=1}^{T}   \sum_{i=1}^{d}\frac{1}{t^{\delta}}\Expect\left[\frac{(\nabla_{i}g(\theta_{t},\xi_{t}))^{2}}{S_{t,i}}\right]\\&\le \frac{2}{r_1}\sum_{t=1}^{T}   \sum_{i=1}^{d}\frac{1}{(t+1)^{\delta}}\Expect\left[\frac{(\nabla_{i}g(\theta_{t},\xi_{t}))^{2}}{S_{t,i}}\right]\mathop{\le}^{\text{\cref{vital_0}}}\frac{2}{r_1}\sum_{t=1}^{T}   \sum_{i=1}^{d}\Expect\left[\zeta^{\delta/4}\frac{(\nabla_{i}g(\theta_{t},\xi_{t}))^{2}}{S^{1+\frac{\delta}{8}}_{t,i}}\right]\\&\le\frac{2}{r_1}\sum_{i=1}^{d}\Expect\left[\zeta^{1/8}\int_{v}^{+\infty}\frac{\text{d}x}{x^{1+\frac{\delta}{8}}}\right]=\frac{16dv^{-\delta/8}}{\delta r_1}\Expect\left[\zeta^{\delta/4}\right]\le \frac{16dv^{-\delta/8}}{\delta r_1}\Expect^{\delta/4}\left[\zeta\right]\mathop{\le}^{\text{\cref{vital_0}}} \mathcal{O}(1)%\frac{16dv^{-\delta/8}C_{\zeta}^{\delta/4}}{\delta\alpha_{1}} .
\end{align*}
We obtain the desired result and complete the proof by substituting the above estimate into \cref{rms_0}. 
\end{proof}

\begin{lem}\label{bounded_v}
Under \cref{ass_g_poi} \ref{ass_g_poi:i}$\sim$\ref{ass_g_poi:i2}, \cref{ass_noise} \ref{ass_noise:i}, \cref{coordinate} \ref{coordinate_i_1}, consider RMSProp. We have
\[\sup_{t\ge 1}\left(\frac{\Sigma_{v_{t}}}{\ln^{2}(t+1)}\right)<+\infty\ \ \text{a.s.},\]
\end{lem}
where \(\Sigma_{v_{t}} := \sum_{i=1}^{d}v_{t,i}\).
\begin{proof}
For notational convenience, we define the auxiliary variable  \(\Sigma_{v_{t}} := \sum_{i=1}^{d}v_{t,i}\). By the recursive formula for \(v_{t}\)
\[
v_{t+1,i} = \left(1-\frac{1}{t+1}\right)v_{t,i}+\frac{1}{t+1}(\nabla_{i}g(\theta_{t},\xi_{t}))^{2} < v_{t,i}+\frac{1}{t+1}(\nabla_{i}g(\theta_{t},\xi_{t}))^{2}
\]
we achieve the recursive relation for $\Sigma_{v_{t}}$ 
\begin{align*}
    \Sigma_{v_{t+1}} <  \Sigma_{v_{t}}+\frac{1}{t+1}\|\nabla g(\theta_{t},\xi_{t})\|^{2}.
\end{align*}
Dividing both sides of the above inequality by \(\ln^{2}(t+1)\) and noting that \(\ln^{2}(t+1) > \ln^{2}t\) for any $t \geq 1$, we have
\begin{align*}
    \frac{ \Sigma_{v_{t+1}}}{\ln^{2}(t+1)} < \frac{\Sigma_{v_{t}}}{\ln^{2}t}+\frac{\|\nabla g(\theta_{t},\xi_{t})\|^{2}}{(t+1)\ln^{2}(t+1)}.
\end{align*}
Next, we consider the sum of the series $\sum_{t=1}^{+\infty}\frac{1}{(t+1)\ln^{2}(t+1)}\Expect\left[\|\nabla g(\theta_{t},\xi_{t})\|^{2}|\mathscr{F}_{t-1}\right].$
By the coordinate-wised affine noise variance condition (\cref{coordinate} \ref{coordinate_i_1}), we find
\begin{align*}
\sum_{t=1}^{+\infty}\frac{\Expect\left[\|\nabla g(\theta_{t},\xi_{t})\|^{2}|\mathscr{F}_{t-1}\right]}{(t+1)\ln^{2}(t+1)} &\le \sum_{t=1}^{+\infty}\frac{(\sigma_{0}\|\nabla g(\theta_{t})\|^{2}+\sigma_{1}d)}{(t+1)\ln^{2}(t+1)} \mathop{\le}^{\text{Lemma \ref{loss_bound}}}  \sum_{t=1}^{+\infty}\frac{(2 L\sigma_{0}g(\theta_{t})+\sigma_{1}d)}{(t+1)\ln^{2}(t+1)} \notag \\
& \le \left(2 L\sigma_{0}\sup_{t\ge 1}g(\theta_{t})+\sigma_{1}d\right)\cdot   \sum_{t=1}^{+\infty}\frac{1}{(t+1)\ln^{2}(t+1)} \mathop{<}^{\text{ \cref{bound''}}}+\infty\ \ \text{a.s.},
\end{align*}
where $\sum_{t=1}^{+\infty}\frac{1}{(t+1)\ln^{2}(t+1)} < \int_{2}^{\infty} \ln^{-2}(x)d( \ln x) < + \infty$.
By applying the \emph{Supermartingale Convergence} theorem, we deduce that the sequence \(\{\Sigma_{v_{t+1}}/\ln^{2}(t+1)\}_{t\ge 1}\) converges almost surely, which implies that
$\sup_{t\ge 1}\left(\frac{\Sigma_{v_{t}}}{\ln^{2}(t+1)}\right)<+\infty\ \ \text{a.s.}$
%From the above recursive formula, we can immediately obtain:
%\begin{align*}
 %   \sup_{t\ge 1}\frac{\Sigma_{v_{t}}}{\ln^{2}(t+1)}\le \mathcal{O}(1)+\sup_{t\ge 1\sum_{t=1}^{+\infty}\frac{1}{(t+1)\ln^{2}(t+1)}
%\end{align*}
\begin{lem}\label{RMSProp_9}
    Under \cref{ass_g_poi} \ref{ass_g_poi:i}$\sim$\ref{ass_g_poi:i2}, \cref{ass_noise} \ref{ass_noise:i}, \cref{coordinate} \ref{coordinate_i_1}, consider RMSProp. We have
\[\sum_{t=1}^{T}\sum_{i=1}^{d}\frac{(\nabla_{i}g(\theta_{t}))^{2}}{t^{\frac{1}{2}+\delta}\ln(t+1)}<+\infty\ \ \text{a.s.} \,\, \text{where} \,\,  0<\delta\le 1/2. \]
\end{lem}
\begin{proof}
According to  \cref{RMSProp_0}, for any $0<\delta\le 1/2,$ we have 
\begin{align*}
\sum_{t=1}^{T}\sum_{i=1}^{d}\Expect\left[\frac{\zeta_{i}(t)}{t^{\delta}}\right]=\mathcal{O}\left(\frac{1}{\delta}\right) .  
\end{align*}
Applying the \emph{Lebesgue's Monotone Convergence} theorem, we have
\begin{align*}
\sum_{t=1}^{T}\sum_{i=1}^{d}\frac{\zeta_{i}(t)}{t^{\delta}}<+\infty\ \ \text{a.s.}.  \end{align*}
Recalling that $\zeta_{i}(t) = (\nabla_i g(\theta_t))^2 \eta_{t-1,i} \geq (\nabla_i g(\theta_t))^2 \eta_{t,i}$ (by \cref{property_0}) and $\eta_{t,i} = \alpha_t / (\sqrt{v_{t,i}} + \epsilon)$, we have
\begin{align*}
\sum_{t=1}^{T}\sum_{i=1}^{d}\frac{\zeta_{i}(t)}{t^{\delta}} \geq  \sum_{t=1}^{T}\sum_{i=1}^{d}\frac{1}{t^{\frac{1}{2}+\delta}}\frac{(\nabla_{i}g(\theta_{t}))^{2}}{\sqrt{v_{t,i}}+\epsilon}\mathop{\ge}^{\text{\cref{bounded_v}}}\mathcal{O}\left( \sum_{t=1}^{T}\sum_{i=1}^{d}\frac{(\nabla_{i}g(\theta_{t}))^{2}}{t^{\frac{1}{2}+\delta}\ln(t+1)}\right)  ,
\end{align*}
where by \text{\cref{bounded_v}}, we have $v_{t,i} \leq \Sigma_{v_t} \leq \sup_{t} \Sigma_{v_t} \leq \mathcal{O}(\ln^2(t+1)) $.
%\[\sum_{t=1}^{T}\sum_{i=1}^{d}\frac{(\nabla_{i}g(\theta_{t}))^{2}}{t^{\frac{1}{2}+\delta}\ln(t+1)}<+\infty\ \ \text{a.s.}\]
\end{proof}
\begin{lem}\label{convergence_v}
   Under \cref{ass_g_poi} \ref{ass_g_poi:i}$\sim$\ref{ass_g_poi:i2}, \cref{ass_noise} \ref{ass_noise:i}, \cref{coordinate} \ref{coordinate_i_1}, consider RMSProp. The vector sequence \(\{v_{n}\}_{n \geq 1}\) converges almost surely.
\end{lem}
\begin{proof}
Recalling the recursive formula for \(v_t\), we have
\begin{align*}
   v_{t+1,i}\le v_{t,i}+\frac{1}{t+1}(\nabla_{i}g(\theta_{t},\xi_{t}))^{2} = v_{t,i}+\frac{\I_{[(\nabla_{i}g(\theta_{t}))^{2}<D_{0}]}}{t+1}(\nabla_{i}g(\theta_{t},\xi_{t}))^{2} +\frac{\I_{[(\nabla_{i}g(\theta_{t}))^{2}\ge D_{0}]}}{t+1}(\nabla_{i}g(\theta_{t},\xi_{t}))^{2}.
\end{align*}
Next, we examine the sum of the two series
\begin{align*}
 \sum_{t=1}^{+\infty}  \frac{\I_{[(\nabla_{i}g(\theta_{t}))^{2}<D_{0}]}}{(t+1)^{2}}\Expect\left[(\nabla_{i}g(\theta_{t},\xi_{t}))^{4} |\mathscr{F}_{t-1}\right],\ \ \text{and}\ \  \sum_{t=1}^{+\infty}  \frac{\I_{[(\nabla_{i}g(\theta_{t}))^{2}\ge D_{0}]}}{t+1}\Expect\left[(\nabla_{i}g(\theta_{t},\xi_{t}))^{2}|\mathscr{F}_{t-1}\right] .
\end{align*}
For the first series, based on \cref{coordinate}~\ref{coordinate_i_2}, it concludes
\begin{align*}
 \sum_{t=1}^{+\infty}  \frac{\I_{[(\nabla_{i}g(\theta_{t}))^{2}<D_{0}]}}{(t+1)^{2}}\Expect\left[(\nabla_{i}g(\theta_{t},\xi_{t}))^{4} |\mathscr{F}_{t-1}\right]&<D_{1}^{2}  \sum_{t=1}^{+\infty}  \frac{1}{(t+1)^{2}}<+\infty\ \ \text{a.s.}
\end{align*}
We apply the coordinate-wise affine noise variance condition when $\nabla_{i}g(\theta_{t}))^{2}\ge D_{0}$  and achieve that $\Expect\left[(\nabla_{i}g(\theta_{t},\xi_{t}))^{2} |\mathscr{F}_{t-1}\right] \leq \left(\sigma_0 \nabla_{i}g(\theta_{t}))^{2}  + \sigma_1 \right) \leq (\sigma_{0}+\frac{\sigma_{1}}{D_{0}})\nabla_{i}g(\theta_{t}))^{2}$ for any $\ i$. For the second series, 
\begin{align*}
 \sum_{t=1}^{+\infty}  \frac{\I_{[(\nabla_{i}g(\theta_{t}))^{2}\ge D_{0}]}}{t+1}\Expect\left[(\nabla_{i}g(\theta_{t},\xi_{t}))^{2} |\mathscr{F}_{t-1}\right]&<\left(\sigma_{0}+\frac{\sigma_{1}}{D_{0}}\right)  \sum_{t=1}^{+\infty}  \frac{\I_{[(\nabla_{i}g(\theta_{t}))^{2}\ge D_{0}]}(\nabla_{i}g(\theta_{t}))^{2}}{(t+1)^{2}}\\& \leq \mathcal{O}\left(\sum_{t=1}^{+\infty}  \sum_{i=1}^{d}\frac{\I_{[(\nabla_{i}g(\theta_{t}))^{2}\ge D_{0}]}(\nabla_{i}g(\theta_{t}))^{2}}{t\ln(t+1)}\right)\\&\mathop{<}^{\text{ \cref{RMSProp_9} with $\delta=1/2$}}+\infty \ \ \text{a.s.}.
\end{align*}
According to the martingale convergence theorem, we have \(\{v_{t,i}\}_{t \geq 1}\) converges almost surely. Repeating the above procedure for each component \(i\), we conclude that all coordinate components converge almost surely which implies that \(\{v_{n}\}_{n \geq 1}\) converges almost surely. 
% \begin{assumpt}\label{coordinate}
% \begin{enumerate}[label=\textnormal{(\roman*)},leftmargin=*]
% \item\label{coordinate_i_1}  (\textbf{coordinate affine noise variance condition})
%Each coordinate $\nabla g_{i}(\theta_{n},\xi_{n})$ of $\nabla g(\theta_{n},\xi_{n})$  satisfies that \(\Expect\left((\nabla g_{i}(\theta_{n},\xi_{n}))^{2} \mid \mathscr{F}_{n-1}\right)  \leq \sigma_{0}(\nabla g_{i}(\theta_{n}))^{2} + \sigma_{1}.\)
% \item\label{coordinate_i_2}  (\textbf{Only for asymptotic convergence}) For any $i\in[1,d],$ there exist $D_0, D_1>0$ such that for any $\theta_n$ that satisfies $(\nabla_{i} g(\theta_{n}))^{2}<D_{0}$, we have $(\nabla_{i} g(\theta_{n},\xi_{n}))^{2}<D_{1}\ \ \text{a.s.}$
%\end{enumerate}
%\end{assumpt}  
\end{proof}

\subsection{The Proof of \cref{bound''}}\label{stability:proof:rmsprop}
The main proof of \cref{bound''} for RMSProp is similar to the proof of AdaGrad. To maintain conciseness, we will use \(\mathcal{O}\) to simplify the relevant constant terms and will omit some straightforward calculations. We first present the following lemmas, \cref{lem:adj:ghat''} and \cref{dsad}, for RMSProp. The proofs of these lemmas are omitted, because they are straightforward and follow the same arguments as the corresponding lemmas, \cref{lem:adj:ghat} and \cref{pro_0}, for AdaGrad-Norm.
\begin{lem}\label{lem:adj:ghat''}
For the Lyapunov function $\hat{g}(\theta_n),$ there is a constant $C_0$ such that for any $\hat{g}(\theta_{n}) \geq C_0$, we have
$$\hat{g}(\theta_{n+1})-\hat{g}(\theta_{n})\leq \hat{g}(\theta_{n})/2.
$$ 
\end{lem}
\begin{property}\label{dsad}
Under~\cref{coordinate_0,coordinate}, the gradient sublevel set $J_{\eta}:=\bigcup_{i=1}^{d}\{\theta \mid (\nabla_{i} g(\theta))^{2}\le \eta\}$ with $\eta >0$ is a closed bounded set. Then, by~\cref{coordinate_0,coordinate}, there exist a constant $\hat{C}_{g} > 0$  such that the function  $\hat{g}(\theta) < \hat{C}_{g}$ for any $\theta \in J_{\eta}$. 
\end{property}

\begin{proof}(of \cref{bound''}) 
First, we define \(\Delta_{0} := \max\{C_{0}, 2\hat{g}(\theta_{1}), \hat{C}_{g}\}\). Based on the value of \(\hat{g}(\theta_{n})\) with respect to \(\Delta_{0}\), we define the following stopping time sequence $\{\tau_{n}\}_{n\ge 1}$
\begin{align}\label{stopping_time'}
&\tau_{1}:=\min\{k\ge 1:\hat{g}(\theta_{k})>\Delta_0\},\ \tau_{2}:=\min\{k\ge \tau_{1}: \hat{g}(\theta_{k})\le  \Delta_0\ \text{or}\ \hat{g}(\theta_{k})>2\Delta_0\},\notag\\&\tau_{3}:=\min\{k\ge \tau_{2}:\hat{g}(\theta_{k})\le \Delta_0\},...,
\notag\\&\tau_{3j-2}:=\min\{k> \tau_{3j-3}:\hat{g}(\theta_{k})>\Delta_0\},\ \tau_{3j-1}:=\min\{k\ge \tau_{3j-2}:\hat{g}(\theta_{k})\le  \Delta_0\ \text{or}\ \hat{g}(\theta_{k})>2\Delta_0\},\notag\\&\tau_{3j}:=\min\{k\ge  \tau_{3j-1}:\hat{g}(\theta_{k})\le  \Delta_0\}.
\end{align}
By the definition of \(\Delta_{0}\), we have \(\Delta_{0} > \hat{g}(\theta_{1})\), which asserts \(\tau_{1} > 1\). Since \(\Delta_{0} > C_{0}\), for any \(j\), we have \(\hat{g}(\theta_{\tau_{3j-2}}) < \Delta_{0} + \frac{\Delta_{0}}{2} < 2\Delta_{0}\), which asserts \(\tau_{3j-1} > \tau_{3j-2}\). For any \(T\) and \(n\), we define the truncated stopping time \(\tau_{n,T} := \tau_{n} \wedge T\). Then, based on the segments by the stopping time $\tau_{n, T}$, we estimate \(\Expect\left[\sup_{1 \leq n < T} \hat{g}(\theta_{n})\right]\) as follows.
\begin{align}\label{RMSprop_7}
   \Expect\left[\sup_{1\le n<T}\hat{g}(\theta_{n})\right]&\le \Expect\left[\sup_{j\ge 1}\left(\sup_{\tau_{3j-2,T}\le n<\tau_{3j,T}}\hat{g}(\theta_{n})\right)\right] +\Expect\left[\sup_{j\ge 1}\left(\sup_{\tau_{3j,T}\le n<\tau_{3j+1,T}}\hat{g}(\theta_{n})\right)\right]\notag\\&\le \Delta_{0}+ \Expect\left[\sup_{j\ge 1}\left(\sup_{\tau_{3j-2,T}\le n<\tau_{3j,T}}\hat{g}(\theta_{n})\right)\right]\notag\\&\le \Delta_{0}+  \Expect\left[\sup_{j\ge 1}\left(\sup_{\tau_{3j-2,T}\le n<\tau_{3j-1,T}}\hat{g}(\theta_{n})\right)\right]+ \Expect\left[\sup_{j\ge 1}\left(\sup_{\tau_{3j-1,T}\le n<\tau_{3j,T}}\hat{g}(\theta_{n})\right)\right]\notag\\&\le 3\Delta_{0}+ \Expect\left[\sup_{j\ge 1}\left(\sup_{\tau_{3j-1,T}\le n<\tau_{3j,T}}\hat{g}(\theta_{n})\right)\right].
\end{align}
Next, we proceed to estimate \(\Expect\left[\sup_{j\ge 1}\left(\sup_{\tau_{3j-1,T}\le n<\tau_{3j,T}}\hat{g}(\theta_{n})\right)\right]\).
\begin{align}\label{RMSprop_5}
 \Expect\left[\sup_{j\ge 1}\left(\sup_{\tau_{3j-1,T}\le n<\tau_{3j,T}}\hat{g}(\theta_{n})\right)\right]& \mathop{\le}^{\text{\cref{lem:adj:ghat''}}}  3\Delta_{0}+  \Expect\left[\sup_{j\ge 1}\left(\sup_{\tau_{3j-1,T}\le n<\tau_{3j,T}}\left(\hat{g}(\theta_{n})-\hat{g}(\theta_{\tau_{3j-1,T}})\right)\right)\right]\notag\\&\le 3\Delta_{0}+  \Expect\left[\sup_{j\ge 1}\left(\sum_{t=\tau_{3j-1,T}}^{\tau_{3j,T}-1}\left|\hat{g}(\theta_{t+1})-\hat{g}(\theta_{t})\right|\right)\right]\notag\\& \mathop{\le}^{(a)} \mathcal{O}(1)+\mathcal{O}\left(\sum_{j=1}^{+\infty}\Expect\left[\sum_{t=\tau_{3j-1,T}}^{\tau_{3j,T}-1}\sum_{i=1}^{d}\zeta_{i}(t)\right]\right),
\end{align}
where we follow the same procedure as \cref{inqu:diff:g:tau2} to derive the inequality $(a)$. The constant hidden within the \(\mathcal{O}\) notation is independent of \(T\). Applying the sufficient descent inequality in \cref{sufficient:lem'}, the last term of RHS of \cref{RMSprop_5} is bounded by
\begin{align}\label{RMSprop_6}
\quad & \leq \sum_{j=1}^{+\infty}\Expect\left[\hat{g}(\theta_{\tau_{3j-1,T}})-\hat{g}(\theta_{\tau_{3j,T}})\right]+\left(\frac{ L}{2}+\frac{(2\sigma_{0}+1) L^{2}}{\sqrt{v}}\right)\sum_{j=1}^{+\infty}\Expect\left[\sum_{t=\tau_{3j-1},T}^{\tau_{3j,T}-1}\|\eta_t\circ \nabla g(\theta_{t},\xi_{t})\|^{2}\right] \notag \\
& \quad +\sum_{j=1}^{+\infty}\Expect\left[\sum_{t=\tau_{3j-1},T}^{\tau_{3j,T}-1}M_{t}\right]\notag\\&=\mathcal{O}\left(\sum_{j=1}^{+\infty}\Expect\left[\I_{\tau_{3j-1,T}<\tau_{3j,T}}\right]\right)+\mathcal{O}\left(\sum_{j=1}^{+\infty}\Expect\left[\sum_{t=\tau_{3j-1},T}^{\tau_{3j,T}-1}\|\eta_t\circ \nabla g(\theta_{t},\xi_{t})\|^{2}\right]\right)+0\notag\\&\mathop{\le}^{(a)}\mathcal{O}\left(\sum_{j=1}^{+\infty}\Expect\left[\I_{\tau_{3j-1,T}<\tau_{3j,T}}\right]\right)+\mathcal{O}\left(\sum_{j=1}^{+\infty}\Expect\left[\sum_{t=\tau_{3j-1,T}}^{\tau_{3j,T}-1}\sum_{i=1}^{d}\frac{\zeta_{i}(t)}{\sqrt{t}}\right]\right)\notag\\&\mathop{\le}^{\text{ \cref{RMSProp_0}}}\mathcal{O}\left(\sum_{j=1}^{+\infty}\Expect\left[\I_{\tau_{3j-1,T}<\tau_{3j,T}}\right]\right)+\mathcal{O}\left(1\right).   
\end{align}
%Now, we focus on the estimation of \(\sum_{j=1}^{+\infty}\Expect\left[\I_{\tau_{3j-1,T}<\tau_{3j,T}}\right]\). 
Similar to the proof of ~\cref{lem:psi:i1}, the following inclusions of the events hold
\begin{align*}
\{\tau_{3j-1,T}<\tau_{3j,T}\}&\subset \{\hat{g}(\theta_{3i-1,T})>2\Delta_{0}\}\subset\left\{\frac{\Delta_{0}}{2}\le \hat{g}(\theta_{\tau_{3j-1,T}})-\hat{g}(\theta_{\tau_{3j-2,T}})\right\}.%\\&\subset\{\text{\cref{power_00} holds}\}.
\end{align*}
To estimate \(\Expect\left[\I_{\tau_{3j-1,T}<\tau_{3j,T}}\right]\), we evaluate the probability of the event \(W=\left\{\frac{\Delta_{0}}{2}\le \hat{g}(\theta_{\tau_{3j-1,T}})-\hat{g}(\theta_{\tau_{3j-2,T}})\right\}.\) Note that when the event $W$ occurs 
\begin{align*}
    \frac{\Delta_{0}}{2}&\le \hat{g}(\theta_{\tau_{3j-1,T}})-\hat{g}(\theta_{\tau_{3j-2,T}}) \mathop{\le}^{\text{ \cref{sufficient:lem'}}}\left({ L}+\frac{(2\sigma_{0}+1) L^{2}}{\sqrt{v}}\right)\sum_{t=\tau_{3j-2,T}}^{\tau_{3j-1,T}-1}\|\eta_t\circ \nabla g(\theta_{t},\xi_{t})\|^{2}+\sum_{t=\tau_{3j-2,T}}^{\tau_{3j-1,T}-1}M_{t}\\&\mathop{\le}^{\text{\emph{AM-GM} inequality}}\left(\frac{ L}{2}+\frac{(2\sigma_{0}+1) L^{2}}{\sqrt{v}}\right)\sum_{t=\tau_{3j-2,T}}^{\tau_{3j-1,T}-1}\|\eta_t\circ \nabla g(\theta_{t},\xi_{t})\|^{2}+\frac{\Delta_{0}}{4}+\frac{1}{\Delta_{0}}\left(\sum_{t=\tau_{3j-2,T}}^{\tau_{3j-1,T}-1}M_{t}\right)^{2},
\end{align*}
which implies that the following inequality holds
\begin{align}\label{RMSprop_2}
    \frac{\Delta_{0}}{4}&\le\left(\frac{ L}{2}+\frac{(2\sigma_{0}+1) L^{2}}{\sqrt{v}}\right)\sum_{t=\tau_{3j-2,T}}^{\tau_{3j-1,T}-1}\|\eta_t\circ \nabla g(\theta_{t},\xi_{t})\|^{2}+\frac{1}{\Delta_{0}}\left(\sum_{t=\tau_{3j-2,T}}^{\tau_{3j-1,T}-1}M_{t}\right)^{2}.
\end{align}
Combining the above derivations, when the event \(\{\tau_{3j-1,T}<\tau_{3j,T}\}\) occurs, the event \(\{\text{\cref{RMSprop_2} holds}\}\) also occurs, which implies that
\begin{align}\label{RMSProp_2}
&\quad \Expect\left[\I_{\tau_{3j-1,T}<\tau_{3j,T}}\right] \notag \\
& \le\Pro\left[\{\text{\cref{RMSprop_2} holds}\}\right]\notag\\&\mathop{\le}^{\text{\emph{Markov's} inequality}}\frac{4}{\Delta_{0}}\left(\frac{ L}{2}+\frac{(2\sigma_{0}+1) L^{2}}{\sqrt{v}}\right)\Expect\left[\sum_{t=\tau_{3j-2,T}}^{\tau_{3j-1,T}-1}\|\eta_t\circ \nabla g(\theta_{t},\xi_{t})\|^{2}\right]+\frac{4}{\Delta_{0}^{2}}\Expect\left[\sum_{t=\tau_{3j-2,T}}^{\tau_{3j-1,T}-1}M_{t}\right]^{2}\notag\\&\mathop{\le}^{\text{\emph{Doob's Stopped} theorem}}\frac{4}{\Delta_{0}}\left(\frac{ L}{2}+\frac{(2\sigma_{0}+1) L^{2}}{\sqrt{v}}\right)\underbrace{\Expect\left[\sum_{t=\tau_{3j-2,T}}^{\tau_{3j-1,T}-1}\|\eta_t\circ \nabla g(\theta_{t},\xi_{t})\|^{2}\right]}_{A_{j,1}}+\frac{4}{\Delta_{0}^{2}}\underbrace{\Expect\left[\sum_{t=\tau_{3j-2,T}}^{\tau_{3j-1,T}-1}M^{2}_{t}\right]}_{A_{j,2}}.
\end{align}
For \(A_{j,1}\), we further estimate it as follows.
\begin{align*}
A_{j,1}&=\Expect\left[\sum_{t=\tau_{3j-2,T}}^{\tau_{3j-1,T}-1}\|\eta_t\circ \nabla g(\theta_{t},\xi_{t})\|^{2}\right]\mathop{=}^{\text{\emph{Doob's Stopped} theorem}}\Expect\left[\sum_{t=\tau_{3j-2,T}}^{\tau_{3j-1,T}-1}\Expect\left[\|\eta_t\circ \nabla g(\theta_{t},\xi_{t})\|^{2}|\mathscr{F}_{t-1}\right]\right]\\&\le\Expect\left[\sum_{t=\tau_{3j-2,T}}^{\tau_{3j-1,T}-1}\sum_{i=1}^{d}\Expect\left[\eta^{2}_{{t},i}(\nabla_{i}g(\theta_{t},\xi_{t}))^{2}|\mathscr{F}_{t-1}\right]\right] \notag \\
& \mathop{\le}^{\eta_{t,i}\le \frac{1}{\epsilon\sqrt{t}}}\frac{1}{\epsilon}\Expect\left[\sum_{t=\tau_{3j-2,T}}^{\tau_{3j-1,T}-1}\sum_{i=1}^{d}\Expect\left[\frac{\eta_{t,i}(\nabla_{i}g(\theta_{t},\xi_{t}))^{2}}{\sqrt{t}}\bigg|\mathscr{F}_{t-1}\right]\right]\\&\mathop{\le}^{\text{\cref{property_0}}}\frac{1}{\epsilon}\Expect\left[\sum_{t=\tau_{3j-2,T}}^{\tau_{3j-1,T}-1}\sum_{i=1}^{d}\Expect\left[\frac{\eta_{{t-1},i}}{\sqrt{t}}(\nabla_{i}g(\theta_{t},\xi_{t}))^{2}\bigg|\mathscr{F}_{t-1}\right]\right] \notag \\
%& =\frac{1}{\epsilon}\Expect\left[\sum_{t=\tau_{3j-2,T}}^{\tau_{3j-1,T}-1}\sum_{i=1}^{d}\frac{\eta_{{t-1},i}}{\sqrt{t}}\Expect\left[(\nabla_{i}g(\theta_{t},\xi_{t}))^{2}|\mathscr{F}_{t-1}\right]\right]\\
&\mathop{\le}^{(a)}\frac{1}{\epsilon}\left(\sigma_{0}+\frac{\sigma_{1}}{\eta}\right)\Expect\left[\sum_{t=\tau_{3j-2,T}}^{\tau_{3j-1,T}-1}\sum_{i=1}^{d}\frac{\eta_{{t-1},i}}{\sqrt{t}}(\nabla_{i}g(\theta_{t}))^{2}\right].
\end{align*}
In \((a)\), if the stopping times \(\tau_{3j-2,T} = \tau_{3j-1,T}\), we define the sum \(\sum_{t=\tau_{3j-2,T}}^{\tau_{3j-1,T}-1} = 0\), so it holds trivially. When \(\tau_{3j-2,T} < \tau_{3j-1,T}\), we know $\hat{g}(\theta_t) \in (\Delta_0, 2\Delta_0]$ where $\Delta_0 > \hat{C}_g$ for any \(t \in [\tau_{3j-2,T}, \tau_{3j-1,T})\). By \cref{{dsad}}, we have \((\nabla_{i}g(\theta_{t}))^{2} > \eta\) for any \(t \in [\tau_{3j-2,T}, \tau_{3j-1,T})\)  and $i \in [d]$. By the coordinated affine noise variance condition, we have \[\Expect\left[(\nabla_{i}g(\theta_{t}, \xi_{t}))^{2} \mid \mathscr{F}_{t-1}\right] \leq \sigma_{0}(\nabla_{i}g(\theta_{t}))^{2} + \sigma_{1} \leq \left(\sigma_{0} + \frac{\sigma_{1}}{\eta}\right)(\nabla_{i}g(\theta_{t}))^{2}.\]
We further show that $ \sum_{j=1}^{+\infty}A_{j,1}$ is uniformly bounded. In fact,
\begin{align*}
    \sum_{j=1}^{+\infty}A_{j,1}&\le \frac{1}{\epsilon}\left(\sigma_{0}+\frac{\sigma_{1}}{\eta}\right)\Expect\left[\sum_{j=1}^{+\infty}\sum_{t=\tau_{3j-2,T}}^{\tau_{3j-1,T}-1}\sum_{i=1}^{d}\frac{\eta_{{t-1},i}}{\sqrt{t}}(\nabla_{i}g(\theta_{t}))^{2}\right] \leq \mathcal{O}\left(\sum_{t=1}^{+\infty}\sum_{i=1}^{d}\frac{\eta_{{t-1},i}}{\sqrt{t}}(\nabla_{i}g(\theta_{t}))^{2}\right)\\&\mathop{\le}^{\text{\cref{RMSProp_0} with $\delta=1/2$}}\mathcal{O}(1).
\end{align*}
Then, following the same procedure as $A_{j,1}$ to estimate \(A_{j,2}\), we obtain that
\begin{align*}
 \sum_{j=1}^{+\infty}A_{j,2}\le \mathcal{O}\left(\sum_{t=1}^{+\infty}\sum_{i=1}^{d}\frac{\eta_{{t-1},i}}{\sqrt{t}}(\nabla_{i}g(\theta_{t}))^{2}\right)\mathop{\le}^{\text{\cref{RMSProp_0} with $\delta=1/2$}}\mathcal{O}(1).  
\end{align*}
According to \cref{RMSprop_2}, combining the estimates for $A_{j,1}$ and $A_{j,2}$ gives
\begin{align*}
\sum_{j=1}^{+\infty}\Expect\left[\I_{\tau_{3j-1,T}<\tau_{3j,T}}\right]\le \mathcal{O}\left(\sum_{j=1}^{+\infty}A_{j,1}\right)+\mathcal{O}\left(\sum_{j=1}^{+\infty}A_{j,2}\right)\le \mathcal{O}(1).
\end{align*}
Substituting the above estimate into \cref{RMSprop_6}, and then into \cref{RMSprop_5} and \cref{RMSprop_7}, we obtain
\begin{align*}
\Expect\left[\sup_{1\le n<T}\hat{g}(\theta_{n})\right]  \le \mathcal{O}(1).
\end{align*}
where the constant hidden in \(\mathcal{O}\) is independent of \(T\). Taking \(T \to +\infty\) and applying the \emph{Lebesgue's Monotone Convergence} theorem, we have
$\Expect\left[\sup_{n\ge 1}\hat{g}(\theta_{n})\right] \leq  \mathcal{O}(1) $
which implies
\[\Expect\left[\sup_{n\ge 1}{g}(\theta_{n})\right] \leq  \mathcal{O}(1).\]
\end{proof}

\subsection{The Proof of Theorem \ref{convergence_1.0'}} \label{sec:proof:thm1:rmsprop}
First, we re-write the RMSProp update rule in \cref{RMSProp} to a form of a standard stochastic approximation iteration% as in \cref{SA}:
\begin{align}\label{SA:rmsprop}
x_{n+1} = x_{n} - \gamma_{n}(g(x_{n})+U_{n}),
\end{align}
where
\[x_{n}:=(\theta_{n},v_{n})^{\top},\ \ \gamma_{n}:=\alpha_{n},\] and
\begin{align*}
    g(x_{n}):=\begin{pmatrix}
\frac{1}{\sqrt{v_{n}}+\epsilon}\circ \nabla g(\theta_{n}) \\
0\end{pmatrix},\ \ U_{n}:=\begin{pmatrix}
\frac{1}{\sqrt{v_{n}}+\epsilon}\circ( \nabla g(\theta_{n},\xi_{n})-\nabla g(\theta_{n}) )\\\frac{1}{\alpha_{n}}(v_{n+1}-v_{n})
\end{pmatrix}.
\end{align*}
Next, we verify that the two conditions in Proposition \ref{SA_p} hold. In fact, based on \cref{bound''} and the coercivity (\cref{extra}~\ref{extra:i1}), we can prove the stability of the iteration sequence $x_n$, which implies that \cref{pros:a1} holds. To verify that \cref{pros:a2} holds, we examine the following term for any $ n\in\mathbb{N}_{+}$
\begin{align*}
    \sup_{m(nT)\le k\le m((n+1)T)}\left\|\sum_{t=m(nT)}^{k}\gamma_{t}U_{t}\right\|&\le  \underbrace{\sup_{m(nT)\le k\le m((n+1)T)}\left\|\sum_{t=m(nT)}^{k}\frac{\alpha_{t}}{\sqrt{v_{t}}+\epsilon}\circ \left(\nabla g(\theta_{t},\xi_{t})-\nabla g(\theta_{t})\right)\right\|}_{B_{n,1}}\\&\quad +\underbrace{\sup_{m(nT)\le t\le k}\left\|v_{k}-v_{m(nT)}\right\|}_{B_{n,2}}.
\end{align*}
First, combining \cref{convergence_v} that \(\{v_{n}\}_{n \geq 1}\) converges almost surely and the \emph{Cauchy's Convergence} principle, we conclude that \(\limsup_{n\rightarrow+\infty}B_{n,2}=\lim_{n\rightarrow+\infty}B_{n,2}=0\ \ \text{a.s.}\)
Then, we adopt a divide-and-conquer strategy and decompose \(B_{n,1}\) by \(B_{n,1,1}\) and \(B_{n,1,2}\) as follows
\begin{align*}
B_{n,1}&\le \underbrace{\sup_{m(nT)\le k\le m((n+1)T)}\left\|\sum_{t=m(nT)}^{k}\sum_{i=1}^{d}\frac{\alpha_{t}\I_{[(\nabla_{i}g(\theta_{t}))^2<D_{0}]}}{\sqrt{v_{t},i}+\epsilon}\cdot\left(\nabla_{i} g(\theta_{t},\xi_{t})-\nabla_{i} g(\theta_{t})\right)\right\|}_{B_{n,1,1}}\\&\quad +\underbrace{\sup_{m(nT)\le k\le m((n+1)T)}\left\|\sum_{t=m(nT)}^{k}\sum_{i=1}^{d}\frac{\alpha_{t}\I_{[(\nabla_{i}g(\theta_{t}))^2 \geq D_0]}}{\sqrt{v_{t},i}+\epsilon}\cdot\left(\nabla_{i} g(\theta_{t},\xi_{t})-\nabla_{i} g(\theta_{t})\right)\right\|}_{B_{n,1,2}}.
\end{align*}
We first investigate $\Expect[B_{n,1,1}^{3}]$ and achieve that by applying \emph{Burkholder's inequality}
\begin{align*}
 \Expect[B_{n,1,1}^{3}]&\le \mathcal{O}(1)\cdot\sum_{t=m(nT)}^{m((n+1)T)}\Expect\left[\left(\sum_{i=1}^{d}\frac{\alpha_{t}\I_{[(\nabla_{i}g(\theta_{t}))^2<D_{0}]}}{\sqrt{v_{t},i}+\epsilon}\cdot\left|\nabla_{i} g(\theta_{t},\xi_{t})-\nabla_{i} g(\theta_{t})\right|\right)^{3}\right] \\&\le \mathcal{O}(1)\cdot\frac{d^{2}}{\epsilon^{3}}\sum_{t=m(nT)}^{m((n+1)T)}\left(\sum_{i=1}^{d}\Expect\left[{\alpha_{t}^3\I_{[(\nabla_{i}g(\theta_{t}))^2<D_{0}]}}\cdot\left|\nabla_{i} g(\theta_{t},\xi_{t})-\nabla_{i} g(\theta_{t})\right|^{3}\right]  \right) \\& \le\mathcal{O}(1)\cdot \frac{4d^{3}(D_{0}^{3/2}+D_{1}^{3/2})}{\epsilon^{3}}\sum_{t=m(nT)}^{m((n+1)T)}\alpha_{t}^{3},
\end{align*}
where $\sqrt{v_{t,i}} + \epsilon > \epsilon$ for all $t \geq 1$ and when $(\nabla_{i}g(\theta_{t}))^2<D_{0}$ we have $(\nabla_{i}g(\theta_{t}; \xi_t))^2<D_{1}$ a.s. (\cref{coordinate}~\ref{coordinate_i_2}).
We set $\alpha_t = O(1/\sqrt{t})$ and conclude $\sum_{n=1}^{+\infty}\Expect[B_{n,1,1}^{3}]<+\infty $.
By the \emph{Lebesgue's Monotone Convergence} theorem, we have
$ \sum_{n=1}^{+\infty}B_{n,1,1}^{3}<+\infty\ \ \text{a.s.},$ which implies that
\begin{align}\label{dasdwqe}
\limsup_{n\rightarrow+\infty}B_{n,1,1}=0\ \ \text{a.s.}
\end{align}
To examine $B_{n,1,2}$, we investigate $\Expect[B_{n,1,2}^{2}].$ Applying \emph{Burkholder's inequality} and using $\eta_{t,i} = \alpha_t / \sqrt{v_{t,i}+\epsilon} \leq \eta_{t-1,i}$ and coordinate the affine noise variance condition when $(\nabla_{i}g(\theta_{t}))^2\ge D_{0}$, we have
\begin{align*}
    \Expect[B_{n,1,2}^{2}]&\le \mathcal{O}(1)\cdot\sum_{t=m(nT)}^{m((n+1)T)}\Expect\left[\left(\sum_{i=1}^{d}\frac{\alpha_{t-1}\I_{[(\nabla_{i}g(\theta_{t}))^2\ge D_{0}]}}{\sqrt{v_{t-1},i}+\epsilon}\cdot\left|\nabla_{i} g(\theta_{t},\xi_{t})-\nabla_{i} g(\theta_{t})\right|\right)^{2}\right] \\&\le \mathcal{O}(1)\cdot \frac{d}{\epsilon}\left(\sigma_{0}+\frac{\sigma_{1}}{D_{0}}\right)\sum_{t=m(nT)}^{m((n+1)T)}\Expect\left[{\frac{1}{\sqrt{t-1}}}\cdot\sum_{i=1}^{d}\frac{1}{\sqrt{v_{t-1,i}}+\epsilon}\left|\nabla_{i} g(\theta_{t})\right|^{2}\right] \\&\le\mathcal{O}\left(\sum_{t=m(nT)}^{m((n+1)T)}\sum_{i=1}^{d}\Expect\left[\frac{\zeta_{i}(t)}{\sqrt{t-1}}\right]\right) \le\mathcal{O}\left(\sum_{t=m(nT)}^{m((n+1)T)}\sum_{i=1}^{d}\Expect\left[\frac{\zeta_{i}(t)}{\sqrt{t}}\right]\right).
\end{align*}
Using \cref{RMSProp_0} with \(\delta = 1/2\), we have  $ \sum_{n=1}^{+\infty}\Expect[B_{n,1,2}^{2}]<+\infty. $
By the \emph{Lebesgue's Monotone Convergence} theorem, we conclude that:
$\sum_{n=1}^{+\infty}B_{n,1,2}^{2}<+\infty\ \ \text{a.s.},
$ which implies that
\begin{align*}
\limsup_{n\rightarrow+\infty}B_{n,1,2}=0\ \ \text{a.s.}
\end{align*}
We combine the above result with \cref{dasdwqe} and get that
$\limsup_{n\rightarrow+\infty}B_{n,1}=0\ \ \text{a.s.} $ Then, because $\limsup_{n\rightarrow+\infty}B_{n,2}=0\ \ \text{a.s.} $, we conclude that \cref{pros:a2} in \cref{SA_p} is satisfied. Moreover, by applying \cref{extra}~\ref{extra:i2},  \cref{pros:a3} in \cref{SA_p} is also satisfied. Thus, using the statement of \cref{SA_p}, we conclude the almost sure convergence of RMSProp, as we desired.  
\end{proof}
\end{document}